\documentclass{article}
\usepackage{amsmath}
\usepackage{amssymb}
\usepackage{titlesec}
\usepackage{graphicx}
\usepackage{fancyhdr}
\usepackage{latexsym}
\usepackage{mathrsfs}%
\usepackage{booktabs}
\usepackage{mathrsfs}
\usepackage{mathtools}
\usepackage{amsthm}
\usepackage{wasysym}
\usepackage{cite}
\usepackage{color}
\usepackage{amsfonts}
\usepackage{amsmath,amssymb}
\usepackage{graphics}
\usepackage{graphics}
\usepackage{cite}
\usepackage{latexsym}
\usepackage{amsmath}
\usepackage{amssymb}
\usepackage{epsfig}
\usepackage{amscd}
\usepackage{amsmath, amsthm}
\usepackage[linkcolor=blue,colorlinks=true,urlcolor=red,bookmarksopen=true]{hyperref}
\numberwithin{equation}{section}
\newcommand{\R}{\mathbb{R}}

\newtheorem{theorem}{Theorem}[section]
\newtheorem{lemma}{Lemma}[section]
\newtheorem{remark}{Remark}[section]
\newtheorem{proposition}{Proposition}[section]

\title{Large data global strong solutions to the quantum Navier-Stokes equations with degenerate viscosity}
\author{Yongteng G{\small U}$^{a}$, Xiangdi H{\small UANG}$^{a}$\thanks{Email addresses: guyongteng@amss.ac.cn (Y. T. Gu); xdhuang@amss.ac.cn (X. D. Huang)}, Muxi L{\small EI}$^{a}$   \\  
	{\normalsize a. State Key Laboratory of Mathematical Sciences, Academy of Mathematics and Systems Sciences,}\\
	{\normalsize Chinese Academy of Sciences, Beijing 100190, China;}\\
}
\date{}

\begin{document}
\maketitle
\begin{abstract}
This work is devoted to study the Cauchy problem for the quantum Navier-Stokes equations in $\mathbb R^N$, $N=2,3$, with density dependent viscosities $\mu(\rho)=\rho,\lambda(\rho)=0$, quantum capillarity $\kappa(\rho)=1/\rho$, which was derived by Dunn and Serrin (see Arch. Ration. Mech. Anal. 88(2):95–133, 1985) as a phase transition model.
For arbitrarily large initial data satisfying suitable regularity and decay assumptions, with density positive at every finite point and vanishing at infinity, we prove the existence and uniqueness of global strong solutions without symmetry assumptions. The admissible pressure range is $\gamma\ge 1$ in two dimensions and $1\le\gamma< 8/3$ in three dimensions. On every finite time interval, the density remains pointwise comparable to its initial profile, and both the effective velocity and the logarithmic density gradient remain bounded. The central estimates require no uniform positive lower bound for the density.
When $\gamma>1$, a suitable combination of the density and effective velocity equations allows the maximum principle
to yield an $L^\infty$ bound for the effective velocity
using only an upper bound for the density.
We then introduce the relative density with respect to
a fixed profile determined by the initial data.
In the isothermal case $\gamma=1$, we combine estimates for
this relative density with a scalar inequality involving
the effective velocity and the logarithmic density to
close the coupled estimates and obtain the effective
velocity bound.
Combining the relative-density bounds with weighted
gradient estimates and localized De Giorgi iteration,
we control the logarithmic density gradient and close
the higher-order estimates.
\end{abstract}
\tableofcontents
\section{Introduction}
We study the global well-posedness of the compressible
Navier--Stokes--Korteweg system in $\mathbb R^N$, with $N=2,3$.
The model describes the dynamics of a compressible fluid subject to both viscous and capillary effects and is governed by
\begin{equation}\label{eq:NSK-system}
\begin{cases}
\partial_t\rho+\nabla\cdot(\rho u)=0,\\[1mm]
\partial_t(\rho u)+\nabla\cdot(\rho u\otimes u)
+\nabla P(\rho)
=\nabla\cdot\mathbb S+\nabla\cdot\mathbb K.
\end{cases}
\end{equation}
Here, $t\ge0$ and $x\in\mathbb R^N$ are the temporal and spatial
variables, respectively. The unknowns $\rho=\rho(t,x)$ and
$u=(u_1,\ldots,u_N)^T$ denote the density and velocity, while the
pressure obeys the law
\[
P(\rho)=\rho^\gamma,\qquad \gamma\ge1.
\]
For a Newtonian fluid, the viscous stress is expressed as
\begin{equation}\label{eq:viscous-stress}
\mathbb S
=2\mu(\rho)\mathcal D(u)
+\lambda(\rho)( \operatorname{div} u)\mathbb I,
\end{equation}
where
\[
\mathcal D(u)=\frac12\bigl(\nabla u+(\nabla u)^T\bigr)
\]
is the symmetric part of the velocity gradient, and $\mathbb I$
is the $N\times N$ identity matrix. The density-dependent viscosity
coefficients $\mu$ and $\lambda$ are assumed to satisfy
\[
\mu(\rho)>0,
\qquad
2\mu(\rho)+N\lambda(\rho)\ge0.
\]
To preserve the BD entropy structure, we impose the following
compatibility condition on the viscosity coefficients:
\[
\lambda(\rho)=2\bigl(\rho\mu'(\rho)-\mu(\rho)\bigr).
\]
The capillary force is represented by the divergence of the
Korteweg stress tensor:
\begin{equation}\label{eq:Korteweg-stress}
\nabla\cdot\mathbb K
=
\nabla\left(
\rho\kappa(\rho)\Delta\rho
+\frac{\kappa(\rho)+\rho\kappa'(\rho)}{2}
|\nabla\rho|^2
\right)
-\nabla\cdot\left(
\kappa(\rho)\nabla\rho\otimes\nabla\rho
\right),
\end{equation}
where the positive function $\kappa(\rho)$ specifies the
density-dependent capillarity coefficient.

The compressible Navier-–Stokes--Korteweg model has its origins in the work of Korteweg \cite{Korteweg}, who introduced capillary forces by means of a stress tensor depending on spatial variations of the density. This idea was subsequently extended by Dunn--Serrin \cite{Dunn-Serrin} , who developed a more general continuum formulation for capillary fluids.

The analysis of the Navier--Stokes--Korteweg system has developed
along several directions. Hattori--Li \cite{Hattori-Li-1, Hattori-Li-2}
established local solvability and global existence for small
perturbations of a positive constant state. Danchin--Desjardins
\cite{Danchin-Desjardins} developed a well--posedness theory in
spaces with critical regularity. For the initial-boundary-value problem, Kotschote
\cite{Kotschote} proved local existence and uniqueness of
strong solutions by maximal regularity and a fixed-point
argument. Haspot \cite{Haspot-large-energy} obtained global
strong solutions in Besov spaces under smallness conditions
that allow certain classes of initial data with large or
infinite energy. Large initial data can also be treated
in several one-dimensional settings. For instance, Charve--Haspot \cite{Charve-Haspot} obtained global strong solutions for
a model with density-dependent viscosity and quantum capillarity,
while Germain--LeFloch \cite{Germain-LeFloch} constructed
finite-energy weak solutions for more general constitutive laws.
Global weak solutions to the one-dimensional periodic problem
were also obtained by Antonelli--Bresch--Spirito
\cite{Antonelli-Bresch-Spirito} under suitable conditions on
the viscosity and capillarity coefficients. The effective velocity formulation was also used by
J\"ungel \cite{Jungel-effective} to derive an entropy identity
and establish global existence for one-dimensional quantum
Navier--Stokes flows with strictly positive density.
More recently, Chen--Zhao \cite{Chen-Zhao} treated the case
of unequal viscosity and dispersion coefficients.
Under suitable assumptions on the coefficients, the pressure
exponent, and the initial data, they obtained global
strong/classical solutions for large initial data and
established a vanishing-dispersion limit.

A closely related line of research concerns compressible
Navier--Stokes equations without capillarity. For viscosity
coefficients satisfying the BD relation, the additional entropy
estimate provides control of density gradients that is not
available from the standard energy inequality alone.
This structure underlies the global weak solution theories
developed by Li--Xin \cite{Li-Xin} and Vasseur--Yu
\cite{Vasseur-Yu} for degenerate viscosities, allowing large
initial data and vacuum. Further extensions to more general
viscosity laws were established by Bresch--Vasseur--Yu
\cite{Bresch-Vasseur-Yu}. These results accommodate vacuum at
the level of weak solutions, but do not by themselves determine
whether initially smooth solutions remain regular.

For large classical solutions, substantial progress has been
made under spherical symmetry. Huang--Meng--Zhang
\cite{Huang-Meng-Zhang-111} established global classical
solutions for the two-dimensional viscous Saint--Venant system
and for related compressible Navier--Stokes models satisfying
the BD relation. Gu--Huang \cite{Gu-Huang} subsequently
extended the shallow-water theory to three dimensions and to
a class of non-isentropic flows with transported entropy.
For the Cauchy problem, Chen--Zhang--Zhu
\cite{Chen-Zhang-Zhu} proved global regularity for arbitrarily
large spherically symmetric data with shallow-water viscosities.
Their results allow the density to approach zero at spatial
infinity while remaining positive at every finite point.

The possibility of vacuum formation in a bounded domain is
illustrated by the recent work of Huang--Li
\cite{Huang-Li}. They considered the three--dimensional
spherically symmetric compressible Navier--Stokes equations
without capillarity on a fixed ball, subject to no-slip
boundary conditions and power-law density-dependent
viscosities. In the range
\[
    1<\alpha<11.7,
    \qquad
    2\alpha-1\leq\gamma<3\alpha-1,
\]
they proved global existence of positive classical solutions
for arbitrarily large compatible data when the initial
radial effective velocity is nonnegative at the outer
boundary. They also constructed a family of smooth,
strictly positive initial densities with negative initial
boundary effective velocity for which the corresponding
solutions break down in finite time. At the terminal time,
the density converges uniformly to a nonnegative profile
whose zero set is precisely the boundary. Thus, even when
the initial density is strictly positive, the evolution
can produce boundary vacuum and a loss of classical
regularity.

For the quantum Navier--Stokes system, recent results have
established global strong solutions for large data without
imposing spherical symmetry. For strong solutions, Haspot \cite{Haspot-quantum} studied
the isothermal Cauchy problem with quantum pressure, using
the effective velocity and a De Giorgi argument to estimate
the inverse density. Huang--Meng--Zhang
\cite[Remark~1.1]{Huang-Meng-Zhang} subsequently pointed out
a difficulty in obtaining the velocity moment estimates
uniformly in the exponent, as required in that argument. Huang--Meng--Zhang
\cite{Huang-Meng-Zhang} treated the periodic problem by
developing a modified Nash--Moser iteration that connects
the effective velocity estimate with control of the inverse
density. This estimate allows them to close the positive
lower bound for the density and continue the strong solution.
Huang--Gu--Lei \cite{Huang-Gu-Lei} addressed the Cauchy
problem with strictly positive far-field density.
Using density truncations and refined iteration arguments,
they obtained global strong solutions for arbitrarily large
initial data when $\gamma\geq1$ in two dimensions and
$1\leq\gamma<8/3$ in three dimensions. Their analysis does
not require symmetry of the initial data. These results have been extended to more general viscosity
and capillarity laws of the form
\begin{equation}\label{eq:general-coefficients}
    \mu(\rho)=\nu\rho^\alpha,\qquad
    \lambda(\rho)=2\nu(\alpha-1)\rho^\alpha,\qquad
    \kappa(\rho)=\varepsilon^2\alpha^2\rho^{2\alpha-3},
    \qquad \nu\geq\varepsilon>0.
\end{equation}
For an admissible range of $\alpha<1$, Gu--Huang--Meng--Zhou
\cite{Gu-Huang-Meng-Zhou} established global strong solutions
to the periodic problem, while Huang--Lei--Zhou
\cite{Huang-Lei-Zhou} treated the corresponding Cauchy problem.
Both results apply in two and three dimensions and allow
arbitrarily large initial data with density bounded away
from zero. Their proofs first establish upper and positive
lower bounds for the density, and then exploit the interaction
between fast diffusion, transport, and the Lam\'e operator
to close the higher-order estimates.

These developments motivate the study of large solutions
with far-field vacuum. The result of Chen--Zhang--Zhu \cite{Chen-Zhang-Zhu}
demonstrates that decay of the density to zero at infinity
need not cause finite-time breakdown, whereas the construction
of Huang--Li \cite{Huang-Li} shows that vacuum formation can terminate
a classical solution in a bounded domain. These conclusions
concern different geometries and constitutive laws, and
neither directly settles the Cauchy problem for the quantum
Navier--Stokes system with general, nonsymmetric data and
far-field vacuum. Moreover, the positive lower bound for
the density used in the work of Huang--Gu--Lei \cite{Huang-Gu-Lei} is
unavailable when the far-field density vanishes.

The question we address is therefore whether strong solutions
of the quantum Navier--Stokes system can persist globally
for arbitrarily large initial data in the presence of
far-field vacuum, or whether the degeneracy associated with
the far-field vacuum can lead to finite-time breakdown.
Our purpose is to establish global existence under the
assumptions stated below, without requiring a uniform
positive lower bound for the density on the whole space.
We consider the shallow-water viscosity coefficients
\[
\mu(\rho)=\rho,\qquad \lambda(\rho)=0,
\]
together with the capillarity law
\[
\kappa(\rho)=\frac{1}{\rho}.
\]
For positive densities, this choice gives
\begin{equation}\label{eq:quantum-capillarity}
\nabla\cdot\mathbb K
=\nabla\cdot\bigl(\rho\nabla^2\log\rho\bigr).
\end{equation}
The governing equations therefore reduce to
\begin{equation}\label{eq:quantum-NS}
\begin{cases}
\partial_t\rho+\nabla\cdot(\rho u)=0,\\[1mm]
\partial_t(\rho u)+\nabla\cdot(\rho u\otimes u)
+\nabla P(\rho)
=\nabla\cdot\bigl(2\rho\mathcal D(u)\bigr)
+\nabla\cdot\bigl(\rho\nabla^2\log\rho\bigr).
\end{cases}
\end{equation}
Our aim is to establish the global existence of strong solutions
to \eqref{eq:quantum-NS} in $\mathbb R^N$, $N=2,3$, in the presence
of far-field vacuum. We prescribe the initial conditions
\begin{equation}\label{eq:initial-conditions}
(\rho,u)(x,0)=(\rho_0,u_0)(x),
\qquad x\in\mathbb R^N,
\end{equation}
and impose
\begin{equation}\label{eq:far-field-vacuum}
\rho(x,t)\to0,\qquad u(x,t)\to0,
\qquad\text{as }|x|\to\infty.
\end{equation}
The density is allowed to vanish at spatial infinity while
remaining positive at every finite spatial point. We next state the global existence result for the Cauchy problem
\eqref{eq:quantum-NS}--\eqref{eq:far-field-vacuum}, allowing
arbitrarily large initial data subject to the assumptions
specified below.

Set
\begin{equation}\label{eq:effective-definitions}
 g=\nabla\log\rho,\qquad v=u+\nabla\log\rho,\qquad
\mathfrak p=\gamma\rho^{\gamma-1}g,\qquad
\Phi=|g|^2-v\cdot g,
\end{equation}
the system \eqref{eq:quantum-NS} can be rewritten as 
\begin{equation}\label{eq:effective-system-1}
    \begin{cases}
        \partial_t\rho-\Delta\rho
        +\operatorname{div}(\rho v)=0,\\[2mm]
        \rho\bigl(\partial_t v+u\cdot\nabla v\bigr)
        +\nabla P(\rho)
        =\operatorname{div}(\rho\nabla v).
    \end{cases}
\end{equation}
The far-field conditions are imposed on $\rho$ and $u=v-g$
(see \eqref{eq:far-field-vacuum}). 
For positive density, \eqref{eq:effective-system-1} is equivalent to
\begin{equation}\label{eq:effective-system}
\begin{cases}
\rho_t-\Delta\rho=-\operatorname{div}(\rho v),\\
v_t-\Delta v=-(v-2g)\cdot\nabla v-\mathfrak p,\\
g_t-\Delta g+\nabla\operatorname{div} v=\nabla\Phi,
\end{cases}
\qquad
(\log\rho)_t-\Delta \log\rho=\Phi-\operatorname{div} v.
\end{equation}
All spatial function spaces below are over $\R^N$.

\begin{theorem}\label{thm:global_strong_solution}
Consider the Cauchy problem
\eqref{eq:quantum-NS}--\eqref{eq:far-field-vacuum}
in dimension $N\in\{2,3\}$, with the pressure exponent restricted by
\begin{equation}
\gamma\in
\begin{cases}
(1,\infty), & N=2,\\
(1,8/3),   & N=3.
\end{cases}
\end{equation}
Suppose that the initial data $(\rho_0,u_0)$ satisfy
\begin{equation}\label{ass:initial_data_global}
\begin{aligned}
&0<\rho_0\in L^1(\mathbb{R}^N)\cap H^3(\mathbb{R}^N),
\qquad
u_0\in H^2(\mathbb{R}^N),\\
&g_0:=\nabla\log\rho_0\in L^\infty(\mathbb{R}^N),
\qquad
\nabla g_0\in H^1(\mathbb{R}^N),\\
&\mathfrak{p}_0:=\gamma\rho_0^{\gamma-1}g_0
\in L^2(\mathbb{R}^N).
\end{aligned}
\end{equation}

Then there exists a unique strong solution $(\rho,u)$
defined for all $t\geq0$. For every finite $T>0$, one can find a constant $C_T\geq1$,
depending only on $N$, $\gamma$, the initial data and $T$,
such that
\begin{equation}\label{rhopointwiseest}
C_T^{-1}\rho_0(x)\leq\rho(x,t)\leq C_T\rho_0(x),
\qquad
(x,t)\in\mathbb{R}^N\times[0,T],
\end{equation}
and
\begin{equation}\label{rhovgestimate}
\sup_{0\leq t\leq T}
\left(
\|\rho(t)\|_{L^\infty}
+\|v(t)\|_{L^\infty}
+\|g(t)\|_{L^\infty}
\right)
\leq C_T.
\end{equation}
Moreover, the solution has the regularity
\begin{equation}
\begin{aligned}
\rho
&\in C([0,T];H^3)\cap L^2(0,T;H^4),\\
\rho_t
&\in L^\infty(0,T;H^1)\cap L^2(0,T;H^2),\\
u
&\in C([0,T];H^2)\cap L^2(0,T;H^3),\\
v-v_0,\ g-g_0
&\in C([0,T];H^2),\\
\nabla v,\nabla g
&\in L^\infty(0,T;H^1)\cap L^2(0,T;H^2),\\
u_t,v_t,g_t
&\in L^\infty(0,T;L^2)\cap L^2(0,T;H^1).
\end{aligned}
\end{equation}
Here all spatial function spaces are taken over $\mathbb{R}^N$, with the definition understood componentwise for vector fields. In addition, $\rho$ and $u$ tend to zero at infinity,
uniformly in $t\in[0,T]$.
\end{theorem}
For $\gamma=1$, the pressure force per unit mass is $g$.
Its non-integrable initial part must be included in the
reference velocity.

\begin{theorem}[Global strong solutions when $\gamma=1$]
\label{thm:global-strong-isothermal}
Let $N=2,3$ and $\gamma=1$. Assume
\begin{equation}\label{eq:isothermal-data}
0<\rho_0\in H^3,\qquad u_0\in H^2,\qquad
g_0:=\nabla\log\rho_0\in L^\infty,\qquad
\nabla g_0\in H^1,
\end{equation}
and
\begin{equation}\label{eq:g0-farfield}
g_0(x)\to 0\quad\text{as }|x|\to \infty.
\end{equation}
Then \eqref{eq:quantum-NS}--\eqref{eq:far-field-vacuum} has a
unique global positive strong solution in the class stated below.
With $v_0=u_0+g_0$, the pointwise estimates \eqref{rhopointwiseest}
hold for every finite $T>0$, with $C_T$ depending only on $N,T$
and the norms in \eqref{eq:isothermal-data}. The regularity is
\begin{equation}\label{eq:isothermal-regularity}
\begin{aligned}
\rho&\in C([0,T];H^3)\cap L^2(0,T;H^4),\\
\rho_t&\in C([0,T];H^1)\cap L^2(0,T;H^2),\\
u+t g_0&\in C([0,T];H^2),\\
v-v_0+t g_0,\quad g-g_0&\in C([0,T];H^2),\\
\nabla u,\quad\nabla v,\quad\nabla g
 &\in C([0,T];H^1)\cap L^2(0,T;H^2),\\
u_t+g_0,\quad v_t+g_0,\quad g_t
 &\in C([0,T];L^2)\cap L^2(0,T;H^1).
\end{aligned}
\end{equation}
The limits $\rho(x,t)\to0$ and $u(x,t),v(x,t),g(x,t)\to0$
hold uniformly for $t\in[0,T]$ as $|x|\to\infty$.
\end{theorem}

\begin{remark}\label{rem:affine-space}
Neither theorem assumes $v_0,g_0\in L^2$.
For example, $\rho_0(x)=(1+|x|^2)^{-2}$ and $u_0=0$ satisfy
\eqref{eq:isothermal-data}--\eqref{eq:g0-farfield}, whereas
$g_0=-4x/(1+|x|^2)\notin L^2$ in both dimensions.
Thus \eqref{eq:isothermal-regularity} implies $u(t)\notin L^2$
for $t>0$ in this example. The shift $t g_0$ cannot be removed.
Likewise, we do not assert $u+t g_0\in L^2(0,T;H^3)$:
the assumptions give only $\nabla g_0\in H^1$.
\end{remark}
\begin{remark}\label{rem:isothermal-velocity-space}
The assumption $u_0\in H^2$ is consistent with the regularity
$u+t g_0\in C([0,T];H^2)$, although $u(t)$ need not belong
to $H^2$ for positive times. Indeed, the isothermal pressure
force per unit mass is $-g=-g_0-(g-g_0)$, and the initial
component $g_0$ need not be square integrable. The corresponding
correction in the velocity is $-t g_0$, so that
\[
u(t)=-t g_0+\widetilde u(t),
\qquad
\widetilde u\in C([0,T];H^2),
\qquad
\widetilde u(0)=u_0.
\]
In particular, if $g_0\notin L^2$, then $u(t)\notin L^2$,
and hence $u(t)\notin H^2$, for every $t>0$.

The initial velocity is nevertheless attained strongly.
Since $g_0\in L^\infty$, $\nabla g_0\in H^1$, and
$H^2(\mathbb R^N)\hookrightarrow L^\infty(\mathbb R^N)$
for $N=2,3$, we have
\[
\lim_{t\downarrow0}
\left(
\|u(t)-u_0\|_{L^\infty}
+\|\nabla u(t)-\nabla u_0\|_{H^1}
\right)=0.
\]
Thus the initial condition is preserved, while the global
$H^2$ regularity is expressed through the corrected velocity
$u+t g_0$.
\end{remark}
The local construction is given in Appendix~\ref{Appendix7}.
The effective velocity estimates for
$\gamma>1$ and $\gamma=1$ are proved separately in
Sections~\ref{effvelocsec} and \ref{sec:iso-vsec}.
Section~\ref{sec:log-density-gradient} gives an estimate for
$\nabla\log\rho$. Higher-order estimates and continuation are
then established in Sections~\ref{sec:higher-order} and \ref{sec:continuation}.

Far-field vacuum creates substantial analytical difficulties.
As the density approaches zero, the viscosity coefficient
degenerates and the capillarity coefficient becomes singular.
In particular, the density-weighted kinetic energy does not
directly control the unweighted velocity norm. Moreover,
estimates involving the logarithmic density gradient
$\nabla\log\rho=\nabla\rho/\rho$ must account for the vanishing
denominator. A central part of the analysis is therefore to
derive estimates that remain uniform without a positive lower
bound for the density on the whole space.

\section{Preliminaries and Main Strategies}
Before presenting the main new estimates of this paper, we collect
several auxiliary lemmas. We begin with a local existence result
for the above Cauchy problem.
\subsection{Local Well-posedness}

\begin{lemma}[Local strong solutions when $\gamma>1$]\label{lem: loc sol for gamma>1}
   Consider the Cauchy problem
\eqref{eq:quantum-NS}--\eqref{eq:far-field-vacuum}
in dimension $N\in\{2,3\}$, and let $\gamma>1$. Suppose that the initial data $(\rho_0,u_0)$ satisfy
\begin{equation}\label{ass: initial data  of local sol}
    \begin{aligned} &0<\rho_0\in H^3(\mathbb R^N), \qquad u_0\in H^2(\mathbb R^N),\\ &g_0:=\nabla\log\rho_0\in L^\infty(\mathbb R^N), \qquad \nabla g_0\in H^1(\mathbb R^N),\\ &\mathfrak p_0:=\gamma\rho_0^{\gamma-1}g_0\in L^2(\mathbb R^N). \end{aligned}
\end{equation}
Then there exists $0<T\le 1$, depending only on $N,\gamma$ and\begin{equation}
    A_0:=1+\|\rho_0\|_{H^3}+\|u_0\|_{H^2} +\|g_0\|_{L^\infty}+\|\nabla g_0\|_{H^1} +\|\mathfrak p_0\|_{L^2},
\end{equation}
such that the problem has a unique local strong solution $(\rho,u)$ on $[0,T ]$. Set
\begin{equation}
    v:=u+\nabla\log\rho,\qquad g:=\nabla\log\rho,\qquad \mathfrak p:=\gamma\rho^{\gamma-1}g,\qquad v_0:=u_0+g_0.
\end{equation}
For some $B<\infty$, depending only on the initial data,
\begin{equation}
    e^{-B}\rho_0(x)\leq \rho(x,t)\leq e^B\rho_0(x), \qquad (x,t)\in\mathbb R^N\times[0,T ].
\end{equation}
Moreover, the solution has the regularity 
\begin{equation}\label{regular of local sol}
    \begin{aligned} \rho &\in C([0,T ];H^3)\cap L^2(0,T ;H^4),\\ \rho_t &\in C([0,T ];H^1)\cap L^2(0,T ;H^2),\\ u &\in C([0,T ];H^2)\cap L^2(0,T ;H^3),\\   v-v_0,\ g-g_0 &\in C([0,T ];H^2),\\ v,g &\in C([0,T ];\mathcal B),\\ \nabla v,\nabla g &\in L^2(0,T ;H^2),\\ u_t, v_t,g_t &\in C([0,T ];L^2)\cap L^2(0,T ;H^1),\\ \mathfrak p &\in C([0,T ];H^1),\\ \log\rho-\log\rho_0 &\in C([0,T ];BUC). \end{aligned}
\end{equation}
Here \begin{equation}\label{space: BUC}
    BUC(\mathbb{R}^N):=\{f\in C(\mathbb{R}^N):\|f \|_{L^{\infty}}<\infty,\ f\text{ uniformly continuous}\}, 
\end{equation}
i.e. the space of bounded uniformly continuous functions, with norm $\|f\|_{L^\infty}:=\sup_{x\in\mathbb{R}^N}|f(x)|$, and
\begin{equation}\label{space: BUC and nabla H1}
    \mathcal{B}:=\{f\in BUC(\mathbb{R}^N):\nabla f\in H^1(\mathbb{R}^N)\},\qquad \|f\|_{\mathcal{B}}:=\|f\|_{L^\infty}+\|\nabla f\|_{H^1}.
\end{equation}
In addition, $\rho$ and $u$ tend to zero at infinity, uniformly in $t\in[0,T ]$.
\end{lemma}

\begin{proof}
    See Subsection \ref{subsec: proof of loc sol}.
\end{proof}
\begin{lemma}[Local strong solutions]\label{lem:local-isothermal}
Let $N=2,3$, and suppose that
\[
0<\rho_{0}\in H^3,\qquad
g_{0}:=\nabla\log\rho_{0}\in\mathcal B,\qquad
v_{0}\in\mathcal B.
\]
There exists $0<T\le1$, depending only on $N$ and
\[
\tilde{A}_{0}:=1+\|\rho_{0}\|_{H^3}
       +\|v_{0}\|_{\mathcal B}+\|g_{0}\|_{\mathcal B},
\]
such that the effective system \eqref{eq:effective-system} with
$\gamma=1$ and initial data $(\rho_{0},v_{0},g_{0})$ has a unique positive
strong solution on $[0,T]$ satisfying
\begin{equation}\label{eq:iso-local-class}
\begin{aligned}
\rho&\in C([0,T];H^3)\cap L^2(0,T;H^4),\\
\rho_t&\in C([0,T];H^1)\cap L^2(0,T;H^2),\\
v,g&\in C([0,T];\mathcal B),\qquad
\nabla v,\nabla g\in L^2(0,T;H^2),\\
v-v_{0}+t g_{0},\quad g-g_{0}&\in C([0,T];H^2),\\
v_t+g_{0},\quad g_t&\in C([0,T];L^2)\cap L^2(0,T;H^1),\\
\log\rho-\log\rho_{0}&\in C([0,T];BUC).
\end{aligned}
\end{equation}
Moreover, $g=\nabla\log\rho$ and
$e^{-C}\rho_{0}\le\rho\le e^C\rho_{0}$, where $C$ depends only on
$N,\tilde{A}_{0}$. With $u=v-g$ and $u_{0}=v_{0}-g_{0}$, one has
\[
u-u_{0}+t g_{0}\in C([0,T];H^2),\qquad
u_t+g_{0}\in C([0,T];L^2)\cap L^2(0,T;H^1),
\]
and $(\rho,u)$ solves \eqref{eq:quantum-NS} with $\gamma=1$.
If $u_{0}\in H^2$, the first of these assertions becomes
$u+t g_{0}\in C([0,T];H^2)$.
The lifespan requires neither $u_{0}\in H^2$ nor $g_{0}\in L^2$.
If $v_{0},g_{0}\in C_0(\mathbb R^N)$, all the corresponding far-field
limits hold uniformly on $[0,T]$.
\end{lemma}
\begin{proof}
See Subsection~\ref{subsec:local-isothermal}.
\end{proof}
\subsection{The Upper bound of $\rho$}
The following lemmas adapt the estimates in Propositions 3.1--3.5
of Huang--Gu--Lei \cite{Huang-Gu-Lei} to the setting of vacuum
at infinity.

Throughout these lemmas, let $N\in\{2,3\}$ and assume that
\begin{equation}\label{eq:admissible-gamma}
\begin{cases}
1\leq\gamma<\infty, & N=2,\\
1\leq\gamma<8/3,    & N=3.
\end{cases}
\end{equation}
Let $(\rho,v)$ be a sufficiently regular solution of
\eqref{eq:effective-system-1}, with $P(\rho)=\rho^\gamma$,
on its interval of existence $[0,T_*)$, subject to
\eqref{eq:initial-conditions} and \eqref{eq:far-field-vacuum}.
We assume that $\rho(x,t)>0$ at every finite spatial point
and set
\[
u=v-\nabla\log\rho.
\]
No uniform positive lower bound for $\rho$ is imposed.

Suppose that $\rho_0\in L^1(\mathbb R^N)\cap L^\infty(\mathbb R^N)$
and that
\begin{equation}\label{eq:initial-control-quantities}
\begin{aligned}
M_0&:=\int_{\mathbb R^N}\rho_0\,dx<\infty,\\
\mathcal E_0
&:=\int_{\mathbb R^N}
\left(
\rho_0|v_0|^2
+|\nabla\sqrt{\rho_0}|^2
+\rho_0^\gamma
\right)\,dx<\infty.
\end{aligned}
\end{equation}
Unless otherwise specified, all spatial integrals and norms
are taken over $\mathbb R^N$.

\begin{lemma}[Mass conservation and basic energy estimate]
\label{lem:basic-energy}
For every $0<T<T_*$, the mass is conserved:
\begin{equation}\label{eq:mass-conservation}
\int_{\mathbb R^N}\rho(x,t)\,dx=M_0,
\qquad 0\leq t\leq T.
\end{equation}
Moreover,
\begin{equation}\label{eq:basic-energy-estimate}
\sup_{0\leq t\leq T}
\int_{\mathbb R^N}
\left(
\rho|u|^2+|\nabla\sqrt\rho|^2+\rho^\gamma
\right)\,dx
+
\int_0^T\int_{\mathbb R^N}
\rho|\mathcal{D}(u)|^2\,dx\,dt
\leq C_T,
\end{equation}
where
\[
\mathcal{D}(u)=\frac12\bigl(\nabla u+(\nabla u)^{\mathsf T}\bigr),
\]
and $C_T$ depends only on $N$, $\gamma$, $M_0$,
$\mathcal E_0$, and $T$.
When $\gamma>1$, the constant in
\eqref{eq:basic-energy-estimate} can be chosen independently of $T$.
\end{lemma}

\begin{lemma}[Weighted estimates for the effective velocity]
\label{lem:effective-velocity-energy}
For every $0<T<T_*$, there exists
$C_T=C(N,\gamma,M_0,\mathcal E_0,T)>0$ such that
\begin{equation}\label{eq:effective-velocity-energy}
\sup_{0\leq t\leq T}
\int_{\mathbb R^N}\rho|v|^2\,dx
+
\int_0^T\int_{\mathbb R^N}
\left(
\rho|\nabla v|^2
+|\nabla\rho^{\gamma/2}|^2
\right)\,dx\,dt
\leq C_T.
\end{equation}
In addition,
\begin{equation}\label{eq:additional-dissipation}
\int_0^T\int_{\mathbb R^N}
\left(
\rho|\nabla u|^2
+\rho|\nabla^2\log\rho|^2
+|\nabla^2\sqrt\rho|^2
\right)\,dx\,dt
\leq C_T.
\end{equation}
For $\gamma>1$, these constants can be taken independently of $T$.
\end{lemma}
\begin{remark}
    The estimates in Lemma \ref{lem:basic-energy} and Lemma \ref{lem:effective-velocity-energy} remain valid for \(\gamma=1\), with constants depending on \(T\). We follow the calculations in Propositions 3.1–3.2 of Huang--Gu--Lei \cite{Huang-Gu-Lei}, but estimate the pressure work directly using the finite mass.

    Testing $\eqref{eq:quantum-NS}_2$ with \(u\), using the continuity equation and integrating by parts, we have
    \begin{equation}
         \frac{d}{dt} \left(\frac12\int\rho|u|^2+2\int|\nabla\sqrt\rho|^2\right)dx +2\int\rho|\mathcal D(u)|^2 dx=\int\rho\operatorname{div}udx\le \int\rho|\mathcal D(u)|^2dx+\frac N4M_0.
    \end{equation}
    Integrating in time, using $u_0=v_0-\nabla\log\rho_0$ and $\rho_0|g_0|^2=4|\nabla\sqrt{\rho_0}|^2,$ we therefore obtain
    \eqref{eq:basic-energy-estimate}.

    For Lemma \ref{lem:effective-velocity-energy}, testing $\eqref{eq:effective-system-1}_2$ with $v$ gives
    \begin{equation}
        \frac12\frac{d}{dt}\int\rho|v|^2 dx+\int\rho|\nabla v|^2 dx=\int\rho\operatorname{div}vdx.
    \end{equation}
    By using $v=u+\nabla\log\rho$, we arrive at $\int\rho\operatorname{div}v=\int\rho\operatorname{div}u -4\int|\nabla\sqrt\rho|^2.$ Consequently,
   $$\frac12\frac{d}{dt}\int\rho|v|^2 dx+\int\rho|\nabla v|^2dx +4\int|\nabla\sqrt\rho|^2 dx\le \int\rho|\mathcal D(u)|^2dx+\frac N4M_0.$$ 
   Integrating this inequality and applying \eqref{eq:basic-energy-estimate} yields
   \eqref{eq:effective-velocity-energy}.

\end{remark}
\begin{lemma}[Integral bounds for the density]
\label{lem:density-integrability}
For every $0<T<T_*$, one has
\begin{equation}\label{eq:density-integrability}
\sup_{0\leq t\leq T}
\left(
\|\sqrt{\rho(t)}\|_{H^1}
+\|\rho(t)\|_{L^2}
\right)
\leq C_T,
\end{equation}
where $C_T$ depends only on $N$, $\gamma$, $M_0$,
$\mathcal E_0$, and $T$.
If $\gamma>1$, the constant can be chosen independently of $T$.
\end{lemma}

\begin{lemma}[Higher integrability of the effective velocity]
\label{lem:effective-velocity-higher-integrability}
Choose $q$ such that
\begin{equation}\label{eq:choice-of-q}
\begin{cases}
q>0, & N=2,\\[1mm]
1<q<4,\qquad
\displaystyle\gamma\leq\frac{2q+6}{q+2}, & N=3.
\end{cases}
\end{equation}
Such a choice is possible for every $\gamma$ satisfying
\eqref{eq:admissible-gamma}.
Assume additionally that
\begin{equation}\label{eq:initial-velocity-moment}
\mathcal V_{0,q}
:=\int_{\mathbb R^N}\rho_0|v_0|^{q+2}\,dx<\infty.
\end{equation}
Then, for every $0<T<T_*$,
\begin{equation}\label{eq:effective-velocity-higher-integrability}
\begin{aligned}
&\sup_{0\leq t\leq T}
\int_{\mathbb R^N}\rho|v|^{q+2}\,dx\\
&\quad+
\int_0^T\int_{\mathbb R^N}
\rho|v|^q
\left(
|\nabla v|^2+q|\nabla|v||^2
\right)\,dx\,dt
\leq C_T,
\end{aligned}
\end{equation}
where $C_T$ depends only on $N$, $\gamma$, $q$, $M_0$,
$\mathcal E_0$, $\mathcal V_{0,q}$, and $T$.
In particular,
\begin{equation}\label{eq:weighted-velocity-Lq}
\sup_{0\leq t\leq T}
\|\rho^{1/(q+2)}v(t)\|_{L^{q+2}}
\leq C_T.
\end{equation}
\end{lemma}

\begin{lemma}[Upper bound for the density]
\label{lem:density-upper-bound}
Under the assumptions of
Lemma~\ref{lem:effective-velocity-higher-integrability},
for every $0<T<T_*$ there exists a constant $C_T>0$ such that
\begin{equation}\label{eq:density-upper-bound}
\sup_{0\leq t\leq T}\|\rho(t)\|_{L^\infty}\leq C_T.
\end{equation}
The constant depends only on $N$, $\gamma$, $q$, $M_0$,
$\mathcal E_0$, $\mathcal V_{0,q}$,
$\|\rho_0\|_{L^\infty}$, and $T$.
In particular, this estimate is independent of any positive
lower bound for the density.
\end{lemma}

\begin{remark}
All the preceding estimates are independent of a uniform
positive lower bound for $\rho$ and are compatible with
$\rho(x,t)\to0$ as $|x|\to\infty$.
For $\gamma=1$, the quantity $\rho^\gamma$ appearing above
is simply $\rho$ and is controlled by mass conservation.
The corresponding bounds are estimates on finite time intervals;
no time-independent bound based on a nonnegative isothermal
internal energy is asserted.
\end{remark}
\subsection{A comparison principle}
\begin{lemma}[Comparison principle]
    \label{lem:max}
    Let $Z$ be bounded and continuous on
    $[0,S]\times\mathbb R^N$. Assume that, for every $R>0$,
    \[
    Z\in H^1(0,S;L^2(B_R))
    \cap L^2(0,S;H^2(B_R)).
    \]
    Let $b:(0,S)\times\mathbb R^N\to\mathbb R^N$
    be measurable and satisfy
    \[
    \|b\|_{L^\infty((0,S)\times\mathbb R^N)}
    \le M<\infty.
    \]
    Suppose that
    \[
    \mathcal LZ
    :=
    \bigl(\partial_t-\Delta+b(t,x)\cdot\nabla\bigr)Z
    \le0
    \]
    almost everywhere in $(0,S)\times\mathbb R^N$.
    Then
    \begin{equation}\label{eq:maxprinciple}
        Z(t,x)
        \le
        \sup_{y\in\mathbb R^N}Z(0,y),
        \qquad
        0\le t\le S,\quad x\in\mathbb R^N.
    \end{equation}
    In particular, the right-hand side of
    \eqref{eq:maxprinciple} is independent of $M$.
\end{lemma}

\begin{proof}
    \textbf{Step 1: Construction of a barrier growing at spatial infinity.}
    Set
    \begin{equation}\label{eq:barrier}
        \lambda=2N+M^2+1,
        \qquad
        \Psi(t,x)=e^{\lambda t}(1+|x|^2).
    \end{equation}
    Direct computation gives
    \[
    \partial_t\Psi
    =
    \lambda e^{\lambda t}(1+|x|^2),
    \qquad
    \nabla\Psi=2e^{\lambda t}x,
    \qquad
    \Delta\Psi=2Ne^{\lambda t}.
    \]
    Since
    \[
    2b\cdot x
    \ge-2M|x|
    \ge-M^2-|x|^2
    \]
    almost everywhere, we have
    \begin{align}\label{eq:Lbarrier}
        \mathcal L\Psi
        &=
        e^{\lambda t}
        \bigl[
            \lambda(1+|x|^2)-2N+2b\cdot x
        \bigr]
        \notag\\
        &\ge
        e^{\lambda t}
        \bigl[
            (\lambda-2N-M^2)
            +(\lambda-1)|x|^2
        \bigr]
        \notag\\
        &=
        e^{\lambda t}
        \bigl[
            1+(2N+M^2)|x|^2
        \bigr]
        >0.
    \end{align}

    \textbf{Step 2: Localization of the positive part.}
    Let
    \[
    m_0=\sup_{y\in\mathbb R^N}Z(0,y).
    \]
    For any $\delta>0$, define
    \[
    Q_\delta(t,x)
    =
    Z(t,x)-m_0-\delta\Psi(t,x).
    \]
    At the initial time,
    \[
    Q_\delta(0,x)
    \le-\delta(1+|x|^2)<0.
    \]
    Since $Z$ is bounded on the entire time interval,
    \[
    Q_\delta(t,x)\to-\infty
    \qquad\text{as }|x|\to\infty,
    \]
    uniformly for $t\in[0,S]$.
    We may therefore choose $R_\delta>0$ sufficiently large
    that
    \[
    Q_\delta(t,x)<0
    \qquad
    \text{whenever }|x|\ge R_\delta/2,\quad 0\le t\le S.
    \]
    Set
    \[
    w_\delta=(Q_\delta)_+.
    \]
    Then $w_\delta$ vanishes outside $B_{R_\delta/2}$
    and has zero boundary trace on $B_{R_\delta}$.
    Moreover, \eqref{eq:Lbarrier} implies
    \begin{equation}\label{eq:LQ}
        \mathcal LQ_\delta
        =
        \mathcal LZ-\delta\mathcal L\Psi
        \le-\delta\mathcal L\Psi<0
    \end{equation}
    almost everywhere.

    \textbf{Step 3: Vanishing of the positive part.}
    The assumed regularity and the smoothness of $\Psi$ give
    \[
    Q_\delta
    \in H^1(0,S;L^2(B_{R_\delta}))
    \cap L^2(0,S;H^2(B_{R_\delta})).
    \]
    Consequently,
    \[
    w_\delta
    \in C([0,S];L^2(B_{R_\delta}))
    \cap L^2(0,S;H_0^1(B_{R_\delta})),
    \]
    and the Sobolev chain rule yields
    \[
    \int_{B_{R_\delta}}
        \partial_tQ_\delta\,w_\delta\,dx
    =
    \frac12\frac{d}{dt}
        \|w_\delta\|_{L^2(B_{R_\delta})}^2
    \]
    for almost every $t\in(0,S)$.
    We also have
    \[
    \nabla w_\delta
    =
    \mathbf 1_{\{Q_\delta>0\}}\nabla Q_\delta
    \qquad\text{almost everywhere}.
    \]

    Multiplying \eqref{eq:LQ} by $w_\delta$ and integrating
    over $B_{R_\delta}$, we obtain
    \[
    \frac12\frac{d}{dt}
        \|w_\delta\|_{L^2(B_{R_\delta})}^2
    +
        \|\nabla w_\delta\|_{L^2(B_{R_\delta})}^2
    \le
    -\int_{B_{R_\delta}}
        (b\cdot\nabla w_\delta)w_\delta\,dx.
    \]
    The diffusion term has been integrated by parts using
    the zero boundary trace of $w_\delta$.
    We estimate the drift term directly:
    \[
    \begin{aligned}
    \left|
        \int_{B_{R_\delta}}
        (b\cdot\nabla w_\delta)w_\delta\,dx
    \right|
    &\le
    M\|\nabla w_\delta\|_{L^2(B_{R_\delta})}
      \|w_\delta\|_{L^2(B_{R_\delta})}\\
    &\le
    \frac12\|\nabla w_\delta\|_{L^2(B_{R_\delta})}^2
    +
    \frac{M^2}{2}\|w_\delta\|_{L^2(B_{R_\delta})}^2.
    \end{aligned}
    \]
    In particular, no derivative of $b$ is required.
    It follows that
    \begin{equation}\label{eq:positive-part-energy}
        \frac{d}{dt}
            \|w_\delta\|_{L^2(B_{R_\delta})}^2
        +
            \|\nabla w_\delta\|_{L^2(B_{R_\delta})}^2
        \le
        M^2\|w_\delta\|_{L^2(B_{R_\delta})}^2.
    \end{equation}
    Since $w_\delta(0)=0$, Gronwall's inequality gives
    \[
    \|w_\delta(t)\|_{L^2(B_{R_\delta})}^2=0,
    \qquad 0\le t\le S.
    \]
    Thus $Q_\delta(t,\cdot)\le0$ almost everywhere
    for every $t\in[0,S]$.
    By continuity, this inequality holds at every spatial point.
    Together with the negativity outside $B_{R_\delta}$,
    we conclude that
    \[
    Q_\delta(t,x)\le0,
    \qquad
    (t,x)\in[0,S]\times\mathbb R^N.
    \]

    \textbf{Step 4: Removal of the barrier.}
    For each fixed $(t,x)$, the preceding inequality gives
    \[
    Z(t,x)
    \le
    m_0+\delta e^{\lambda t}(1+|x|^2).
    \]
    Keeping $t$ and $x$ fixed,
    we let $\delta\downarrow0$ to obtain
    \eqref{eq:maxprinciple}.

    Although $\lambda$ depends on $M$, it appears only
    in the barrier term, which vanishes in this limit.
    The energy argument gives a zero positive part
    for every finite $M$.
    Hence $M$ does not enter the final comparison bound.
\end{proof}
\subsection{Main Strategies}
For initial densities bounded away from zero, or approaching a
strictly positive far-field state, Huang et al.\
\cite{Huang-Meng-Zhang,Huang-Gu-Lei} developed refined Moser
iteration and De Giorgi arguments to establish a control
relation between
\[
\|v\|_{L^\infty}
\quad\text{and}\quad
\sqrt{\log\left(e^{25/9}+\|\rho^{-1}\|_{L^\infty}\right)}.
\]
This relation enabled them to obtain both an upper bound for
the effective velocity and a positive lower bound for the density. 

These arguments, however, rely crucially on the density remaining
uniformly positive. They therefore do not apply directly to the
far-field vacuum setting, where $\rho^{-1}$ is unbounded.
This motivates the search for an estimate of the effective
velocity that does not require a uniform positive lower bound
for the density. In this paper, we show that the intrinsic
structure of the equations provides such an estimate: once
an upper bound for the density is available, an $L^\infty$
bound for the effective velocity follows without invoking
a lower bound for the density. 

The continuation of the local strong solution to a global one is based on the following blow-up criterion:
\begin{equation}
T_{*}<\infty
\quad\Longrightarrow\quad
\limsup_{t\uparrow T_{*}}
\left(
\|\rho(t)\|_{L^\infty}
+\|v(t)\|_{L^\infty}
+\|\nabla\log\rho(t)\|_{L^\infty}
\right)=\infty.
\end{equation}
Following the argument of Huang--Gu--Lei \cite{Huang-Gu-Lei}, one can obtain a uniform upper bound for the density $\rho$. Hence, in view of the above blow-up criterion, it remains to control
\[
\|v(t)\|_{L^\infty}
\qquad\text{and}\qquad
\|\nabla\log\rho(t)\|_{L^\infty}.
\]
The rest of this section is therefore devoted to deriving the required estimates for these two quantities.
\subsubsection{Outline of the effective velocity estimate}
\paragraph{The effective velocity when $\gamma>1$.}
We rewrite the original system in the form
\begin{equation}
\begin{cases}
\displaystyle
\partial_t c-\Delta c
+\bigl(v-2\nabla\log\rho\bigr)\cdot\nabla c
=
-\frac{\gamma-1}{2}c\,\operatorname{div}v
-\frac{\gamma^2-1}{4}c|\nabla\log\rho|^2,
\\[3mm]
\displaystyle
\partial_t v-\Delta v
+\bigl(v-2\nabla\log\rho\bigr)\cdot\nabla v
=
-c^2\nabla\log\rho,
\end{cases}
\qquad
c=\sqrt{\gamma}\,\rho^{\frac{\gamma-1}{2}}.
\end{equation}
The idea is to combine the equation for $|v|$ with that for
$2c/(\gamma-1)$ and apply the maximum principle to their sum.
To avoid the singularity of $|v|$ at $v=0$, we work with
\[
s=\sqrt{|v|^2+\varepsilon^2},\qquad
e=\frac{v}{s},\qquad
W_\varepsilon=s+\frac{2c}{\gamma-1}.
\]
Multiplying the equation for $c$ by $2/(\gamma-1)$ and combining
it with the equation for $s$, we absorb the pressure
contribution into an additional drift $ce$.
Completing the square in the remaining velocity-gradient terms
then gives, with $g=\nabla\log\rho$,
\begin{equation}\label{eq:velocity-strategy}
\begin{aligned}
&\underbrace{\partial_tW_\varepsilon-\Delta W_\varepsilon}
_{\text{parabolic part}} +\underbrace{(v-2g+ce)\cdot\nabla W_\varepsilon}
_{\text{bounded drift on the local interval}}\\
&\qquad=
\underbrace{-\mathcal Q_\varepsilon
-\frac{\gamma+1}{2}c|g|^2}_{\text{nonpositive terms}}
+\underbrace{\frac{c^2s}{4}
\bigl(N-|e|^2\bigr)}
_{\displaystyle\le k(t)W_\varepsilon},
\end{aligned}
\end{equation}
where $\mathcal Q_\varepsilon\ge0$ is the term obtained by
completing the square, and
\[
k(t)=\frac{N\gamma}{4}
\|\rho(t)\|_{L^\infty}^{\gamma-1}.
\]

Thus the only positive contribution is a zeroth-order term
controlled by the density upper bound. It is removed by setting
\[
Z_\varepsilon(t,x)
=\exp\left(-\int_0^t k(\tau)\,\mathrm d\tau\right)
W_\varepsilon(t,x).
\]
Indeed, \eqref{eq:velocity-strategy} implies
\[
\underbrace{
\left[\partial_t-\Delta+
(v-2g+ce)\cdot\nabla\right]Z_\varepsilon
\le0
}_{\text{the form required by the maximum principle}}.
\]
The comparison principle bounds $Z_\varepsilon$ by its initial
supremum. Letting $\varepsilon\downarrow0$ therefore yields
\[
\|v(t)\|_{L^\infty}
\le
\left\||v_0|+\frac{2c_0}{\gamma-1}\right\|_{L^\infty}
\exp\left(
\frac{N\gamma}{4}
\int_0^t\|\rho(\tau)\|_{L^\infty}^{\gamma-1}
\,\mathrm d\tau
\right).
\]
The local boundedness of the drift is needed only to justify
the comparison argument and its size does not enter the resulting
estimate. In particular, this velocity bound requires no
uniform positive lower bound for the density.
\paragraph{The effective velocity when $\gamma=1$.}
The preceding sound-speed quantity is singular as $\gamma\downarrow1$.
Writing $h=\log\rho$ and $g=\nabla h$, we instead use
\[
\mathcal P=\partial_t-\Delta+(v-2g)\cdot\nabla,
\qquad
\mathcal Pv=-g,\qquad
\mathcal Ph=-\operatorname{div}v-|g|^2.
\]
Let $H=\eta*h_0$ be a fixed smoothing of $h_0=\log\rho_0$,
where $\eta\ge0$ is a smooth mollifier supported in $B_1$
with unit integral. Set $\ell=\nabla H$ and
$L=\|g_0 \|_{L^{\infty}}$. Then
\[
|H-h_0|\le L,\qquad
\|\ell \|_{L^{\infty}}\le L,\qquad
\|\Delta H \|_{L^{\infty}}\le C_\eta L.
\]
The idea is to combine the equations for $|v|^2/2$ and $h$
in the quantity
\[
W=\frac12|v|^2+h-H.
\]
Completing the squares gives
\begin{equation}\label{eq:isothermal-velocity-strategy}
\begin{aligned}
\mathcal PW
={}&\underbrace{
-\left|\nabla v+\frac12 I_N\right|_F^2
-\left|g+\frac12v-\ell\right|^2}
_{\text{nonpositive terms}}\\
&+\frac14|v|^2-2v\cdot\ell
  +|\ell|^2+\Delta H+\frac N4
\le\frac12|v|^2+A,
\end{aligned}
\end{equation}
where $A=5L^2+\|\Delta H \|_{L^{\infty}}+N/4\le C(1+L^2)$.
Throughout this argument, $C$ depends only on $N$ and $\eta$.
Put
\[
X(t)=\|v(t) \|_{L^{\infty}}^2,\qquad
F(t)=\sup_{0\le s\le t}X(s),\qquad
V_0=\|v_0 \|_{L^{\infty}}.
\]
The function
\[
Z=W-At-\frac12\int_0^tX(s)\,\mathrm ds
\]
satisfies $\mathcal PZ\le0$. Since $W$ and the drift are
bounded on each local existence interval, the comparison
principle yields
\[
\sup_xW(t,x)
\le\frac12V_0^2+L+At+\frac12\int_0^tX(s)\,\mathrm ds.
\]

To recover an upper bound for $v$, we also need to control
the negative part of $h-H$. For this purpose, introduce
\[
a=\frac{\rho}{e^H},\qquad q=a^{-1},\qquad
a_t-\Delta a+\operatorname{div}(ba)=da,
\]
where
\[
b=v-2\ell,\qquad
d=|\ell|^2-\operatorname{div}\ell-v\cdot\ell.
\]
Negative-power tests for $a$, followed by the Moser iteration, give, with $K_0=1+L^2$ and
$\alpha_N=(N+2)/4$,
\begin{equation}\label{eq:isothermal-strategy-reciprocal}
\begin{aligned}
\|q\|_{L^\infty((0,t)\times\mathbb R^N)}
\le{}&Ce^L(1+t)^{1/2}
      \bigl[C(K_0+F(t))\bigr]^{\alpha_N}\\
&\times\exp\left(
CK_0t+C\int_0^tX(s)\,\mathrm ds
\right).
\end{aligned}
\end{equation}
Here the time supremum $F(t)$ occurs only in a polynomial
factor; the exponential retains the time integral of $X$.
Thus, using $|v|^2/2=W+\log q$ and taking logarithms in
\eqref{eq:isothermal-strategy-reciprocal}, we obtain
\[
F(t)\le C(1+V_0^2+L^2)(1+t)
       +2\alpha_N\log(1+F(t))
       +C\int_0^tF(s)\,\mathrm ds.
\]
The inequality
\[
2\alpha_N\log(1+F)\le\frac12F+C_N
\]
absorbs the logarithmic term. Gronwall's inequality then gives
\[
\sup_{0\le s\le t}\|v(s) \|_{L^{\infty}}^2
\le C(1+V_0^2+L^2)e^{Ct},
\qquad N=2,3.
\]
The size of the local bound for $g$ does not enter this
estimate. No finite-mass assumption or uniform positive
lower bound for the density is needed.
\subsubsection{Outline of the estimate for the logarithmic density gradient}
\label{sec:log-gradient-strategy}

Set $g=\nabla\log\rho$. Differentiating the logarithmic form of
the continuity equation gives
\begin{equation}\label{strat:g-equation}
g_t-\Delta g
=\nabla\bigl(|g|^2-v\cdot g-\operatorname{div}v\bigr).
\end{equation}
A direct De Giorgi argument for this equation does not close
with the bounds available at this stage. Indeed, truncating
a component of $g$ produces terms involving $|g|^4$ and
$|\operatorname{div}v|^2$ on its superlevel sets.
Controlling these terms requires additional integrability of
the gradients, which is not supplied by the $L^\infty$ bounds
for $\rho$ and $v$. In particular, the quadratic term $|g|^2$
cannot yet be treated as a controlled source.

We therefore still use a fixed weight determined by the initial
density:
\[
H=\eta*\log\rho_0,\qquad
\omega=e^H,\qquad
\ell=\nabla H,
\]
where $\eta$ is a fixed nonnegative spatial mollifier.
Writing $L=\|\nabla\log\rho_0\|_{L^\infty}$, we have
\[
e^{-L}\rho_0\le\omega\le e^L\rho_0,
\qquad
\|\ell \|_{L^{\infty}}+\|\operatorname{div}\ell \|_{L^{\infty}}
\le C_\eta L.
\]
For $a=\rho/\omega$, the continuity equation becomes
\begin{equation}\label{strat:a-equation}
a_t-\Delta a
+\operatorname{div}\bigl((v-2\ell)a\bigr)
=
\bigl(|\ell|^2-\operatorname{div}\ell-v\cdot\ell\bigr)a.
\end{equation}
The coefficients in this equation are controlled by the
velocity bound and the initial data. Testing with positive
and negative powers of $a$, followed by iteration, yields (When $\gamma=1$, the $L^\infty$ bound for $v$ and the
relative-density bounds are obtained together by closing
the coupled estimates.)
\begin{equation}\label{strat:a-bound}
C_\rho^{-1}\le a(t,x)\le C_\rho,
\qquad 0\le t\le S,
\end{equation}
with $C_\rho$ independent of the local endpoint $S$.

To estimate the density gradient, we use the conservative
variables
\[
m=\rho v,\qquad
Z_k=\frac{\partial_k\rho}{\omega},\qquad
M_{ki}=\frac{\partial_km_i}{\omega}.
\]
Their advantage is seen from
\[
\rho_t-\Delta\rho=-\operatorname{div}m,
\qquad
m_t-\Delta m
=-\operatorname{div}(m\otimes v)-\nabla(\rho^\gamma).
\]
After differentiation and division by $\omega$, the equations
for $\mathcal U=(Z,M)$ can be written as
\begin{equation}\label{strat:normalized-system}
\mathcal U_t-\Delta\mathcal U
=\operatorname{div}\mathbb F+\mathcal H,
\qquad
|\mathbb F|+|\mathcal H|\le C|\mathcal U|.
\end{equation}
Here $C$ depends on the bounds for $\rho$, $v$, and the fixed
weight, but not on the desired bound for $g$.
Thus the differentiated sources have linear growth in the
normalized gradients.

A localized sixth-power energy estimate for
\eqref{strat:normalized-system}, followed by Gronwall's
inequality, gives
\begin{equation}\label{strat:normalized-L6}
\sup_{0\le t\le S}\sup_{y\in\mathbb R^N}
\left(
\|Z(t)\|_{L^6(B_2(y))}
+\|M(t)\|_{L^6(B_2(y))}
\right)\le M_6.
\end{equation}
The initial quantity is controlled by $g_0\in L^\infty$, $v_0=u_0+g_0\in L^\infty$, $\nabla v_0=\nabla u_0+\nabla g_0\in H^1\hookrightarrow L^6$ for $N=2,3,$ and the comparison between
$\omega$ and $\rho_0$.

We then return to the scalar equation for each $Z_k$:
\begin{equation}\label{strat:Z-equation}
\partial_tZ_k-\Delta Z_k
=\operatorname{div}F_k+H_k,
\end{equation}
where
\[
(F_k)_j=2\ell_jZ_k-M_{kj},
\qquad
H_k=(|\ell|^2-\operatorname{div}\ell)Z_k
-\sum_j\ell_jM_{kj}.
\]
By \eqref{strat:normalized-L6}, these sources belong to
$L^6$ on bounded space-time cylinders, uniformly with respect
to their spatial centers. Since $6>N+2$ for $N=2,3$, a
localized De Giorgi iteration yields
\[
\sup_{0\le t\le S}
\left\|\frac{\nabla\rho(t)}{\omega}\right\|_{L^\infty}
\le M_\nabla.
\]
Finally, using \eqref{strat:a-bound}, we recover
\begin{equation}\label{strat:log-gradient-bound}
\sup_{0\le t\le S}\|\nabla\log\rho(t)\|_{L^\infty}
=
\sup_{0\le t\le S}
\left\|
\frac{\nabla\rho(t)/\omega}{\rho(t)/\omega}
\right\|_{L^\infty}
\le C_\rho M_\nabla.
\end{equation}
The lower bound used here is for $\rho/\omega$.
Since $\omega$ decays with $\rho_0$, this argument remains
compatible with far-field vacuum.
\section{Relative density estimates}\label{sec:relative-density}

In this section, we derive two-sided bounds for the density
relative to its initial profile. We introduce a smooth weight
comparable to $\rho_0$ and obtain these bounds through positive-
and negative-power estimates and Moser iteration. The argument
applies to all $\gamma\ge1$. We keep the dependence on the
effective velocity explicit, since it is needed to close the
coupled estimates in the isothermal case.
\subsection{Regularization of the initial density}\label{sec:fixed-weight}

Let \(N=2,3\), and suppose that
\[
\rho_0(x)>0,
\qquad
L:=\|\nabla\log\rho_0\|_{L^\infty(\mathbb R^N)}<\infty.
\]
Write $h_0=\log\rho_0$. Fix \(T<\infty\) and
\(0<S<\min\{T,T_*\}\). On this fixed local interval, let $V$ satisfy
\[
\|v\|_{L^\infty((0,S)\times\mathbb R^N)}\le V.
\]
At this stage, $V$ is a local bound and may depend on $S$. We first derive relative-density estimates that retain the time dependence of the coefficients. These estimates apply to both $\gamma>1$ and $\gamma=1$. In Section \ref{sec:iso-vsec}, the estimate for the reciprocal relative density will be used to close the effective-velocity bound in the isothermal case. Once the effective velocity has been bounded independently of $S$, the resulting two-sided density bounds are also independent of $S$.

We begin with a regularization procedure. Choose a fixed nonnegative
function
\[
\eta\in C_c^\infty(B_1),
\qquad
\int_{\mathbb R^N}\eta(x)\,dx=1,
\]
and define
\begin{equation}\label{eq:en-weight}
H=\eta*\log\rho_0,
\qquad
\omega=e^H,
\qquad
\ell=\nabla H.
\end{equation}
By the mean value theorem, we have
\[
|\log\rho_0(x)-\log\rho_0(y)|\le L|x-y|,
\]
the function \(\log\rho_0\) has at most linear growth, and the convolution
in \eqref{eq:en-weight} is well defined.

The normalization and nonnegativity of \(\eta\) give
\begin{align*}
|H(x)-\log\rho_0(x)|
&=
\left|
\int\eta(y)\bigl(h_0(x-y)-h_0(x)\bigr)\,dy
\right|\\
&\le L\int |y|\eta(y)\,dy
\le L.
\end{align*}
Moreover,
\[
\ell=\eta*\nabla\log\rho_0,
\qquad
\partial_i\ell_j=(\partial_i\eta)*\partial_jh_0.
\]
Consequently,
\begin{equation}\label{eq:en-weight-bounds}
|H-\log\rho_0|\le L,
\qquad
|\ell|\le L,
\qquad
|\operatorname{div}\ell|\le C_\eta L,
\qquad
C_\eta:=\|\nabla\eta\|_{L^1}.
\end{equation}
In particular, no bound on \(\nabla^2h_0\) is required.
Exponentiating the first inequality of \eqref{eq:en-weight-bounds}
yields
\begin{equation}\label{eq:en-weight-comparison}
e^{-L}\rho_0\le\omega\le e^L\rho_0.
\end{equation}
We also have
\begin{equation}\label{eq:en-weight-laplacian}
\frac{\nabla\omega}{\omega}=\ell,
\qquad
\frac{\Delta\omega}{\omega}
=\operatorname{div}\ell+|\ell|^2.
\end{equation}

Define the relative density and its reciprocal by
\begin{equation}\label{eq:en-relative-density}
a=\frac{\rho}{\omega},
\qquad
z=\frac{\omega}{\rho}=\frac1a.
\end{equation}
Since \(\omega\) is independent of time, substituting
\(\rho=\omega a\) into
\[
\rho_t-\Delta\rho+\operatorname{div}(\rho v)=0,
\]
gives
\begin{align*}
a_t-\Delta a
={}&(2\ell-v)\cdot\nabla a\\
&+
\bigl(
\operatorname{div}\ell+|\ell|^2
-\operatorname{div}v-v\cdot\ell
\bigr)a.
\end{align*}
Set
\begin{equation}\label{eq:en-bd}
b=v-2\ell,
\qquad
d=|\ell|^2-\operatorname{div}\ell-v\cdot\ell.
\end{equation}
To retain the time dependence of the coefficients in the relative-density estimates, set
\begin{equation}\label{iso-v:lambda}
\lambda(s)
=1+\|b(s) \|_{L^{\infty}}^2+\|d(s) \|_{L^{\infty}},
\qquad
\Lambda(t)=\sup_{0\le s\le t}\lambda(s),
\end{equation}
where $b,d$ are defined in \eqref{eq:en-bd}.

Then
\begin{equation}\label{eq:en-a-equation}
a_t-\Delta a+\operatorname{div}(ba)=da.
\end{equation}
The coefficients satisfy
\[
|b|\le V+2L,
\qquad
|d|\le L^2+C_\eta L+VL.
\]
We therefore fix
\begin{equation}\label{eq:en-Lambda}
\Lambda_*
=
1+(V+2L)^2+L^2+C_\eta L+VL,
\end{equation}
so that
\begin{equation}\label{eq:en-Lambda1111}
\lambda(t)\le\Lambda(S)\le\Lambda_*, \qquad 0\le t\le S.
\end{equation}
The quantity $\Lambda_*$ depends only on $V$, $L$, and the fixed mollifier $\eta$. Dependence of constants on $\eta$ is suppressed.

The initial values satisfy
\begin{equation}\label{iso-v:initial-comparison}
e^{-L}\le a(0,x)\le e^L,
\qquad
e^{-L}\le z(0,x)\le e^L.
\end{equation}
For later use, write
\begin{equation}\label{eq:en-initial-norms}
M_a:=\|a(0)\|_{L^\infty}\le e^L,
\qquad
M_z:=\|z(0)\|_{L^\infty}\le e^L.
\end{equation}

\subsection{Upper bounds for the relative density
and its reciprocal}
\label{sec:separate-density-bounds}
On each fixed local interval, $\log\rho-\log\rho_0$ is bounded;
together with \eqref{eq:en-weight-comparison}, this ensures that
$a$ and $z$ are bounded before the estimates below are applied.
Spatial cutoffs and standard convergence theorems justify the weighted
tests. The resulting constants do not use these temporary bounds.

We first estimate \(a=\rho/\omega\), which gives an upper bound
for \(\rho/\omega\). We then estimate \(z=\omega/\rho\), which gives
a lower bound for \(\rho/\omega\).

\subsubsection{An upper bound for \(a=\rho/\omega\)}
\label{subsec:a-upper}

\paragraph{Step 1: the positive-power energy estimate.}
Let \(p\ge2\) and let
\(0\le\zeta\in C_c^\infty(\mathbb R^N)\).
Multiplying \eqref{eq:en-a-equation} by
\(\zeta^2a^{p-1}\) and integrating by parts, we obtain
\begin{align}
\frac1p\frac{d}{dt}\int\zeta^2a^pdx
+(p-1)\int\zeta^2a^{p-2}|\nabla a|^2dx
={}&-2\int\zeta a^{p-1}\nabla a\cdot\nabla\zeta dx
\notag\\
&+(p-1)\int\zeta^2a^{p-1}b\cdot\nabla a dx
\notag\\
&+2\int\zeta a^p b\cdot\nabla\zeta dx
+\int\zeta^2da^pdx.
\label{eq:en-a-energy-identity}
\end{align}
In particular, the divergence term is integrated by parts
without estimating \(\operatorname{div}b\).

Young's inequality gives
\begin{align*}
2\left|\int\zeta a^{p-1}\nabla a\cdot\nabla\zeta dx\right|
&\le\frac{p-1}{4}\int\zeta^2a^{p-2}|\nabla a|^2dx
+\frac{C}{p-1}\int|\nabla\zeta|^2a^pdx,\\
(p-1)\left|\int\zeta^2a^{p-1}b\cdot\nabla a dx\right|
&\le\frac{p-1}{4}\int\zeta^2a^{p-2}|\nabla a|^2dx
+C(p-1)\|b\|_{L^\infty}^2\int\zeta^2a^pdx,\\
2\left|\int\zeta a^p b\cdot\nabla\zeta dx\right|
&\le\|b\|_{L^\infty}^2\int\zeta^2a^pdx
+C\int|\nabla\zeta|^2a^pdx,\\
\left|\int\zeta^2da^p dx\right|
&\le\|d\|_{L^\infty}\int\zeta^2a^pdx.
\end{align*}
Since
\[
|\nabla a^{p/2}|^2
=\frac{p^2}{4}a^{p-2}|\nabla a|^2,
\]
we absorb the gradient terms, multiply by \(p\) and use the definition of $\lambda(s)$ to obtain
\begin{equation}\label{eq:en-a-energy}
\frac{d}{dt}\int\zeta^2a^p
+c\int\zeta^2|\nabla a^{p/2}|^2
\le
Cp^2\int
\bigl(\lambda(t)\zeta^2+|\nabla\zeta|^2\bigr)a^p.
\end{equation}
The constants \(c,C>0\) are independent of \(p\) and \(S\).

\paragraph{Step 2: Boundedness of the initial iterate.}
The function \(a\) need not belong to \(L^2(\mathbb R^N)\).
For each \(y\in\mathbb R^N\), introduce
\[
\phi_y(x)=e^{-\sqrt{1+|x-y|^2}}.
\]
Then
\[
\int\phi_y\,dx\le C_N,
\qquad
|\nabla\phi_y|\le\phi_y,
\qquad
\phi_y\ge e^{-\sqrt5}\quad\text{on }B_2(y).
\]
Taking \(p=2\) and \(\zeta=\phi_y^{1/2}\) in
\eqref{eq:en-a-energy}, first with an additional compactly
supported cutoff and then letting its radius tend to infinity,
we find
\[
\frac{d}{dt}\int\phi_y a(t)^2dx
\le C_N\lambda(t)\int\phi_y a(t)^2dx.
\]
Gronwall's inequality and \(a(0)\le M_a\) yield
\begin{equation}\label{eq:en-a-L2}
\sup_{y\in\mathbb R^N}\sup_{0\le t\le S}
\int\phi_y a(t)^2\,dx
\le C_NM_a^2\exp\left(C_N\int^{S}_{0}\lambda(t)dt \right).
\end{equation}

\paragraph{Step 3: the Moser iteration.}
Fix \(y\in\mathbb{R}^{N}\), and write
\[
Q_r^{S}=(0,S)\times B_r(y).
\]
For \(1\le r'<r\le2\), choose a cutoff \(\zeta\) supported in
\(B_r(y)\), equal to one on \(B_{r'}(y)\), and satisfying
\[
0\le\zeta\le1,
\qquad
|\nabla\zeta|\le\frac{C_N}{r-r'}.
\]
Integrating \eqref{eq:en-a-energy} in time, using
\eqref{eq:en-initial-norms} and estimating
\(\nabla(\zeta a^{p/2})\), we obtain
\begin{align}
&\sup_{0\le t\le S}\int(\zeta a^{p/2})^2
+\int_0^S\int|\nabla(\zeta a^{p/2})|^2
\notag\\
&\qquad\le
\frac{C_N\Lambda(S)p^2}{(r-r')^2}
\left(
\iint_{Q_r^{S}}a^p+|B_2|M_a^p
\right).
\label{eq:en-a-local-energy}
\end{align}
The initial term is retained, so the estimate includes the
initial time.

Let
\[
\chi=\frac{N+2}{N}.
\]
The spatial Sobolev inequality, followed by integration in time,
gives
\begin{equation}\label{eq:en-parabolic-sobolev}
\|f\|_{L^{2\chi}(Q_r^{S})}^2
\le C_N
\left(\sup_{0\le t\le S}\|f(t)\|_{L^{2}}^2\right)^{2/(N+2)}
\left(\int_0^S\|\nabla f(t)\|_{L^{2}}^2\,dt\right)^{N/(N+2)}
\end{equation}
for \(f\) with zero spatial boundary trace.
Applying this to \(f=\zeta a^{p/2}\) gives
\begin{equation}\label{eq:en-a-step}
\|a\|_{L^{p\chi}(Q_{r'}^{S})}
\le
\left(\frac{C_N\Lambda(S)p^2}{(r-r')^2}\right)^{1/p}
\left(
\iint_{Q_r^{S}}a^p+|B_2|M_a^p
\right)^{1/p}.
\end{equation}

Set
\[
p_n=2\chi^n,
\qquad
r_n=1+2^{-n},
\qquad
D_n^{(a)}(S)
=
\left(
\iint_{Q_{r_n}^{S}}a^{p_n}+M_a^{p_n}
\right)^{1/p_n}.
\]
Let $r=r_n$, $r'=r_{n+1}$ and we have
\(r_n-r_{n+1}=2^{-(n+1)}\). Absorbing
\(\max\{1,|B_2|\}\) into the dimensional constant and write
\[
\|a\|_{L^{p_{n+1}}(Q_{r_{n+1}}^{S})}
\le (C_N\Lambda(S)p_n^2\,4^{n+1})^{1/p_n}D_n^{(a)}(S).
\]
To include the initial term in \(D_{n+1}^{(a)}\), observe that
\(M_a\le D_n^{(a)}\). Hence
\begin{align*}
\bigl(D_{n+1}^{(a)} (S)\bigr)^{p_{n+1}}
&=
\iint_{Q_{r_{n+1}}^{S}}a^{p_{n+1}}
+M_a^{p_{n+1}}\\
&\le
\bigl((C_N\Lambda(S)p_n^2\,4^{n+1})^\chi+1\bigr)
\bigl(D_n^{(a)}(S)\bigr)^{p_{n+1}}\\
&\le
2(C_N\Lambda(S)p_n^2\,4^{n+1})^\chi
\bigl(D_n^{(a)}(S)\bigr)^{p_{n+1}}.
\end{align*}
Taking the \(p_{n+1}\)-th root, and enlarging \(C_N\), gives
\begin{equation}\label{eq:en-a-recursion}
D_{n+1}^{(a)}(S)
\le
\bigl(C_N\Lambda(S)p_n^2\,4^{n+1}\bigr)^{1/p_n}
D_n^{(a)}(S).
\end{equation}
Indeed, the additional factor
\(2^{1/p_{n+1}}=(2^{1/\chi})^{1/p_n}\)
is absorbed into \((C_N)^{1/p_n}\).

Since
\[
\sum_{n=0}^\infty\frac1{p_n}=\frac{N+2}{4},
\qquad
\sum_{n=0}^\infty
\frac{n+\log p_n}{p_n}<\infty,
\]
iteration yields
\[
\sup_nD_n^{(a)}(S)
\le C_N\Lambda(S)^{(N+2)/4}D_0^{(a)}(S).
\]
By \eqref{eq:en-a-L2},
\[
D_0^{(a)}(S)
\le C_N(1+S)^{1/2}M_a\exp\left(C_N\int^S_0\lambda(t)dt\right).
\]
Moreover, \(Q_1^{S}\subset Q_{r_n}^{S}\), so
\(\|a\|_{L^{p_n}(Q_1^{S})}\le D_n^{(a)}(S)\).
Letting \(n\to\infty\), we conclude that
\[
\|a\|_{L^\infty(Q_1^{S})}
\le
C_N\Lambda(S)^{(N+2)/4}(1+S)^{1/2}
M_a\exp\left(C_N\int^S_0\lambda(t)dt\right).
\]
The constant is independent of \(y\) and \(S\). Define 
\begin{equation}\label{eq:en-Theta}
    \Theta(S):=C_N\Lambda(S)^{(N+2)/4}(1+S)^{1/2} \exp\!\left(C_N\int_0^S\lambda(t)\,dt\right).
\end{equation}
Choose $C_N$ large enough so that $\Theta(S)\ge 1.$ Covering \(\mathbb R^N\) by countably many unit balls gives
\begin{equation}\label{eq:en-a-upper}
\left\|\frac{\rho}{\omega}\right\|_
{L^\infty((0,S)\times\mathbb R^N)}
\le \Theta(S) M_a.
\end{equation}
Thus the density satisfies
\begin{equation}\label{eq:en-density-upper}
\rho(t,x)\le\Theta(S)M_a\,\omega(x) \le \Theta(S)e^L\omega(x), \qquad 0\le t\le S. 
\end{equation}

\subsubsection{An upper bound for \(z=\omega/\rho\)}
\label{subsec:z-upper}

\paragraph{Step 1: the reciprocal equation.}
Expanding \eqref{eq:en-a-equation}, we have
\[
a_t-\Delta a+b\cdot\nabla a
=(d-\operatorname{div}b)a.
\]
Since \(z=a^{-1}\),
\[
z_t=-a^{-2}a_t,
\qquad
\nabla z=-a^{-2}\nabla a,
\qquad
\Delta z=-a^{-2}\Delta a+2a^{-3}|\nabla a|^2.
\]
It follows that
\begin{equation}\label{eq:en-z-equation}
z_t-\Delta z+b\cdot\nabla z
=
(\operatorname{div}b-d)z
-2\frac{|\nabla z|^2}{z}.
\end{equation}

\paragraph{Step 2: the negative-power energy estimate.}
Multiply \eqref{eq:en-z-equation} by
\(\zeta^2z^{p-1}\), where \(p\ge2\).
The last term contributes
\[
-2\int\zeta^2z^{p-2}|\nabla z|^2dx.
\]
Moving it to the left increases the diffusion coefficient
from \(p-1\) to \(p+1\).

The term involving \(\operatorname{div}b\) is treated by
integration by parts:
\[
\int\zeta^2z^p\operatorname{div}b dx
=
-2\int\zeta z^p b\cdot\nabla\zeta dx
-p\int\zeta^2z^{p-1}b\cdot\nabla zdx.
\]
Combining this with the transport term on the left, we obtain
\begin{align}
\frac1p\frac{d}{dt}\int\zeta^2z^pdx
+(p+1)\int\zeta^2z^{p-2}|\nabla z|^2dx
={}&-2\int\zeta z^{p-1}\nabla z\cdot\nabla\zeta dx
\notag\\
&-(p+1)\int\zeta^2z^{p-1}b\cdot\nabla z dx
\notag\\
&-2\int\zeta z^p b\cdot\nabla\zeta dx
-\int\zeta^2dz^pdx.
\label{eq:en-z-energy-identity}
\end{align}
Thus no bound on \(\operatorname{div}b\) is used in the estimate.

The two gradient terms satisfy
\begin{align*}
2\left|\int\zeta z^{p-1}\nabla z\cdot\nabla\zeta dx\right|
&\le\frac{p+1}{4}\int\zeta^2z^{p-2}|\nabla z|^2dx
+\frac{C}{p+1}\int|\nabla\zeta|^2z^pdx,\\
(p+1)\left|\int\zeta^2z^{p-1}b\cdot\nabla z\right|
&\le\frac{p+1}{4}\int\zeta^2z^{p-2}|\nabla z|^2dx
+C(p+1)\|b \|_{L^{\infty}}^2\int\zeta^2z^pdx.
\end{align*}
The remaining terms obey
\begin{align*}
2\left|\int\zeta z^p b\cdot\nabla\zeta dx\right|
&\le\|b \|_{L^{\infty}}^2\int\zeta^2z^p dx
+\int|\nabla\zeta|^2z^p dx,\\
\left|\int\zeta^2dz^p dx\right|
&\le\|d \|_{L^{\infty}}\int\zeta^2z^p dx.
\end{align*}
Using
\[
|\nabla z^{p/2}|^2
=\frac{p^2}{4}z^{p-2}|\nabla z|^2
\]
and the definition of $\lambda(t),$ we obtain
\begin{equation}\label{eq:en-z-energy}
\frac{d}{dt}\int\zeta^2z^pdx
+c\int\zeta^2|\nabla z^{p/2}|^2dx
\le
Cp^2\int
\bigl(\lambda(t)\zeta^2+|\nabla\zeta|^2\bigr)z^pdx.
\end{equation}
The constants $c,C>0$ are independent of $p$ and $S$. In particular, this estimate does not require a bound on $\operatorname{div}v$.

\paragraph{Step 3: the quadratic starting estimate.}
Taking \(p=2\) and using the integrable weight \(\phi_y\)
as above gives
\[
\frac{d}{dt}\int\phi_y z^2dx
\le C_N\lambda(t)\int\phi_y z^2dx.
\]
Since \(z(0)\le M_z\), Gronwall's inequality yields
\begin{equation}\label{eq:en-z-L2}
\sup_{y\in\mathbb R^N}\sup_{0\le t\le S} \int_{\mathbb R^N}\phi_y(x)z(t,x)^2\,dx \le C_NM_z^2 \exp\left(C_N\int_0^S\lambda(t)\,dt\right).
\end{equation}
This estimate does not require \(z(0)\in L^2(\mathbb R^N)\).

\paragraph{Step 4: iteration for the reciprocal.}
 We now repeat the iteration for $z$ on $Q_r^S=(0,S)\times B_r(y)$. Since $\lambda(t)\le\Lambda(S)$ for $0\le t\le S$, integrating
\eqref{eq:en-z-energy} gives
\begin{align}
&\sup_{0\le t\le S}\int(\zeta z^{p/2})^2
+\int_0^S\int|\nabla(\zeta z^{p/2})|^2
\notag\\
&\qquad\le
\frac{C_N\Lambda(S)p^2}{(r-r')^2}
\left(
\iint_{Q_r^{S}}z^p+|B_2|M_z^p
\right).
\label{eq:en-z-local-energy}
\end{align}
Applying \eqref{eq:en-parabolic-sobolev} to
\(f=\zeta z^{p/2}\), we find
\begin{equation}\label{eq:en-z-step}
\|z\|_{L^{p\chi}(Q_{r'}^{S})}
\le
\left(\frac{C_N\Lambda(S)p^2}{(r-r')^2}\right)^{1/p}
\left(
\iint_{Q_r^{S}}z^p+|B_2|M_z^p
\right)^{1/p}.
\end{equation}

Using the exponents $p_n$ and radii $r_n$ from
Subsection~\ref{subsec:a-upper}, define
\[
D_n^{(z)}(S)
=
\left(
\iint_{Q_{r_n}^{S}}z^{p_n}dxdt+M_z^{p_n}
\right)^{1/p_n}.
\]
Let $r=r_n$, $r'=r_{n+1}$ and we have
\(r_n-r_{n+1}=2^{-(n+1)}\).
After absorbing \(|B_2|\) into the dimensional constant,
\eqref{eq:en-z-step} gives
\[
\|z\|_{L^{p_{n+1}}(Q_{r_{n+1}}^{S})}
\le (C_N\Lambda(S)p_n^2\,4^{n+1})^{1/p_n}D_n^{(z)}(S),
\qquad
C_N\Lambda(S)p_n^2\,4^{n+1}\ge1.
\]
Since \(M_z\le D_n^{(z)}(S)\),
\begin{align*}
\bigl(D_{n+1}^{(z)}(S)\bigr)^{p_{n+1}}
&=
\iint_{Q_{r_{n+1}}^{S}}z^{p_{n+1}}
+M_z^{p_{n+1}}\\
&\le
((C_N\Lambda(S)p_n^2\,4^{n+1})^\chi+1)
\bigl(D_n^{(z)}(S)\bigr)^{p_{n+1}}\\
&\le
2(C_N\Lambda(S)p_n^2\,4^{n+1})^\chi
\bigl(D_n^{(z)}(S)\bigr)^{p_{n+1}}.
\end{align*}
Hence, taking the \(p_{n+1}\)-th root, and enlarging \(C_N\),
we arrive at
\[
D_{n+1}^{(z)}(S)
\le
\bigl(C_N\Lambda(S)p_n^2\,4^{n+1}\bigr)^{1/p_n}
D_n^{(z)}(S).
\]
The same convergent product therefore gives
\[
\sup_nD_n^{(z)}(S)
\le C_N\Lambda(S)^{(N+2)/4}D_0^{(z)}(S).
\]
By \eqref{eq:en-z-L2},
\[
D_0^{(z)}
\le C_N(1+S)^{1/2}M_z\exp\left(C_N\int^S_0\lambda(t)dt\right).
\]
Passing to the limit \(p_n\to\infty\) and covering
\(\mathbb R^N\) by countably many unit balls, we conclude that
\begin{equation}\label{eq:en-z-upper}
\left\|\frac{\omega}{\rho}\right\|_
{L^\infty((0,S)\times\mathbb R^N)}
\le\Theta(S) M_z,
\end{equation}
where the constant in \eqref{eq:en-Theta} is chosen large
enough for both estimates.

Since \(z=1/a>0\), it follows that
\begin{equation}\label{eq:en-density-lower}
\rho(t,x)
\ge\frac{\omega(x)}{\Theta(S) M_z}
\ge\Theta(S)^{-1}e^{-L}\omega(x),\qquad 0\le t\le S.
\end{equation}

\subsubsection{The resulting upper and lower bounds for the density}
\label{subsec:density-two-sided}

Combining \eqref{eq:en-density-upper} and
\eqref{eq:en-density-lower}, we obtain the more precise bounds
\begin{equation}\label{eq:en-two-sided-precise}
\frac{1}{\Theta(S) M_z}
\le\frac{\rho(t,x)}{\omega(x)}
\le\Theta(S) M_a,\qquad 0\le t\le S.
\end{equation}
In particular, since \(M_a,M_z\le e^L\),
\begin{equation}\label{eq:en-two-sided-weight}
\Theta(S)^{-1}e^{-L}\omega(x)
\le\rho(t,x)
\le\Theta(S)e^L\omega(x).
\end{equation}
Finally, using
\(e^{-L}\rho_0\le\omega\le e^L\rho_0\), we also obtain
\begin{equation}\label{eq:en-two-sided-initial}
\Theta(S)^{-1}e^{-2L}\rho_0(x)
\le\rho(t,x)
\le\Theta(S)e^{2L}\rho_0(x),
\qquad 0\le t\le S.
\end{equation}
  These estimates hold for every $0<S<\min\{T,T_{*}\}$, with $\Theta(S)$ defined above. The lower bound is relative to the initial density and does not assert a positive constant lower bound on $\mathbb R^N$. It is therefore compatible with $\rho_0(x)\to0$ as $|x|\to\infty$. A bound for $\Theta(S)$ independent of $S$ will follow from the effective-velocity estimates in Section \ref{effvelocsec} and Section \ref{sec:iso-vsec}.
\section{An upper bound for the effective velocity when $\gamma> 1$}\label{effvelocsec}
\subsection{Reformulation of the equations}

To further reformulate the system, set
\begin{equation}
    c=\sqrt{\gamma}\,\rho^{\frac{\gamma-1}{2}}.
\end{equation}
The system \eqref{eq:effective-system-1} can then be written as
\begin{equation}
    \begin{cases}
        \partial_t\rho-\Delta\rho
        +\bigl(v-2\nabla\log\rho\bigr)\cdot\nabla\rho
        =
        -\rho\operatorname{div}v
        -2\rho|\nabla\log\rho|^2,\\[2mm]
        \partial_t v-\Delta v
        +\bigl(v-2\nabla\log\rho\bigr)\cdot\nabla v
        =
        -c^2\nabla\log\rho.
    \end{cases}
\end{equation}
Expressing the density equation in terms of $c$ yields
\begin{equation}
    \begin{cases}
        \displaystyle
        \partial_t c-\Delta c
        +\bigl(v-2\nabla\log\rho\bigr)\cdot\nabla c
        =
        -\frac{\gamma-1}{2}c\,\operatorname{div}v
        -\frac{\gamma^2-1}{4}c|\nabla\log\rho|^2,
        \\[3mm]
        \displaystyle
        \partial_t v-\Delta v
        +\bigl(v-2\nabla\log\rho\bigr)\cdot\nabla v
        =
        -c^2\nabla\log\rho,
    \end{cases}
    \qquad
    c=\sqrt{\gamma}\,\rho^{\frac{\gamma-1}{2}}.
\end{equation}

Fix $\varepsilon>0$ and define
\[
\mathcal P
=
\partial_t-\Delta
+\bigl(v-2\nabla\log\rho\bigr)\cdot\nabla,
\qquad
s=\sqrt{|v|^2+\varepsilon^2},
\qquad
e=\frac{v}{s},
\qquad
W_\varepsilon=s+\frac{2c}{\gamma-1}.
\]
The preceding system gives
\begin{equation}\label{2cgamma-1}
    \mathcal Pv=-c^2\nabla\log\rho,
    \qquad
    \mathcal P\left(\frac{2c}{\gamma-1}\right)
    =
    -c\,\operatorname{div}v
    -\frac{\gamma+1}{2}c|\nabla\log\rho|^2.
\end{equation}

Let
\[
G=(\partial_jv_i)_{i,j},
\qquad
B=I_N-e\otimes e.
\]
Direct differentiation gives
\[
\partial_js=e\cdot\partial_jv,
\qquad
\Delta s
=
e\cdot\Delta v
+\frac1s\sum_{j=1}^{N}
\left(
    |\partial_jv|^2-(e\cdot\partial_jv)^2
\right).
\]
Consequently,
\begin{equation}\label{Dsequ}
    \mathcal Ps
    =
    -c^2e\cdot\nabla\log\rho
    -\frac1s\operatorname{tr}(G^TBG).
\end{equation}

Adding \eqref{Dsequ} to the equation for
$2c/(\gamma-1)$ in \eqref{2cgamma-1}, we obtain
\[
\mathcal PW_\varepsilon
=
-c^2e\cdot\nabla\log\rho
-\frac1s\operatorname{tr}(G^TBG)
-c\,\operatorname{div}v
-\frac{\gamma+1}{2}c|\nabla\log\rho|^2.
\]
On the other hand, since
$c=\sqrt{\gamma}\rho^{(\gamma-1)/2}$,
\[
\nabla\left(\frac{2c}{\gamma-1}\right)
=
c\nabla\log\rho.
\]
Thus
\[
ce\cdot\nabla W_\varepsilon
=
ce\cdot\nabla s+c^2e\cdot\nabla\log\rho.
\]
Adding this identity to the preceding equation cancels
the terms involving $c^2e\cdot\nabla\log\rho$.
Moreover,
\[
\operatorname{div}v-e\cdot\nabla s
=
\operatorname{tr}G-e^TGe
=
\operatorname{tr}(BG).
\]
It follows that
\[
\begin{aligned}
    &\left[
        \partial_t-\Delta
        +(v-2\nabla\log\rho+ce)\cdot\nabla
    \right]W_\varepsilon\\
    &\qquad
    =
    -\frac1s\operatorname{tr}(G^TBG)
    -c\,\operatorname{tr}(BG)
    -\frac{\gamma+1}{2}c|\nabla\log\rho|^2.
\end{aligned}
\]

Since $B$ is symmetric and positive semidefinite,
completing the square gives
\[
\left|
    B^{1/2}\left(G+\frac{cs}{2}I_N\right)
\right|_F^2
=
\operatorname{tr}(G^TBG)
+cs\,\operatorname{tr}(BG)
+\frac{c^2s^2}{4}\operatorname{tr}B,
\]
where $|\cdot|_F$ denotes the Frobenius norm,
\[
|A|_F=\sqrt{\operatorname{tr}(A^TA)}.
\]
Substituting this identity into the preceding equation yields
\begin{equation}\label{LeWeN4c}
    \begin{aligned}
        &\left[
            \partial_t-\Delta
            +(v-2\nabla\log\rho+ce)\cdot\nabla
        \right]W_\varepsilon\\
        &=
        -\frac1s
        \left|
            B^{1/2}\left(G+\frac{cs}{2}I_N\right)
        \right|_F^2
        +\frac{c^2s}{4}\operatorname{tr}B
        -\frac{\gamma+1}{2}c|\nabla\log\rho|^2\\
        &\le
        \frac{c^2s}{4}\operatorname{tr}B
        \le
        \frac N4c^2s
        \le
        \frac N4c^2W_\varepsilon.
    \end{aligned}
\end{equation}
Here we have used the nonnegativity of the squared norm,
$c>0$,
\[
\operatorname{tr}B=N-|e|^2\le N,
\]
and $s\le W_\varepsilon$, which follows from $\gamma>1$.

\subsection{Bounding the effective velocity from above via the comparison lemma}

\noindent
\begin{proposition}[An effective velocity bound independent of a uniform
positive lower bound for the density]\label{prop: v L infty}
    Let $(\rho,v)$ be a local strong solution on
$[0,S]\times\mathbb R^N$, continuous up to $t=0$,
with $\rho>0$. Assume that
\[
\rho,v\in C([0,S];L^\infty(\mathbb R^N)),
\qquad
\sup_{0<t\le S}
\|\nabla\log\rho(t) \|_{L^{\infty}}<\infty.
\]
Then, for every $0\le t\le S$,
\begin{equation}\label{est: v L infty}
    \left\|
        |v(t)|
        +\frac{2\sqrt{\gamma}}{\gamma-1}
        \rho(t)^{(\gamma-1)/2}
    \right \|_{L^{\infty}}
    \le
    W_0\exp\left(
        \frac{N\gamma}{4}
        \int_0^t
        \|\rho(s) \|_{L^{\infty}}^{\gamma-1}\,\mathrm ds
    \right),
\end{equation}
where
\[
W_0
=
\left\|
    |v_0|
    +\frac{2\sqrt{\gamma}}{\gamma-1}
    \rho_0^{(\gamma-1)/2}
\right \|_{L^{\infty}}.
\]
\end{proposition}

\begin{proof}
    Define
    \begin{equation}\label{eq:k}
        k(t)
        =
        \frac{N\gamma}{4}
        \|\rho(t) \|_{L^{\infty}}^{\gamma-1},
        \qquad
        K(t)=\int_0^t k(\tau)\,\mathrm d\tau.
    \end{equation}
    Since
    $\rho\in C([0,S];L^\infty(\mathbb R^N))$,
    the map $t\mapsto\|\rho(t) \|_{L^{\infty}}$ is continuous.
    Consequently, $K\in C^1([0,S])$.

    Set
    \[
    \mathcal L_\varepsilon
    =
    \partial_t-\Delta
    +(v-2\nabla\log\rho+ce)\cdot\nabla.
    \]
    Since $W_\varepsilon>0$ and $\gamma>1$,
    \eqref{LeWeN4c} gives
    \begin{equation}\label{eq:Lreact}
        \mathcal L_\varepsilon W_\varepsilon
        \le k(t)W_\varepsilon.
    \end{equation}
    Let
    \[
    Z_\varepsilon(t,x)
    =
    e^{-K(t)}W_\varepsilon(t,x).
    \]
    As the integrating factor is independent of $x$,
    \begin{align}\label{eq:factor}
        \mathcal L_\varepsilon Z_\varepsilon
        &=
        e^{-K(t)}
        \mathcal L_\varepsilon W_\varepsilon
        -k(t)e^{-K(t)}W_\varepsilon
        \notag\\
        &\le0.
    \end{align}
    No additional spatial derivative terms arise.

    \textbf{Step 1: Verification of the hypotheses
    of the comparison principle.}

    Fix $S$ and $0<\varepsilon\le1$.
    The function
    \[
    s=\sqrt{|v|^2+\varepsilon^2}
    \]
    remains smooth at points where $v=0$.
    Moreover, since $\rho>0$, the function
    \[
    \frac{\sqrt{\gamma}}{\theta}\rho^\theta,
    \qquad
    \theta=\frac{\gamma-1}{2},
    \]
    has the required differential regularity
    at every finite spatial point.
    Thus $Z_\varepsilon$ satisfies
    the regularity assumptions of Lemma~\ref{lem:max}.

    Write
    \[
    R_S=\sup_{0\le t\le S}\|\rho(t) \|_{L^{\infty}},
    \qquad
    V_S=\sup_{0\le t\le S}\|v(t) \|_{L^{\infty}},
    \qquad
    L_S=\sup_{0<t\le S}
    \|\nabla\log\rho(t) \|_{L^{\infty}}.
    \]
    These quantities are finite by the hypotheses
    on the local solution.
    Since
    \[
    0<W_\varepsilon
    \le
    V_S+1+\frac{\sqrt{\gamma}}{\theta}R_S^\theta
    \]
    and $K(t)\ge0$, the function $Z_\varepsilon$
    is bounded on $[0,S]\times\mathbb R^N$.

    The drift is
    \[
    b_\varepsilon
    =
    v-2\nabla\log\rho
    +c\frac{v}{\sqrt{|v|^2+\varepsilon^2}}.
    \]
    Using
    \[
    \left|
        \frac{v}{\sqrt{|v|^2+\varepsilon^2}}
    \right|
    \le1,
    \]
    we obtain
    \begin{equation}\label{eq:driftbound}
        \|b_\varepsilon\|_{L^\infty_{t,x}}
        \le
        V_S+2L_S+\sqrt{\gamma}R_S^\theta
        =:M_S<\infty.
    \end{equation}
    This bound is uniform for $0<\varepsilon\le1$.
    The comparison principle requires no estimates
    on $\nabla b_\varepsilon$ or
    $\operatorname{div}b_\varepsilon$.

    Applying Lemma~\ref{lem:max} to \eqref{eq:factor},
    we find
    \[
    Z_\varepsilon(t,x)
    \le
    \sup_{y\in\mathbb R^N}Z_\varepsilon(0,y)
    =
    \|W_\varepsilon(0) \|_{L^{\infty}}.
    \]
    Therefore,
    \begin{equation}\label{eq:epsbound}
        \|W_\varepsilon(t) \|_{L^{\infty}}
        \le
        e^{K(t)}\|W_\varepsilon(0) \|_{L^{\infty}}.
    \end{equation}
    The drift bound $M_S$ used to apply
    the comparison principle does not appear
    on the right-hand side of this estimate.

    \textbf{Step 2: Passage to the limit,
    including at zeros of the velocity.}

    Define
    \[
    W(t,x)
    =
    |v(t,x)|+\frac{2c(t,x)}{\gamma-1},
    \qquad
    W_\varepsilon
    =
    s+\frac{2c}{\gamma-1}.
    \]
    Since
    \[
    0
    \le
    \sqrt{|v|^2+\varepsilon^2}-|v|
    =
    \frac{\varepsilon^2}
    {\sqrt{|v|^2+\varepsilon^2}+|v|}
    \le\varepsilon,
    \]
    we have, uniformly in time and space,
    \begin{equation}\label{eq:uniformeps}
        0\le W_\varepsilon-W\le\varepsilon.
    \end{equation}
    In particular, recalling that
    \[
    W_0
    =
    \left\|
        |v_0|
        +\frac{2\sqrt{\gamma}}{\gamma-1}
        \rho_0^{(\gamma-1)/2}
    \right \|_{L^{\infty}},
    \]
    we obtain
    \begin{equation}\label{eq:init-eps}
        \|W_\varepsilon(0) \|_{L^{\infty}}
        \le
        \|W(0) \|_{L^{\infty}}+\varepsilon
        =
        W_0+\varepsilon.
    \end{equation}
    Combining this with \eqref{eq:epsbound} gives
    \[
    \|W(t) \|_{L^{\infty}}
    \le
    \|W_\varepsilon(t) \|_{L^{\infty}}
    \le
    e^{K(t)}(W_0+\varepsilon).
    \]
    Letting $\varepsilon\downarrow0$, we conclude that
    \begin{equation}\label{eq:final}
            \begin{aligned}
                &\left\|
                    |v(t)|
                    +\frac{2\sqrt{\gamma}}{\gamma-1}
                    \rho(t)^{(\gamma-1)/2}
                \right \|_{L^{\infty}}   \le
                W_0\exp\left[
                    \frac{N\gamma}{4}
                    \int_0^t
                    \|\rho(\tau) \|_{L^{\infty}}^{\gamma-1}
                    \,\mathrm d\tau
                \right].
            \end{aligned}
    \end{equation}
    This proves Proposition \ref{prop: v L infty}, namely estimate \eqref{est: v L infty}.
    Since $K(0)=0$, the estimate also includes
    the initial time.

    \paragraph{Order of the limits.}
    We first fix the existing solution interval $[0,S]$
    and $\varepsilon>0$.
    In the comparison argument, we let the spatial
    barrier parameter $\delta\downarrow0$
    to obtain \eqref{eq:epsbound}.
    We then let $\varepsilon\downarrow0$.
    This procedure avoids taking second derivatives
    of $|v|$ at its zeros and does not require
    passing to the limit directly in the regularized
    differential inequality.
\end{proof}
\section{An upper bound for the effective velocity when $\gamma=1$}
\label{sec:iso-vsec}
In this section we apply the relative-density estimates from
Section~\ref{sec:relative-density} to estimate the effective
velocity in the isothermal case.
\subsection{Reformulation of the equations}

When $\gamma=1$, the effective system becomes
\begin{equation}\label{iso-v:system}
\begin{cases}
\rho_t-\Delta\rho
=-\operatorname{div}(\rho v),\\[1mm]
v_t-\Delta v
+(v-2g)\cdot\nabla v=-g,
\end{cases}
\qquad
g=\nabla\log\rho.
\end{equation}
Set $h=\log\rho$ and retain $h_0$, $L$, $H$, $\omega$, and $\ell$
from Subsection~\ref{sec:fixed-weight}. By
\eqref{eq:en-weight-bounds},
\[
|H-h_0|\le L,\qquad
\|\ell \|_{L^{\infty}}\le L,\qquad
\|\Delta H \|_{L^{\infty}}\le C_\eta L.
\]

Define
\[
\mathcal P
=\partial_t-\Delta+(v-2g)\cdot\nabla.
\]
The density equation in \eqref{iso-v:system} gives
\[
h_t-\Delta h
=|g|^2-v\cdot g-\operatorname{div}v.
\]
Consequently,
\begin{equation}\label{iso-v:PhPv}
\mathcal Pv=-g,\qquad
\mathcal Ph=-\operatorname{div}v-|g|^2.
\end{equation}

Introduce the comparison quantity
\begin{equation}\label{iso-v:comparison}
W=\frac12|v|^2+h-H
=\frac12|v|^2+\log\frac{\rho}{\omega}.
\end{equation}
Direct differentiation yields
\[
\mathcal P\left(\frac12|v|^2\right)
=-|\nabla v|_F^2-v\cdot g,
\]
and
\[
\mathcal PH
=-\Delta H+(v-2g)\cdot\ell.
\]
Combining these identities with \eqref{iso-v:PhPv}, we obtain
\[
\begin{aligned}
\mathcal PW
={}&-|\nabla v|_F^2-\operatorname{div}v
    -|g|^2-v\cdot g\\
&+\Delta H-v\cdot\ell+2g\cdot\ell.
\end{aligned}
\]
Completing the squares gives
\begin{equation}\label{iso-v:square-identity}
\begin{aligned}
\mathcal PW
={}&
-\left|\nabla v+\frac12I_N\right|_F^2
-\left|g+\frac12v-\ell\right|^2\\
&+\frac14|v|^2-2v\cdot\ell
  +|\ell|^2+\Delta H+\frac N4.
\end{aligned}
\end{equation}
Since
\[
2|v||\ell|
\le\frac14|v|^2+4L^2,
\]
it follows from \eqref{eq:en-weight-bounds} that
\begin{equation}\label{iso-v:comparison-inequality}
\mathcal PW\le\frac12|v|^2+A,
\end{equation}
where
\begin{equation}\label{iso-v:A}
A:=5L^2+\|\Delta H \|_{L^{\infty}}+\frac N4
\le C(1+L^2).
\end{equation}

We retain the relative density $a$ and its reciprocal $z$
from \eqref{eq:en-relative-density}. Their equations and the
initial comparisons have already been established in
\eqref{eq:en-a-equation}, \eqref{eq:en-z-equation}, and
\eqref{iso-v:initial-comparison}, respectively.

\subsection{Bounding the effective velocity from above via the comparison lemma}

\begin{proposition}[An upper bound for the effective velocity]
\label{iso-v:proposition}
Let $N=2,3$ and $\gamma=1$, and let $(\rho,v)$ be a positive
local strong solution of \eqref{iso-v:system} on
$[0,T_*)\times\mathbb R^N$, continuous up to $t=0$.
Assume that
\[
V_0:=\|v_0 \|_{L^{\infty}}<\infty,\qquad
L<\infty,
\]
and that
\[
v,\quad g,\quad h-h_0
\in L^\infty((0,S)\times\mathbb R^N)
\qquad\text{for every }S<T_*.
\]
Then, for every finite $T>0$ and
$0<S<\min\{T,T_*\}$,
\begin{equation}\label{iso-v:main-estimate}
\sup_{0\le t\le S}\|v(t) \|_{L^{\infty}}^2
\le C(1+V_0^2+L^2)e^{CT},
\end{equation}
where $C>0$ depends only on $N$ and the fixed mollifier $\eta$.
In particular, the constant is independent of $S$ and of the
local bounds for $v$, $g$ and $h-h_0$.
\end{proposition}

\begin{proof}
Fix $T<\infty$ and $0<S<\min\{T,T_*\}$.
For $0\le t\le S$, write
\[
X(t)=\|v(t) \|_{L^{\infty}}^2,\qquad
F(t)=\sup_{0\le s\le t}X(s),\qquad
K_0=1+L^2.
\]
Throughout the proof, $C$ denotes a generic positive constant depending only on $N$ and $\eta.$ The calculations below are justified on each fixed local interval
by spatial truncation and approximation.

\medskip
\noindent
\textbf{Step 1: Application of the comparison principle.}

Define
\[
Z(t,x)
=W(t,x)-At-\frac12\int_0^tX(s)\,ds.
\]
By \eqref{iso-v:comparison-inequality},
\[
\mathcal PZ
\le\frac12|v(t,x)|^2-\frac12X(t)\le0.
\]
On the fixed interval $[0,S]$, the drift $v-2g$ is bounded.
Furthermore,
\[
W=\frac12|v|^2+(h-h_0)+(h_0-H)
\]
is bounded by the local assumptions and
\eqref{eq:en-weight-bounds}.
Thus the comparison principle in Lemma~\ref{lem:max} gives
\[
Z(t,x)\le\sup_{y\in\mathbb R^N}W(0,y).
\]
Since
\[
\sup_xW(0,x)\le\frac12V_0^2+L,
\]
we obtain
\begin{equation}\label{iso-v:comparison-bound}
\sup_xW(t,x)
\le\frac12V_0^2+L+At
+\frac12\int_0^tX(s)\,ds.
\end{equation}
As in Lemma~\ref{lem:max}, the size of the local drift does not
enter the right-hand side.

\medskip
\noindent
\textbf{Step 2: An upper bound for the reciprocal relative density.}

Recall $\lambda$ and $\Lambda(t)$ from \eqref{iso-v:lambda}.
From \eqref{eq:en-weight-bounds} and
\eqref{eq:en-bd}, we have
\begin{equation}\label{iso-v:coefficient-bounds}
\lambda(s)\le C(K_0+X(s)),
\qquad
\Lambda(t)\le C(K_0+F(t)).
\end{equation}
For each $0<t\le S$, apply the estimate for $z$ from Subsection~\ref{subsec:z-upper} on the interval $[0,t]$. This gives 
\begin{equation}\label{iso-v:z-infinity}
   \|z\|_{L^\infty((0,t)\times\mathbb R^N)} \le M_z\Theta(t)=CM_z(1+t)^{1/2}\Lambda(t)^{(N+2)/4}
\exp\left(C\int_0^t\lambda(s)\,ds\right). 
\end{equation}
Let
\[
\alpha_N=\frac{N+2}{4}.
\]
Taking logarithms in \eqref{iso-v:z-infinity} and using
$M_z\le e^L$ from \eqref{eq:en-initial-norms} as well as \eqref{iso-v:coefficient-bounds}, we obtain
\begin{equation}\label{iso-v:log-density-bound}
\begin{aligned}
\log\|z\|_{L^\infty((0,t)\times\mathbb R^N)}
&\le
L+C+\frac12\log(1+t)\\
&+\alpha_N\log\bigl(C(K_0+F(t))\bigr)
+CK_0t+C\int_0^tX(s)\,ds.
\end{aligned}
\end{equation}

\medskip
\noindent
\textbf{Step 3: Closing the effective velocity estimate.}

By \eqref{iso-v:comparison},
\[
\frac12|v|^2=W+\log z.
\]
Combining \eqref{iso-v:comparison-bound} with
\eqref{iso-v:log-density-bound}, and taking the supremum in
time and space, yields
\begin{equation}\label{iso-v:combined-bound}
\begin{aligned}
\frac12F(t)
\le{}&
\frac12V_0^2+2L+C+\frac12\log(1+t)+CK_0t\\
&+\alpha_N\log\bigl(C(K_0+F(t))\bigr)
+C\int_0^tF(s)\,ds.
\end{aligned}
\end{equation}
Since $K_0\ge1$,
\[
\log(K_0+F(t))
\le\log K_0+\log(1+F(t)).
\]
Using
\[
L\le K_0,\qquad
\log K_0\le K_0,\qquad
\log(1+t)\le t,
\]
we infer from \eqref{iso-v:combined-bound} that
\begin{equation}\label{iso-v:logarithmic-closure}
F(t)
\le C(1+V_0^2+L^2)(1+t)
+2\alpha_N\log(1+F(t))
+C\int_0^tF(s)\,ds.
\end{equation}
For every $x\ge0$,
\[
2\alpha_N\log(1+x)\le\frac12x+C_N.
\]
Absorbing the logarithmic term in
\eqref{iso-v:logarithmic-closure}, we obtain
\[
F(t)
\le C(1+V_0^2+L^2)(1+t)
+C\int_0^tF(s)\,ds.
\]
Gronwall's inequality then yields
\[
F(t)\le C(1+V_0^2+L^2)e^{Ct},
\qquad 0\le t\le S.
\]
Since $S<T$, this proves \eqref{iso-v:main-estimate}.

The argument applies in both dimensions, with
\[
\alpha_N=
\begin{cases}
1, & N=2,\\[1mm]
5/4, & N=3.
\end{cases}
\]
The initial time is included by continuity.
No uniform positive lower bound for $\rho$ is required.
\end{proof}

\section{An upper bound for $\|\nabla\log\rho\|_{L^\infty}$}
\label{sec:log-density-gradient}

\subsection{Previously established relative-density bounds}\label{sec: relative-density bounds independent of S}

We retain the fixed weight, the relative density, and its reciprocal
from \eqref{eq:en-weight} and \eqref{eq:en-relative-density}.
Fix $T<\infty$ and $0<S<\min\{T,T_*\}$.
The estimates below hold for all $\gamma\ge1$. Indeed, the effective-velocity estimate of Section \ref{effvelocsec}, and Proposition~\ref{iso-v:proposition} for $\gamma=1$, give
\[
\|v\|_{L^\infty((0,S)\times\mathbb R^N)}\le V,
\]
where $V$ depends only on the initial data and $T$, and is independent
of $S$. With this choice of $V$, the quantity $\Lambda_*$ in \eqref{eq:en-Lambda} is independent of $S$. Since $\Lambda(S)\le\Lambda_*$ and $\int_0^S\lambda(t)\,dt\le T\Lambda_*,$ we obtain 
  \begin{equation}
      \Theta(S)\le\Theta_T, \qquad \Theta_T:= C_N\Lambda_*^{(N+2)/4}(1+T)^{1/2} e^{C_N\Lambda_*T}.
  \end{equation}
  The constant $\Theta_T$ depends only on the initial data and $T$, and is independent of $S$. Hence \eqref{eq:en-two-sided-precise} gives
\[
C_\rho^{-1}\le\frac{\rho(t,x)}{\omega(x)}\le C_\rho,
\qquad
C_\rho:=\Theta_T\max\{1,M_a,M_z\}.
\]
By \eqref{eq:en-two-sided-initial}, we may also choose
\[
R:=\Theta_Te^{2L}\|\rho_0 \|_{L^{\infty}}
\]
so that $\rho\le R$ on the same interval. We use these bounds below
without repeating the relative-density iteration.

\subsection{The bound for the
$\left\|\frac{\nabla\rho(t)}{\omega}\right\|_{L^\infty}$}

\subsubsection[An estimate for the weighted gradients of the conservative variables]
{An $L^6$ estimate for the weighted gradients of the conservative variables}
\label{en:weighted-gradients}

Set $m_i=\rho v_i$ and retain $g=\nabla\log\rho$.
The equation for $v_i$ reads
\[
\partial_t v_i-\Delta v_i
=-\sum_{j=1}^N(v_j-2g_j)\partial_jv_i
-\gamma\rho^{\gamma-1}g_i.
\]
Using this equation together with the continuity equation, we obtain
\begin{align*}
\partial_t m_i-\Delta m_i
&=v_i(\rho_t-\Delta\rho)
+\rho(\partial_t v_i-\Delta v_i)
-2\sum_{j=1}^N\partial_j\rho\,\partial_jv_i\\
&=-v_i\sum_{j=1}^N\partial_j(\rho v_j)
-\rho\sum_{j=1}^Nv_j\partial_jv_i
-\partial_i(\rho^\gamma)\\
&=-\sum_{j=1}^N\partial_j(\rho v_iv_j)
-\partial_i(\rho^\gamma).
\end{align*}
In the second equality, the terms containing
$2\rho g_j\partial_jv_i$ cancel because $\partial_j\rho=\rho g_j$.
Thus
\begin{equation}\label{en:conservative-system}
\begin{cases}
\displaystyle
\rho_t-\Delta\rho=-\sum_{j=1}^N\partial_jm_j,\\[4pt]
\displaystyle
\partial_t m_i-\Delta m_i
=-\sum_{j=1}^N\partial_j\!\left(\frac{m_im_j}{\rho}\right)
-\partial_i(\rho^\gamma),\qquad i=1,\ldots,N.
\end{cases}
\end{equation}

Introduce
\begin{equation}\label{en:weighted-components}
Z_k=\frac{\partial_k\rho}{\omega},
\qquad M_{ki}=\frac{\partial_km_i}{\omega},
\qquad
\Xi=\sum_{k=1}^NZ_k^2+\sum_{k,i=1}^NM_{ki}^2.
\end{equation}
Since $m_i=\rho v_i$, direct differentiation gives
\[
\frac1\omega\partial_j\!\left(\frac{m_im_j}{\rho}\right)
=v_jM_{ji}+v_iM_{jj}-v_iv_jZ_j,
\qquad
\frac1\omega\partial_i(\rho^\gamma)
=\gamma\rho^{\gamma-1}Z_i.
\]
We therefore set
\begin{equation}\label{en:weighted-sources}
\mathcal A=\sum_{j=1}^NM_{jj},
\qquad
\mathcal B_i
=\sum_{j=1}^N
\bigl(v_jM_{ji}+v_iM_{jj}-v_iv_jZ_j\bigr)
+\gamma\rho^{\gamma-1}Z_i.
\end{equation}
The right-hand sides of \eqref{en:conservative-system} are
$-\omega\mathcal A$ and $-\omega\mathcal B_i$, respectively.
The assumed bounds for $\rho$ and $v$ imply
\begin{equation}\label{en:source-coefficient-bound}
|\mathcal A|^2+\sum_{i=1}^N|\mathcal B_i|^2\le B^2\Xi,
\qquad B=C_N\bigl(1+V^2+\gamma R^{\gamma-1}\bigr).
\end{equation}
In particular, these coefficients can be estimated without using a bound
for $g$.

Differentiate \eqref{en:conservative-system} with respect to $x_k$
and divide by $\omega$. Since $\omega$ is independent of time,
$\nabla\omega=\omega\ell$, and
$\Delta\omega=\omega(\operatorname{div}\ell+|\ell|^2)$, we find
\begin{align}
\partial_t Z_k-\Delta Z_k
&=2\ell\cdot\nabla Z_k
+(\operatorname{div}\ell+|\ell|^2)Z_k
-\partial_k\mathcal A-\ell_k\mathcal A,
\label{en:Z-first}\\
\partial_t M_{ki}-\Delta M_{ki}
&=2\ell\cdot\nabla M_{ki}
+(\operatorname{div}\ell+|\ell|^2)M_{ki}
-\partial_k\mathcal B_i-\ell_k\mathcal B_i.
\label{en:M-first}
\end{align}
We keep $\partial_k\mathcal A$ and $\partial_k\mathcal B_i$
in this form and integrate them by parts when testing the equations.

For this purpose, put $d_*=|\ell|^2-\operatorname{div}\ell$ and define
the vector fields $\mathcal F_k$ and $\mathcal G_{ki}$ by
\begin{align*}
(\mathcal F_k)_j&=2\ell_jZ_k-\delta_{jk}\mathcal A,
&\mathcal R_k&=d_*Z_k-\ell_k\mathcal A,\\
(\mathcal G_{ki})_j&=2\ell_jM_{ki}-\delta_{jk}\mathcal B_i,
&\mathcal S_{ki}&=d_*M_{ki}-\ell_k\mathcal B_i.
\end{align*}
Equations \eqref{en:Z-first}--\eqref{en:M-first} become
\begin{equation}\label{en:component-divergence}
\begin{cases}
\partial_t Z_k-\Delta Z_k
=\operatorname{div}\mathcal F_k+\mathcal R_k,\\
\partial_t M_{ki}-\Delta M_{ki}
=\operatorname{div}\mathcal G_{ki}+\mathcal S_{ki}.
\end{cases}
\end{equation}
By \eqref{eq:en-weight-bounds} and
\eqref{en:source-coefficient-bound}, we may choose
\[
K=C_N\bigl\{1+L+L^2+(1+L)B\bigr\}
\]
such that
\begin{equation}\label{en:flux-remainder-bounds}
\begin{aligned}
\sum_k|\mathcal F_k|^2+\sum_{k,i}|\mathcal G_{ki}|^2
&\le K^2\Xi,\\
\sum_k|\mathcal R_k|^2+\sum_{k,i}|\mathcal S_{ki}|^2
&\le K^2\Xi.
\end{aligned}
\end{equation}

Fix $y\in\mathbb R^N$ and use the weight $\phi_y$ introduced in
Subsection~\ref{subsec:a-upper}. Then $|\nabla\phi_y|\le\phi_y$,
$\int\phi_ydx \le C_N$, and
$\phi_y\ge e^{-\sqrt5}$ on $B_2(y)$.
Write
\[
\mathcal D=\sum_k|\nabla Z_k|^2
+\sum_{k,i}|\nabla M_{ki}|^2,
\qquad
E_y=\int\phi_y\Xi^3dx,
\qquad
D_y=\int\phi_y\Xi^2\mathcal Ddx.
\]
We test the first equation in \eqref{en:component-divergence}
with $\phi_y\Xi^2Z_k$ and sum over $k$.
We then test the second equation with $\phi_y\Xi^2M_{ki}$
and sum over $k,i$. Adding the resulting identities, the time derivatives
give
\[
\sum_k\int\phi_y\Xi^2Z_k\partial_tZ_kdx
+\sum_{k,i}\int\phi_y\Xi^2M_{ki}\partial_tM_{ki}dx
=\frac16\frac{d}{dt}E_y.
\]
For the diffusion terms, observe that
\[
\sum_k Z_k\nabla Z_k+\sum_{k,i}M_{ki}\nabla M_{ki}
=\frac12\nabla\Xi.
\]
Consequently, the terms in which no derivative falls on $\phi_y$ are
\begin{align*}
&\sum_k\int\phi_y\nabla Z_k\cdot\nabla(\Xi^2Z_k)dx
+\sum_{k,i}\int\phi_y\nabla M_{ki}\cdot\nabla(\Xi^2M_{ki})dx\\
&\qquad
=D_y+\int\phi_y\Xi|\nabla\Xi|^2dx.
\end{align*}
The second term on the right is nonnegative. Keeping the terms containing
$\nabla\phi_y$ on the other side, we obtain
\begin{equation}\label{en:sixth-identity}
\frac16E_y'+D_y+\int\phi_y\Xi|\nabla\Xi|^2dx
=I_1+I_2+I_3+I_4,
\end{equation}
where
\begin{align*}
I_1={}&-\int\Xi^2\nabla\phi_y\cdot
\left(\sum_kZ_k\nabla Z_k
+\sum_{k,i}M_{ki}\nabla M_{ki}\right)dx,\\
I_2={}&-\sum_k\int\phi_y\mathcal F_k\cdot\nabla(\Xi^2Z_k)dx
-\sum_{k,i}\int\phi_y\mathcal G_{ki}\cdot\nabla(\Xi^2M_{ki})dx,\\
I_3={}&-\int\Xi^2\nabla\phi_y\cdot
\left(\sum_kZ_k\mathcal F_k
+\sum_{k,i}M_{ki}\mathcal G_{ki}\right)dx,\\
I_4={}&\int\phi_y\Xi^2
\left(\sum_k\mathcal R_kZ_k
+\sum_{k,i}\mathcal S_{ki}M_{ki}\right)dx.
\end{align*}

Cauchy's inequality gives
\[
|\nabla\Xi|\le C_N\Xi^{1/2}\mathcal D^{1/2}.
\]
Using the product rule and the triangle inequality, we also have
\begin{align*}
&\left(\sum_k|\nabla(\Xi^2Z_k)|^2
+\sum_{k,i}|\nabla(\Xi^2M_{ki})|^2\right)^{1/2}\\
&\qquad\le C_N\Xi^2\mathcal D^{1/2}
+C_N\Xi^{3/2}|\nabla\Xi|
\le C_N\Xi^2\mathcal D^{1/2}.
\end{align*}
It follows from $|\nabla\phi_y|\le\phi_y$ and Young's inequality that
\begin{align*}
|I_1|
&\le C_N\int\phi_y\Xi^{5/2}\mathcal D^{1/2}dx
\le\frac14D_y+C_NE_y,\\
|I_2|
&\le C_NK\int\phi_y\Xi^{5/2}\mathcal D^{1/2}dx
\le\frac14D_y+C_NK^2E_y.
\end{align*}
The remaining two terms are estimated directly by
\eqref{en:flux-remainder-bounds}:
\[
|I_3|\le C_NKE_y,
\qquad |I_4|\le C_NKE_y.
\]
Substitution into \eqref{en:sixth-identity} therefore yields
\begin{equation}\label{en:sixth-energy}
E_y'+3D_y\le C_N(1+K^2)E_y.
\end{equation}
At $t=0$,
\[
Z_k(0)=\frac{\rho_0}{\omega}(g_0)_k,
\qquad
M_{ki}(0)=\frac{\rho_0}{\omega}
\bigl(\partial_k(v_0)_i+(v_0)_i(g_0)_k\bigr).
\]
Since $v_0=u_0+g_0$, $u_0\in H^2\hookrightarrow L^\infty$, $g_0\in L^\infty$, and $\nabla u_0,\nabla g_0\in H^1\hookrightarrow L^6$ for $N=2,3$, we have $v_0\in L^\infty$ and $\nabla v_0\in L^6$.
Using \eqref{eq:en-weight-bounds} and the integrability of $\phi_y$,
we find
\begin{equation}\label{en:initial-sixth}
\sup_yE_y(0)^{1/6}
\le C_Ne^L\left\{\|\nabla v_0\|_{L^6}
+L\bigl(1+\|v_0\|_{L^\infty}\bigr)\right\}
=:M_E.
\end{equation}
No global $L^6$ assumption on $g_0$ is needed here, since its bounded
contribution is integrated against $\phi_y$.
Gronwall's inequality applied to \eqref{en:sixth-energy}, followed by
the lower bound for $\phi_y$ on $B_2(y)$, gives
\begin{equation}\label{en:local-sixth}
\sup_{y\in\mathbb R^N}\sup_{0\le t\le S}
\|\Xi(t)^{1/2}\|_{L^6(B_2(y))}
\le M_6,
\qquad
M_6=C_Ne^{C_N(1+K^2)T}M_{\mathrm E}.
\end{equation}
Thus each $Z_k$ and $M_{ki}$ has a uniformly local $L^6$ bound,
with a constant independent of $S$.

\subsubsection{Estimating
$\left\|\frac{\nabla\rho(t)}{\omega}\right\|_{L^\infty}$
by truncation}
\label{en:truncation-section}

To pass from \eqref{en:local-sixth} to a bound for the density gradient,
we use the following scalar estimate.

\begin{lemma}\label{en:truncation-lemma}
Let $\Omega=B_2(y)\subset\mathbb R^N$, $N=2,3$, and
$Q_S=(0,S)\times\Omega$. Suppose that $U$ is an energy solution of
\[
U_t-\Delta U=\operatorname{div}F+H,
\qquad U|_{\partial\Omega}=0,
\qquad \|U(0)\|_{L^\infty(\Omega)}\le M_0,
\]
where $F,H\in L^6(Q_S)$. Set
\[
M=\|F\|_{L^6(Q_S)}+\|H\|_{L^6(Q_S)}.
\]
Then
\begin{equation}\label{en:truncation-result}
\|U\|_{L^\infty(Q_S)}
\le M_0+C_NM|Q_S|^{\sigma_N},
\qquad \sigma_N=\frac1{N+2}-\frac16>0.
\end{equation}
The constant is independent of $y$ and $S$.
\end{lemma}

\begin{proof}
For $k\ge M_0$, let
\[
w=(U-k)_+,
\qquad A_k=\{(t,x)\in Q_S:U(t,x)>k\}.
\]
Then $w(0)=0$ and $w$ has zero boundary trace.
Testing the equation with $w$ and integrating from $0$ to $t$ gives
\[
\frac12\|w(t)\|_{L^2}^2
+\int_0^t\|\nabla w(s)\|_{L^2}^2\,\mathrm ds
=-\int_0^t\!\!\int_\Omega F\cdot\nabla w
+\int_0^t\!\!\int_\Omega Hw.
\]
The first term on the right is bounded by
\[
\left|\int_0^t\!\!\int_\Omega F\cdot\nabla w\right|
\le\frac14\int_0^t\!\!\int_\Omega|\nabla w|^2
+C\iint_{A_k}|F|^2.
\]
For the second term, Poincar\'e's inequality and Young's inequality give
\begin{align*}
\left|\int_0^t\!\!\int_\Omega Hw\right|
&\le C_N\int_0^t
\|H(s)\mathbf1_{\{U(s)>k\}}\|_{L^2}
\|\nabla w(s)\|_{L^2}\,\mathrm ds\\
&\le\frac14\int_0^t\!\!\int_\Omega|\nabla w|^2
+C_N\iint_{A_k}|H|^2.
\end{align*}
The Poincar\'e constant is the same for every translate of $B_2$.
H\"older's inequality in space and time yields
\[
\iint_{A_k}(|F|^2+|H|^2)\le M^2|A_k|^{2/3}.
\]
Hence
\begin{equation}\label{en:level-energy}
\sup_{0\le t\le S}\|w(t)\|_{L^2}^2
+\iint_{Q_S}|\nabla w|^2
\le C_NM^2|A_k|^{2/3}.
\end{equation}

With $\chi=(N+2)/N$ as above, the parabolic Sobolev inequality
states that
\[
\|w\|_{L^{2\chi}(Q_S)}^2
\le C_N
\left(\sup_{0\le t\le S}\|w(t)\|_{L^2}^2\right)^{2/(N+2)}
\left(\iint_{Q_S}|\nabla w|^2\right)^{N/(N+2)}.
\]
Since $w\ge h-k$ on $A_h$ whenever $h>k$, it follows that
\begin{equation}\label{en:level-inequality}
(h-k)^2|A_h|^{1/\chi}
\le\|w\|_{L^{2\chi}(Q_S)}^2
\le C_NM^2|A_k|^{2/3}.
\end{equation}
The exponent $\beta=2\chi/3$ is strictly greater than one:
$\beta=10/9$ for $N=3$ and $\beta=4/3$ for $N=2$.

Choose $D>0$, to be fixed below, and set
\[
k_n=M_0+D(1-2^{-n}),
\qquad q_n=\frac{|A_{k_n}|}{|Q_S|},
\qquad \delta=\beta-1,
\qquad b_*=2^{2\chi}.
\]
Taking $h=k_{n+1}$ and $k=k_n$ in
\eqref{en:level-inequality}, and using
$k_{n+1}-k_n=D2^{-n-1}$, we obtain
\[
q_{n+1}
\le\left(\frac{C_NM^2}{D^2}\right)^\chi
2^{2\chi(n+1)}|Q_S|^\delta q_n^{1+\delta}.
\]
Assume first that $M>0$. Take
\[
D=C_NM|Q_S|^{\delta/(2\chi)},
\]
with $C_N$ sufficiently large that the preceding inequality implies
\[
q_{n+1}\le b_*^{-1/\delta}b_*^nq_n^{1+\delta}.
\]
Since $q_0\le1$, induction gives $q_n\le b_*^{-n/\delta}$.
Indeed, this bound at level $n$ yields
\[
q_{n+1}
\le b_*^{-1/\delta+n-n(1+\delta)/\delta}
=b_*^{-(n+1)/\delta}.
\]
Thus $|A_{k_n}|\to0$ and $U\le M_0+D$ almost everywhere in $Q_S$.
Applying the same argument to $-U$ gives $U\ge-M_0-D$.
Finally,
\[
\frac{\delta}{2\chi}=\frac1{N+2}-\frac16=\sigma_N,
\]
which proves \eqref{en:truncation-result}.
If $M=0$, testing with $(U-M_0)_+$ and $(-U-M_0)_+$
gives the conclusion directly.
\end{proof}

The condition $6>N+2$ is what makes the exponent in the level-set
iteration greater than one.

We now return to the continuity equation.
Differentiating it with respect to $x_k$ and using
$\partial_km_j=\omega M_{kj}$, we find that
\eqref{en:Z-first} is equivalent to
\begin{align}
\partial_t Z_k-\Delta Z_k
&=2\ell\cdot\nabla Z_k
+(\operatorname{div}\ell+|\ell|^2)Z_k
-\sum_{j=1}^N(\partial_jM_{kj}+\ell_jM_{kj})\notag\\
&=\operatorname{div}F_k+H_k,
\label{en:scalar-Z-equation}
\end{align}
where
\begin{equation}\label{en:scalar-Z-sources}
(F_k)_j=2\ell_jZ_k-M_{kj},
\qquad
H_k=(|\ell|^2-\operatorname{div}\ell)Z_k
-\sum_{j=1}^N\ell_jM_{kj}.
\end{equation}
Here the continuity equation has been differentiated before the
factor $\omega$ is removed. In particular,
\[
\frac1\omega\partial_k\!\left(\sum_j\partial_jm_j\right)
=\sum_j\frac1\omega\partial_j(\omega M_{kj})
=\sum_j(\partial_jM_{kj}+\ell_jM_{kj}).
\]
By \eqref{en:local-sixth},
\begin{equation}\label{en:scalar-source-bound}
\begin{aligned}
\sup_y\sup_{0\le t\le S}\bigl(
&\|Z_k(t)\|_{L^6(B_2(y))}
+\|F_k(t)\|_{L^6(B_2(y))}\\
&+\|H_k(t)\|_{L^6(B_2(y))}\bigr)
\le C_N(1+L+L^2)M_6.
\end{aligned}
\end{equation}

Fix $y\in\mathbb R^N$ and choose $0\le\zeta\le1$ with
$\zeta\in C_c^\infty(B_2(y))$, $\zeta=1$ on $B_1(y)$, and
$|\nabla\zeta|+|\Delta\zeta|\le C_N$.
For $U=\zeta Z_k$, the product rule gives
\[
U_t-\Delta U
=\zeta\operatorname{div}F_k+\zeta H_k
-2\nabla\zeta\cdot\nabla Z_k-Z_k\Delta\zeta.
\]
Using
\[
-2\nabla\zeta\cdot\nabla Z_k
=-2\operatorname{div}(Z_k\nabla\zeta)+2Z_k\Delta\zeta,
\]
we can write this as
\begin{equation}\label{en:cutoff-equation}
U_t-\Delta U=\operatorname{div}\widetilde F+\widetilde H,
\qquad
\begin{cases}
\widetilde F=\zeta F_k-2Z_k\nabla\zeta,\\
\widetilde H=\zeta H_k-\nabla\zeta\cdot F_k+Z_k\Delta\zeta.
\end{cases}
\end{equation}
Since $\zeta$ is compactly supported in $B_2(y)$, $U$ has zero
boundary trace. Moreover,
\[
\|U(0)\|_{L^\infty}\le e^LL,
\]
and \eqref{en:scalar-source-bound} implies
\begin{equation}\label{en:cutoff-source-bound}
\|\widetilde F\|_{L^6(Q_S)}
+\|\widetilde H\|_{L^6(Q_S)}
\le C_N(1+L+L^2)S^{1/6}M_6,
\qquad Q_S=(0,S)\times B_2(y).
\end{equation}
Applying Lemma~\ref{en:truncation-lemma} to
\eqref{en:cutoff-equation} and noting that
$|Q_S|=S|B_2|$ and $1/6+\sigma_N=1/(N+2)$, we obtain
\[
\|Z_k\|_{L^\infty((0,S)\times B_1(y))}
\le e^LL+C_N(1+L+L^2)S^{1/(N+2)}M_6.
\]
The constant is independent of $y$. Taking the supremum over the
ball centers and combining the component estimates therefore gives
\begin{equation}\label{en:density-gradient-bound}
\sup_{0\le t\le S}
\left\|\frac{\nabla\rho(t)}{\omega}\right\|_{L^\infty}
\le C_Ne^LL+C_N(1+L+L^2)T^{1/(N+2)}M_6
=:M_\nabla.
\end{equation}
The estimate first holds as an essential supremum in space and time;
the time continuity supplied by the local solution class gives the
stated bound at every time.

\subsection{The bound for the logarithmic density gradient}

The relative density estimate recalled at the beginning of this section is
\[
C_\rho^{-1}\le\frac{\rho(t,x)}{\omega(x)}\le C_\rho,
\qquad 0\le t\le S,
\]
where $C_\rho$ is independent of $S$.
Dividing the bound for $\nabla\rho/\omega$ by the lower bound for
$\rho/\omega$, we conclude that
\begin{equation}\label{en:final-log-bound}
\begin{aligned}
\sup_{0\le t\le S}\|\nabla\log\rho(t)\|_{L^\infty}
&=\sup_{0\le t\le S}
\left\|\frac{\nabla\rho(t)/\omega}{\rho(t)/\omega}\right\|_{L^\infty}\\
&\le C_\rho M_\nabla.
\end{aligned}
\end{equation}
\section{Higher-order estimates for the density and the effective velocity}
\label{sec:higher-order}
\subsection{Higher-order estimates when $\gamma>1$}
We now derive higher-order estimates from the bounds obtained in the
preceding sections. Throughout this section, let $\gamma>1$ and write
\[
g=\nabla\log\rho,\qquad
\mathfrak p=\frac{\nabla P(\rho)}{\rho}
           =\gamma\rho^{\gamma-1}g,\qquad
\Phi=|g|^2-v\cdot g.
\]
The equations take the form
\begin{equation}\label{ho:system}
\begin{cases}
\rho_t-\Delta\rho=-\operatorname{div}(\rho v),\\
v_t-\Delta v=-(v-2g)\cdot\nabla v-\mathfrak p,\\
g_t-\Delta g=\nabla\Phi-\nabla\operatorname{div}v.
\end{cases}
\end{equation}
We retain the density equation in divergence form.

Fix $T<\infty$ and $0<S<\min\{T,T_*\}$. The preceding estimates give
\begin{equation}\label{ho:known-bounds}
\sup_{0\le t\le S}
\bigl(
\|\rho(t) \|_{L^{\infty}}+\|v(t) \|_{L^{\infty}}+\|g(t) \|_{L^{\infty}}
\bigr)
\le R+V+G_\infty, 
\end{equation}
where the three norms are bounded by $R$, $V$ and $G_\infty$,
respectively. Moreover,
\[
\frac{\rho(t,x)}{\omega(x)}\le C_\rho,
\qquad
\omega(x)\le e^L\rho_0(x),
\]
and hence
\begin{equation}\label{ho:relative-upper}
\rho(t,x)\le A_\rho\rho_0(x),
\qquad
A_\rho=C_\rho e^L.
\end{equation}
All these constants are independent of $S$.

We use the initial regularity of the local solution:
\begin{equation}\label{ho:initial-data}
\begin{aligned}
&0<\rho_0\in H^3,\qquad u_0\in H^2,\\
&g_0=\nabla\log\rho_0\in L^\infty,\qquad
\nabla g_0\in H^1,\qquad
\mathfrak p_0=\gamma\rho_0^{\gamma-1}g_0\in L^2.
\end{aligned}
\end{equation}
Setting $v_0=u_0+g_0$, we have
\[
v_0\in L^\infty,\qquad \nabla v_0\in H^1.
\]
We do not assume that $v_0$ or $g_0$ belongs to $L^2$.
All spatial norms below are taken over $\mathbb R^N$.
The constant $C$ denotes a positive constant that may depend on
$N,\gamma,T,R,V,G_\infty,A_\rho$ and the initial norms in
\eqref{ho:initial-data}, but is independent of $S$.

\medskip
\noindent
\begin{proposition}[Higher-order estimates]\label{higherordergamma>1}
Let $N=2,3$, and let $(\rho,v)$ be the positive local strong solution
with initial data satisfying \eqref{ho:initial-data}.
Under \eqref{ho:known-bounds}--\eqref{ho:relative-upper}, we have
\begin{equation}\label{ho:main-estimate}
\begin{aligned}
&\sup_{0\le t\le S}\Bigl(
 \|\rho(t)\|_{H^3}^2+\|u(t)\|_{H^2}^2
 +\|\nabla v(t)\|_{H^1}^2+\|\nabla g(t)\|_{H^1}^2\\
&\qquad
 +\|v(t)-v_0\|_{H^2}^2
 +\|g(t)-g_0\|_{H^2}^2
 +\|\mathfrak p(t)\|_{L^{2}}^2
 \Bigr)\\
&\qquad
+\int_0^S\Bigl(
 \|\rho(t)\|_{H^4}^2+\|u(t)\|_{H^3}^2
 +\|\nabla v(t)\|_{H^2}^2
 +\|\nabla g(t)\|_{H^2}^2
 \Bigr)\,\mathrm dt
\le C.
\end{aligned}
\end{equation}
In addition,
\begin{equation}\label{ho:time-estimate}
\begin{aligned}
&\sup_{0\le t\le S}\Bigl(
 \|\rho_t(t)\|_{H^1}^2+\|u_t(t)\|_{L^{2}}^2
 +\|v_t(t)\|_{L^{2}}^2+\|g_t(t)\|_{L^{2}}^2
 \Bigr)\\
&\qquad
+\int_0^S\Bigl(
 \|\rho_t(t)\|_{H^2}^2+\|u_t(t)\|_{H^1}^2
 +\|v_t(t)\|_{H^1}^2+\|g_t(t)\|_{H^1}^2
 \Bigr)\,\mathrm dt
\le C.
\end{aligned}
\end{equation}
Moreover,
\begin{equation}\label{ho:time-continuity}
\begin{aligned}
&\rho\in C([0,S];H^3),\\
&u,\ v-v_0,\ g-g_0\in C([0,S];H^2),\\
&\nabla v,\nabla g\in C([0,S];H^1).
\end{aligned}
\end{equation}
\end{proposition}

\begin{proof}
For each fixed $S<T_*$, the local solution has the Sobolev regularity
needed for the following tests. They can be justified by smoothing,
spatial truncation and passage to the limit. The purpose of the
estimates is to obtain constants independent of $S$.

\medskip
\noindent
\textbf{Step 1: Estimates for the pressure term.}

Set
\[
r=v-v_0,\qquad z=g-g_0,\qquad
J(t)=\|r(t)\|_{L^{2}}^2,\qquad Q(t)=\|z(t)\|_{L^{2}}^2,\qquad
P_0=\|\mathfrak p_0\|_{L^{2}}^2.
\]
Since $r,z\in C([0,S];H^2)$ in the local solution class,
$J(t)$ and $Q(t)$ are well defined. We write
\begin{equation}\label{ho:pressure-decomposition}
\mathfrak p
=\gamma\rho^{\gamma-1}(g-g_0)
 +\left(\frac{\rho}{\rho_0}\right)^{\gamma-1}\mathfrak p_0.
\end{equation}
Using \eqref{ho:known-bounds} and \eqref{ho:relative-upper}, we obtain
\begin{equation}\label{ho:pressure-zero}
\|\mathfrak p(t)\|_{L^{2}}^2
\le C\bigl(Q(t)+P_0\bigr).
\end{equation}
This estimate does not require $\rho^{\gamma-1}\in L^2$.

Since $\partial_k\rho=\rho g_k$, direct differentiation gives
\begin{equation}\label{ho:pressure-first}
\partial_k\mathfrak p_i
=\gamma\rho^{\gamma-1}\partial_k g_i
 +(\gamma-1)g_k\mathfrak p_i.
\end{equation}
Differentiating once more yields
\begin{align}\label{ho:pressure-second}
\partial_{\ell k}\mathfrak p_i
={}&\gamma\rho^{\gamma-1}\partial_{\ell k}g_i
 +\gamma(\gamma-1)\rho^{\gamma-1}
  g_\ell\partial_k g_i
 \notag\\
&+(\gamma-1)(\partial_\ell g_k)\mathfrak p_i
 +(\gamma-1)g_k\partial_\ell\mathfrak p_i.
\end{align}
As $\|\mathfrak p \|_{L^{\infty}}\le\gamma R^{\gamma-1}G_\infty$,
\eqref{ho:pressure-zero}--\eqref{ho:pressure-second} imply
\begin{equation}\label{ho:pressure-derivatives}
\begin{aligned}
\|\nabla\mathfrak p\|_{L^{2}}^2
&\le C\bigl(\|\nabla g\|_{L^{2}}^2+Q+P_0\bigr),\\
\|\nabla^2\mathfrak p\|_{L^{2}}^2
&\le C\bigl(
 \|\nabla^2g\|_{L^{2}}^2+\|\nabla g\|_{L^{2}}^2+Q+P_0
 \bigr).
\end{aligned}
\end{equation}
In particular, no negative power of the density needs to be estimated.

\medskip
\noindent
\textbf{Step 2: Interpolation and product estimates.}

We shall use the following Gagliardo--Nirenberg interpolation inequalities
\begin{align}
\|\nabla f\|_{L^{4}}
&\le C_N\|f \|_{L^{\infty}}^{1/2}\|\nabla^2f\|_{L^{2}}^{1/2},
\label{ho:interpolation-two}\\
\|\nabla f\|_{L^{6}}
&\le C_N\|f \|_{L^{\infty}}^{2/3}\|\nabla^3f\|_{L^{2}}^{1/3},
\qquad
\|\nabla^2f\|_{L^{3}}
\le C_N\|f \|_{L^{\infty}}^{1/3}\|\nabla^3f\|_{L^{2}}^{2/3}.
\label{ho:interpolation-three}
\end{align}
These inequalities apply to bounded functions whose indicated
derivatives are square integrable and do not require $f\in L^2$. Thus they apply to both $v$ and $g$.


If $f$ and $h$ are bounded and have the required derivatives,
the product rule and these inequalities imply
\begin{align}
\|\nabla^j(fh)\|_{L^{2}}
&\le C\bigl(\|\nabla^jf\|_{L^{2}}+\|\nabla^jh\|_{L^{2}}\bigr),
&&j=1,2,3,
\label{ho:product}\\
\|\nabla^j(f\partial_kh)\|_{L^{2}}
&\le C\bigl(
 \|\nabla^{j+1}f\|_{L^{2}}+\|\nabla^{j+1}h\|_{L^{2}}
 \bigr),
&&j=1,2,
\label{ho:transport-product}
\end{align}
where the constants depend on their $L^\infty$ bounds.
The mixed terms are controlled, for example, by
\[
\|\nabla f\,\nabla h\|_{L^{2}}
\le\|\nabla f\|_{L^{4}}\|\nabla h\|_{L^{4}}
\le C\bigl(\|\nabla^2f\|_{L^{2}}+\|\nabla^2h\|_{L^{2}}\bigr)
\]
and
\[
\|\nabla f\,\nabla^2h\|_{L^{2}}
\le\|\nabla f\|_{L^{6}}\|\nabla^2h\|_{L^{3}}
\le C\bigl(\|\nabla^3f\|_{L^{2}}+\|\nabla^3h\|_{L^{2}}\bigr).
\]

\medskip
\noindent
\textbf{Step 3: First- and second-order energy estimates.}

Let $\lambda_0=4N^2$. For $j=1,2$, define
\begin{equation}\label{ho:energies}
\begin{aligned}
\mathcal E_j
&=\sum_{|\alpha|=j}
 \left(
 \lambda_0\|\partial^\alpha v\|_{L^{2}}^2
 +\|\partial^\alpha g\|_{L^{2}}^2
 +\|\partial^\alpha\rho\|_{L^{2}}^2
 \right),\\
\mathcal D_j
&=\sum_{|\alpha|=j}
 \left(
 \lambda_0\|\nabla\partial^\alpha v\|_{L^{2}}^2
 +\|\nabla\partial^\alpha g\|_{L^{2}}^2
 +\|\nabla\partial^\alpha\rho\|_{L^{2}}^2
 \right).
\end{aligned}
\end{equation}
Apply $\partial^\alpha$ to \eqref{ho:system}, test the three equations
with $\partial^\alpha\rho$, $\lambda_0\partial^\alpha v$
and $\partial^\alpha g$, respectively, and sum over $|\alpha|=j$.
Integration by parts gives
\begin{equation}\label{ho:energy-identity}
\frac12\mathcal E_j'+\mathcal D_j
=I_1+I_2+I_3+I_4+I_5,
\end{equation}
where
\begin{align*}
I_1
&=-\lambda_0\sum_{|\alpha|=j}
 \int
 \partial^\alpha\bigl((v-2g)\cdot\nabla v\bigr)
 \cdot\partial^\alpha v\,dx,\\
I_2
&=-\lambda_0\sum_{|\alpha|=j}
 \int
 \partial^\alpha\mathfrak p
 \cdot\partial^\alpha v\,dx,\\
I_3
&=-\sum_{|\alpha|=j}
 \int
 \partial^\alpha\Phi\,
 \operatorname{div}\partial^\alpha g\,dx,\\
I_4
&=\sum_{|\alpha|=j}
 \int
 \operatorname{div}\partial^\alpha v\,
 \operatorname{div}\partial^\alpha g\,dx,\\
I_5
&=\sum_{|\alpha|=j}
 \int
 \partial^\alpha(\rho v)
 \cdot\nabla\partial^\alpha\rho\,dx.
\end{align*}

By \eqref{ho:transport-product} and Young's inequality,
\[
|I_1|
\le C\mathcal D_j^{1/2}\mathcal E_j^{1/2}
\le\varepsilon\mathcal D_j+C_\varepsilon\mathcal E_j.
\]
The pressure estimates \eqref{ho:pressure-derivatives} give
\[
|I_2|
\le C\bigl(\mathcal E_1+\mathcal E_j+Q+P_0\bigr).
\]
Since $\Phi=|g|^2-v\cdot g$, \eqref{ho:product} yields
\[
\|\nabla^j\Phi\|_{L^{2}}
\le C\bigl(\|\nabla^jg\|_{L^{2}}+\|\nabla^jv\|_{L^{2}}\bigr).
\]
Consequently,
\[
|I_3|
\le\varepsilon\mathcal D_j+C_\varepsilon\mathcal E_j.
\]
For the second-order coupling, we use
$|\operatorname{div}w|\le\sqrt N|\nabla w|$ to obtain
\begin{align*}
|I_4|
&\le N\sum_{|\alpha|=j}
 \|\nabla\partial^\alpha v\|_{L^{2}}
 \|\nabla\partial^\alpha g\|_{L^{2}}\\
&\le\sum_{|\alpha|=j}
 \left(
 N^2\|\nabla\partial^\alpha v\|_{L^{2}}^2
 +\frac14\|\nabla\partial^\alpha g\|_{L^{2}}^2
 \right)
\le\frac14\mathcal D_j.
\end{align*}
Finally, applying \eqref{ho:product} directly to $\rho v$ gives
\[
|I_5|
\le\varepsilon\mathcal D_j+C_\varepsilon\mathcal E_j.
\]
Taking $\varepsilon=1/12$ in these estimates and substituting into
\eqref{ho:energy-identity}, we find
\begin{equation}\label{ho:joint-energy}
(\mathcal E_1+\mathcal E_2)'
+\mathcal D_1+\mathcal D_2
\le C\bigl(\mathcal E_1+\mathcal E_2+Q+P_0\bigr).
\end{equation}

We also need an estimate for $Q$. The last equation in
\eqref{ho:system} and \eqref{ho:product} imply
\[
\|g_t\|_{L^{2}}
\le C\bigl(
 \|\nabla^2g\|_{L^{2}}+\|\nabla^2v\|_{L^{2}}
 +\|\nabla g\|_{L^{2}}+\|\nabla v\|_{L^{2}}
 \bigr).
\]
Since $z_t=g_t$, it follows that
\begin{equation}\label{ho:affine-energy}
Q'=2\int z\cdot g_t\,dx
\le C\bigl(Q+\mathcal E_1+\mathcal E_2\bigr).
\end{equation}
Thus the term $Q$ in \eqref{ho:pressure-zero} is controlled
jointly with $\mathcal E_1+\mathcal E_2$ through the closed
system \eqref{ho:joint-energy}--\eqref{ho:affine-energy}.

\medskip
\noindent
\textbf{Step 4: The third-order density estimate.}

Set
\[
Z_3=\sum_{|\alpha|=3}\|\partial^\alpha\rho\|_{L^{2}}^2,
\qquad
Q_4=\sum_{|\alpha|=3}\|\nabla\partial^\alpha\rho\|_{L^{2}}^2.
\]
Applying $\partial^\alpha$ to the density equation in
\eqref{ho:system}, testing with $\partial^\alpha\rho$, and integrating by parts over
$\mathbb R^N$, we obtain, after summing over $|\alpha|=3$,
\begin{equation}\label{ho:density-third-identity}
\frac12Z_3'+Q_4
=\sum_{|\alpha|=3}
 \int
 \partial^\alpha(\rho v)
 \cdot\nabla\partial^\alpha\rho\,dx.
\end{equation}
The third-order product estimate in \eqref{ho:product} shows that
\[
\|\nabla^3(\rho v)\|_{L^{2}}
\le C\bigl(\|\nabla^3\rho\|_{L^{2}}+\|\nabla^3v\|_{L^{2}}\bigr).
\]
Therefore, for a fixed $C_1\ge1$,
\begin{equation}\label{ho:density-third}
Z_3'+Q_4\le C_1\bigl(Z_3+\mathcal D_2\bigr).
\end{equation}
Here we used $\|\nabla^3v\|_{L^{2}}^2\le C\mathcal D_2$,
which follows directly from the definition of $\mathcal D_2$.

The zeroth-order density estimate is obtained by testing its
equation with $\rho$:
\[
\frac12\frac{d}{dt}\|\rho\|_{L^{2}}^2
+\|\nabla\rho\|_{L^{2}}^2
=\int\rho v\cdot\nabla\rho\,dx
\le\frac12\|\nabla\rho\|_{L^{2}}^2+C\|\rho\|_{L^{2}}^2.
\]
Hence
\begin{equation}\label{ho:density-zero}
\frac{d}{dt}\|\rho\|_{L^{2}}^2
+\|\nabla\rho\|_{L^{2}}^2
\le C\|\rho\|_{L^{2}}^2.
\end{equation}

For the effective velocity, we estimate $r=v-v_0$.
By $\eqref{ho:system}_2$
and \eqref{ho:pressure-zero},
\[
\|v_t\|_{L^{2}}^2
\le C\bigl(
 \|\nabla^2v\|_{L^{2}}^2+\|\nabla v\|_{L^{2}}^2
 +Q+P_0
 \bigr)
\le C\bigl(
 \mathcal E_1+\mathcal E_2+Q+P_0
 \bigr).
\]
Since $r_t=v_t$, we obtain
\begin{equation}\label{ho:velocity-zero}
J'=2\int r\cdot v_t\,dx
\le C\bigl(
 J+\mathcal E_1+\mathcal E_2+Q+P_0
 \bigr).
\end{equation}

\medskip
\noindent
\textbf{Step 5: Closing the energy estimate.}

Choose $A_*=2C_1+1$ and define
\begin{align*}
\mathcal E
&=1+P_0+A_*(\mathcal E_1+\mathcal E_2)
  +J+Q+\|\rho\|_{L^{2}}^2+Z_3,\\
\mathcal D
&=\frac{A_*}{2}(\mathcal D_1+\mathcal D_2)
  +\|\nabla\rho\|_{L^{2}}^2+Q_4.
\end{align*}
Adding $A_*$ times \eqref{ho:joint-energy} to
\eqref{ho:affine-energy}, \eqref{ho:density-third},
\eqref{ho:density-zero} and \eqref{ho:velocity-zero},
we absorb $C_1\mathcal D_2$ into the left-hand side and obtain
\begin{equation}\label{ho:closed-energy}
\mathcal E'+\mathcal D\le C\mathcal E.
\end{equation}
The constants in this inequality do not involve higher-order norms.
Since $J(0)=Q(0)=0$, the initial energy is finite under
\eqref{ho:initial-data}.
Gronwall's inequality gives
\begin{equation}\label{ho:gronwall}
\sup_{0\le t\le S}\mathcal E(t)
+\int_0^S\mathcal D(t)\,\mathrm dt
\le 2e^{CT}\mathcal E(0).
\end{equation}

The energy controls
\[
\|\rho\|_{H^3}^2
+\|\nabla v\|_{H^1}^2
+\|\nabla g\|_{H^1}^2+J+Q.
\]
Moreover,
\begin{align*}
\|v-v_0\|_{H^2}^2
&\le C\bigl(
 J+\|\nabla v\|_{H^1}^2+\|\nabla v_0\|_{H^1}^2
 \bigr),\\
\|g-g_0\|_{H^2}^2
&\le C\bigl(
 Q+\|\nabla g\|_{H^1}^2+\|\nabla g_0\|_{H^1}^2
 \bigr),
\end{align*}
and \eqref{ho:pressure-zero} controls $\|\mathfrak p\|_{L^{2}}^2$.

The dissipation controls the highest-order derivatives in
$\|\rho\|_{H^4}$, $\|\nabla v\|_{H^2}$ and
$\|\nabla g\|_{H^2}$.
The lower-order parts of these time-integrated norms are bounded by
$T\sup_{t\le S}\mathcal E(t)$.
In particular,
\[
\int_0^S
\bigl(\|\nabla v\|_{L^{2}}^2+\|\nabla g\|_{L^{2}}^2\bigr)\,\mathrm dt
\le CT\sup_{0\le t\le S}\mathcal E(t).
\]

Finally, since
\[
u=u_0+r-z,\qquad \nabla u=\nabla v-\nabla g,
\]
we have
\[
\sup_{0\le t\le S}\|u(t)\|_{H^2}^2
\le C\left(
 \|u_0\|_{H^2}^2
 +\sup_{0\le t\le S}
  \bigl(\|r(t)\|_{H^2}^2+\|z(t)\|_{H^2}^2\bigr)
 \right)
\le C.
\]
Also,
\begin{align*}
\int_0^S\|u\|_{H^3}^2\,\mathrm dt
&\le C\int_0^S
 \bigl(
 \|u\|_{L^{2}}^2+\|\nabla v\|_{H^2}^2
          +\|\nabla g\|_{H^2}^2
 \bigr)\,\mathrm dt\\
&\le C.
\end{align*}
This proves \eqref{ho:main-estimate}.

\medskip
\noindent
\textbf{Step 6: Estimates for the time derivatives.}

From \eqref{ho:system}, we have
\begin{align*}
\|v_t\|_{L^{2}}
&\le C\bigl(
 \|\nabla^2v\|_{L^{2}}+\|\nabla v\|_{L^{2}}+\|\mathfrak p\|_{L^{2}}
 \bigr),\\
\|g_t\|_{L^{2}}
&\le C\bigl(
 \|\nabla^2g\|_{L^{2}}+\|\nabla^2v\|_{L^{2}}
 +\|\nabla g\|_{L^{2}}+\|\nabla v\|_{L^{2}}
 \bigr).
\end{align*}
Thus $v_t$ and $g_t$ are uniformly bounded in $L^2$.
Taking one spatial derivative of their equations and using
\eqref{ho:product}--\eqref{ho:transport-product}, we obtain
\begin{align*}
\|\nabla v_t\|_{L^{2}}
&\le C\bigl(
 \|\nabla^3v\|_{L^{2}}+\|\nabla^2v\|_{L^{2}}+\|\nabla^2g\|_{L^{2}}
 +\|\nabla\mathfrak p\|_{L^{2}}
 \bigr),\\
\|\nabla g_t\|_{L^{2}}
&\le C\bigl(
 \|\nabla^3g\|_{L^{2}}+\|\nabla^3v\|_{L^{2}}
 +\|\nabla^2g\|_{L^{2}}+\|\nabla^2v\|_{L^{2}}
 \bigr).
\end{align*}
The right-hand sides are square integrable in time by
\eqref{ho:pressure-derivatives} and \eqref{ho:main-estimate}.
Since $u_t=v_t-g_t$, the same bounds give
\[
\sup_{0\le t\le S}\|u_t(t)\|_{L^{2}}^2
+\int_0^S\|u_t(t)\|_{H^1}^2\,\mathrm dt
\le C.
\]

The density equation gives
\begin{align*}
\|\rho_t\|_{H^1}
&\le C\bigl(
 \|\rho\|_{H^3}+\|\nabla v\|_{H^1}
 \bigr),\\
\|\rho_t\|_{H^2}
&\le C\bigl(
 \|\rho\|_{H^4}+\|\nabla v\|_{H^2}
 \bigr).
\end{align*}
These follow by applying \eqref{ho:product} to
$\operatorname{div}(\rho v)$.
Combining the preceding bounds proves \eqref{ho:time-estimate}.

It remains to verify the time continuity in
\eqref{ho:time-continuity}.
Since
\[
\rho\in L^2(0,S;H^4),
\qquad
\rho_t\in L^2(0,S;H^2),
\]
the standard Sobolev time-continuity theorem yields
\[
\rho\in C([0,S];H^3).
\]
Similarly,
\[
\nabla v,\nabla g\in L^2(0,S;H^2),
\qquad
\nabla v_t,\nabla g_t\in L^2(0,S;L^2)
\]
imply
\[
\nabla v,\nabla g\in C([0,S];H^1).
\]
Since $r_t=v_t$ and $z_t=g_t$, we also have
\[
r,z\in H^1(0,S;L^2)\subset C([0,S];L^2).
\]
Combining this with
\[
\nabla r=\nabla v-\nabla v_0,
\qquad
\nabla z=\nabla g-\nabla g_0,
\]
we conclude that
\[
r,z\in C([0,S];H^2).
\]
The identity $u=u_0+r-z$ then gives
$u\in C([0,S];H^2)$.
This completes the proof.
\end{proof}

The density estimates above are obtained directly from its equation.
The factor $\rho^{\gamma-1}$ appears only in the pressure term.
No separate evolution equation or higher-order energy estimate for $\rho^{\gamma-1}$ is needed.
\subsection{Higher-order estimates when $\gamma=1$}

For the isothermal pressure, the estimates for the density and for
the spatial derivatives of $v$ and $g$ remain valid. The zeroth-order
velocity estimate must be taken in the affine class of
Theorem~\ref{thm:global-strong-isothermal}.

\begin{proposition}[Higher-order estimates when $\gamma=1$]
\label{prop:higher-order-isothermal}
Let $N=2,3$, and let $(\rho,v)$ be the positive local strong solution
with initial data satisfying \eqref{eq:isothermal-data}.
Suppose that \eqref{ho:known-bounds} holds on $[0,S]$, with its
constants independent of $S<\min\{T,T_*\}$. Set
\[
r=v-v_0+t g_0,\qquad z=g-g_0.
\]
Then
\begin{equation}\label{ho:iso-main-estimate}
\begin{aligned}
&\sup_{0\le t\le S}\Bigl(
 \|\rho(t)\|_{H^3}^2+\|r(t)\|_{H^2}^2+\|z(t)\|_{H^2}^2
 +\|u(t)+t g_0\|_{H^2}^2\\
&\quad
 +\|\nabla v(t)\|_{H^1}^2+\|\nabla g(t)\|_{H^1}^2
 \Bigr)\\
&\quad+\int_0^S\Bigl(
 \|\rho\|_{H^4}^2+\|\nabla v\|_{H^2}^2
 +\|\nabla g\|_{H^2}^2+\|\nabla u\|_{H^2}^2
 \Bigr)\,\mathrm dt\le C,
\end{aligned}
\end{equation}
and
\begin{equation}\label{ho:iso-time-estimate}
\begin{aligned}
&\sup_{0\le t\le S}\Bigl(
 \|\rho_t(t)\|_{H^1}^2+\|v_t(t)+g_0\|_{L^{2}}^2
 +\|g_t(t)\|_{L^{2}}^2+\|u_t(t)+g_0\|_{L^{2}}^2
 \Bigr)\\
&\quad+\int_0^S\Bigl(
 \|\rho_t\|_{H^2}^2+\|v_t+g_0\|_{H^1}^2
 +\|g_t\|_{H^1}^2+\|u_t+g_0\|_{H^1}^2
 \Bigr)\,\mathrm dt\le C.
\end{aligned}
\end{equation}
Here $C$ depends only on $N,T,R,V,G_\infty$ and the initial
norms in \eqref{eq:isothermal-data}, and is independent of $S$.
Moreover, all the time-continuity assertions in
\eqref{eq:isothermal-regularity} hold on $[0,S]$.
\end{proposition}

\begin{proof}
The system \eqref{ho:system} holds with $\mathfrak p=g$.
We retain the energies $\mathcal E_j,\mathcal D_j$ in
\eqref{ho:energies}. For $j=1,2$, the pressure term now satisfies
\[
\left|\lambda_0\sum_{|\alpha|=j}
 \int\partial^\alpha g\cdot\partial^\alpha v\,dx\right|
\le C\mathcal E_j.
\]
All the other terms in \eqref{ho:energy-identity} are unchanged.
Thus the same product estimates give
\begin{equation}\label{ho:iso-joint-energy}
(\mathcal E_1+\mathcal E_2)'
+\mathcal D_1+\mathcal D_2
\le C(\mathcal E_1+\mathcal E_2).
\end{equation}
The estimates \eqref{ho:density-third} and
\eqref{ho:density-zero} also remain valid.

We also need to estimate the $L^2$ norms of $r$ and $z$. Since $r_t=v_t+g_0$, $z_t=g_t$, and $g=g_0+z$, the equations for $v$ and $g$ become
\begin{equation}\label{ho:iso-shifted-equations}
\begin{aligned}
r_t=v_t+g_0
 &=\Delta v-(v-2g)\cdot\nabla v-z,\\
z_t=g_t
 &=\Delta g-\nabla\operatorname{div}v+\nabla\Phi.
\end{aligned}
\end{equation}
Writing $J=\|r\|_{L^{2}}^2$ and $Q=\|z\|_{L^{2}}^2$, we obtain
\[
\|r_t\|_{L^{2}}^2\le C(\mathcal E_1+\mathcal E_2+Q),
\qquad
\|z_t\|_{L^{2}}^2\le C(\mathcal E_1+\mathcal E_2),
\]
and consequently
\begin{equation}\label{ho:iso-zero-energy}
(J+Q)'\le C(J+Q+\mathcal E_1+\mathcal E_2).
\end{equation}
In particular, no $L^2$ norm of $g$ or $g_0$ occurs.

Take $A_*=2C_1+1$ as in Step~5 above and set
\[
\mathcal E_{\rm iso}
=1+A_*(\mathcal E_1+\mathcal E_2)
 +J+Q+\|\rho\|_{L^{2}}^2+Z_3.
\]
Adding \eqref{ho:iso-joint-energy},
\eqref{ho:iso-zero-energy}, \eqref{ho:density-third}
and \eqref{ho:density-zero}, with the same weights as before,
gives
\[
\mathcal E_{\rm iso}'+\mathcal D\le C\mathcal E_{\rm iso}.
\]
Gronwall's inequality yields the analogue of \eqref{ho:gronwall}.
The affine norms are recovered from
\[
\nabla r=\nabla v-\nabla v_0+t\nabla g_0,
\qquad
\nabla z=\nabla g-\nabla g_0,
\qquad
u+t g_0=u_0+r-z.
\]
Since $t\le T$ and $\nabla g_0,\nabla v_0\in H^1$, these
identities give the supremum bounds in
\eqref{ho:iso-main-estimate}. The integrated bounds follow
from the same dissipation and from $\nabla u=\nabla v-\nabla g$.

We next establish the time-derivative estimates. Taking a spatial derivative in \eqref{ho:iso-shifted-equations}
and using \eqref{ho:transport-product} gives
\[
\begin{aligned}
\|\nabla r_t\|_{L^{2}}
&\le C\bigl(
 \|\nabla^3v\|_{L^{2}}+\|\nabla^2v\|_{L^{2}}+\|\nabla^2g\|_{L^{2}}
 +\|\nabla g\|_{L^{2}}+\|\nabla g_0\|_{L^{2}}\bigr),\\
\|\nabla z_t\|_{L^{2}}
&\le C\bigl(
 \|\nabla^3g\|_{L^{2}}+\|\nabla^3v\|_{L^{2}}
 +\|\nabla^2g\|_{L^{2}}+\|\nabla^2v\|_{L^{2}}\bigr).
\end{aligned}
\]
These quantities are square integrable in time. Since
$u_t+g_0=r_t-z_t$, the time-derivative bounds follow exactly
as in Step~6, including the unchanged estimates for $\rho_t$.

For time continuity, observe that
\[
\partial_t\nabla v=\nabla r_t-\nabla g_0,
\qquad
\partial_t\nabla g=\nabla z_t.
\]
The Sobolev time-continuity theorem therefore gives
$\nabla v,\nabla g\in C([0,S];H^1)$, while
$r,z\in H^1(0,S;L^2)$. The preceding affine identities imply
$r,z,u+t g_0\in C([0,S];H^2)$ and
$\nabla u\in C([0,S];H^1)$. They also give
$v,g\in C([0,S];L^\infty)$ by $H^2\hookrightarrow L^\infty$.
The right-hand sides of \eqref{ho:iso-shifted-equations}
are consequently continuous into $L^2$.
Finally, $\rho\in C([0,S];H^3)$ follows as before, and
the density equation gives $\rho_t\in C([0,S];H^1)$.
This proves all the stated continuity assertions.
\end{proof}
\section{Proof of the main theorems}
\label{sec:continuation}

In this section, we prove Theorems~\ref{thm:global_strong_solution}
and \ref{thm:global-strong-isothermal} by a continuation argument.

\begin{proof}[Proof of Theorems~\ref{thm:global_strong_solution}
and \ref{thm:global-strong-isothermal}]
Let $T_*>0$ be the maximal existence time of the positive local
solution given by Lemma~\ref{lem: loc sol for gamma>1} when $\gamma>1$
and by Lemma~\ref{lem:local-isothermal} when $\gamma=1$.
Suppose, by contradiction, that $T_*<\infty$, and fix $T>T_*$.
We apply the preceding estimates on $[0,S]$, where $S<T_*$.
The constants depend on $T$ and the initial data, but not on $S$.

For $\gamma>1$, Lemma~\ref{lem:density-upper-bound}
and the effective velocity estimate of
Section~\ref{effvelocsec} give uniform bounds for $\rho$ and $v$.
For $\gamma=1$, the bound for $v$ follows from
Proposition~\ref{iso-v:proposition}.
Subsection \ref{sec: relative-density bounds independent of S} then gives the relative-density bounds with constants independent of $S$ in both cases. Together with
\eqref{en:final-log-bound}, these yield
\begin{equation}\label{cont:pointwise}
C_T^{-1}\rho_0(x)\le\rho(t,x)\le C_T\rho_0(x),
\qquad
\|\rho(t) \|_{L^{\infty}}+\|v(t) \|_{L^{\infty}}+\|g(t) \|_{L^{\infty}}
\le C_T
\end{equation}
for $0\le t<T_*$.
Propositions~\ref{higherordergamma>1}
and \ref{prop:higher-order-isothermal} give the higher-order
estimates in the two cases. Letting $S\uparrow T_*$ and using
monotone convergence theorem, we obtain these estimates on $[0,T_*)$.

\medskip
\noindent\textbf{Step 1: Limits as $t\uparrow T_*$.}

The preceding estimates imply
\[
\rho\in L^2(0,T_*;H^4),\qquad
\rho_t\in L^2(0,T_*;H^2).
\]
Hence by Lions-Magenes theorem, $\rho$ has a representative in $C([0,T_*];H^3)$.
In both cases,
\[
\nabla v,\nabla g\in L^2(0,T_*;H^2),\qquad
\partial_t\nabla v,\partial_t\nabla g
\in L^2(0,T_*;L^2).
\]
For $\gamma=1$, the estimate for $\partial_t\nabla v$ follows from
\[
\partial_t\nabla v=\nabla(v_t+g_0)-\nabla g_0,
\]
and $\nabla g_0\in L^2$.
Applying the same time-continuity theorem, we obtain
\begin{equation}\label{cont:derivative-traces}
\nabla v,\nabla g\in C([0,T_*];H^1).
\end{equation}

Set
\[
z=g-g_0,\qquad
r=\begin{cases}
v-v_0,&\gamma>1,\\
v-v_0+t g_0,&\gamma=1.
\end{cases}
\]
The estimates for $r_t$ and $z_t$ imply
$r,z\in H^1(0,T_*;L^2)$. Thus $r$ and $z$ have limits in $L^2$
as $t\uparrow T_*$. Together with
\eqref{cont:derivative-traces}, this gives
\begin{equation}\label{cont:affine-traces}
r,z\in C([0,T_*];H^2).
\end{equation}
Define
\[
\rho_*:=\rho(T_*),\qquad g_*:=g_0+z(T_*),\qquad
v_*:=\begin{cases}
v_0+r(T_*),&\gamma>1,\\
v_0-T_*g_0+r(T_*),&\gamma=1.
\end{cases}
\]
Thus $v_*,g_*\in\mathcal B$ and $\rho_*\in H^3$.
The convergence above and \eqref{cont:pointwise} give
\[
\rho_*(x)\ge C_T^{-1}\rho_0(x)>0,
\qquad g_*=\nabla\log\rho_*.
\]
The identity for $g_*$ follows by passing to the limit in
$\nabla\rho=\rho g$ using $H^3\hookrightarrow C^1$
and the uniform convergence of $g$.
When $\gamma>1$, we also have
\[
u_*:=v_*-g_*\in H^2,\qquad
\rho_*\in L^1,\qquad
\mathfrak p_*:=\gamma\rho_*^{\gamma-1}g_*\in L^2.
\]
The inequality $\rho_*\le C_T\rho_0$ gives $\rho_*\in L^1$.
The uniform $L^2$ bound for $\mathfrak p$ and Fatou's lemma give
$\mathfrak p_*\in L^2$.
In fact, \eqref{ho:pressure-decomposition} gives
$\mathfrak p(t)\to\mathfrak p_*$ strongly in $L^2$:
the first term converges by the uniform convergence of
$\rho^{\gamma-1}$ and the $L^2$ convergence of $z$, while
dominated convergence applies to the second term, using
$\rho(t)/\rho_0\le C_T$ and $\mathfrak p_0\in L^2$.
Thus $(\rho_*,u_*)$ satisfies the assumptions of
Lemma~\ref{lem: loc sol for gamma>1} when $\gamma>1$, and
$(\rho_*,v_*,g_*)$ satisfies those of
Lemma~\ref{lem:local-isothermal} when $\gamma=1$.

\medskip
\noindent\textbf{Step 2: Continuation beyond $T_*$.}

Applying the local existence result at $t=T_*$, we obtain a
positive solution on $[T_*,T_*+\tau]$ for some $\tau>0$.
It agrees with the limiting values of the original solution
at $T_*$. We can therefore join the two solutions, and the
equations hold in distributions on $(0,T_*+\tau)$.
The continuity of the time derivatives stated in the theorems
follows from the equations and the continuity of their
right-hand sides.

For $\gamma>1$, we have $v-v_*,g-g_*\in C H^2$ on the new
interval. Since $v_*-v_0,g_*-g_0\in H^2$, it follows that
$v-v_0,g-g_0,u\in C H^2$ on the extended interval.
For $\gamma=1$, we apply Lemma~\ref{lem:local-isothermal} with
initial data $v_*,g_*\in\mathcal B$. This lemma does not require
$u_*\in H^2$. Set
\[
r_*(t)=v(t)-v_*+(t-T_*)g_*,\qquad
z_*(t)=g(t)-g_*.
\]
Then $r_*,z_*\in C H^2$, and
\begin{equation}\label{cont:affine-restart}
r(t)=r_*(t)+r(T_*)-(t-T_*)z(T_*),\qquad
z(t)=z_*(t)+z(T_*).
\end{equation}
Since $r(T_*),z(T_*)\in H^2$, these identities give
$r,z\in C H^2$ on the extended interval. Moreover,
\[
v_t+g_0=(v_t+g_*)-z(T_*),
\]
which gives the required regularity of $v_t+g_0$.
The other time-derivative estimates follow in the same way.
The local comparison of $\rho$ with $\rho_*$, together with
\eqref{cont:pointwise}, gives the comparison with $\rho_0$
on the extended interval.
Thus the solution extends beyond $T_*$ with the regularity
stated in the corresponding theorem. This contradicts the
maximality of $T_*$, and hence $T_*=\infty$.

\medskip
\noindent\textbf{Step 3: Uniqueness and far-field behavior.}

For any finite $T>0$, the preceding estimates give the pointwise
bounds and regularity stated in the two theorems. Applying
local uniqueness successively on short time intervals gives
uniqueness on $[0,T]$.

It remains to verify the uniformity of the far-field limits.
Since $H^2(\mathbb R^N)\hookrightarrow C_0(\mathbb R^N)$
for $N=2,3$, the set $\{f(t):0\le t\le T\}$ is compact in
$C_0$ whenever $f\in C([0,T];H^2)$. Covering this set by
finitely many balls in the supremum norm gives
\[
\lim_{R\to\infty}\sup_{\substack{0\le t\le T\\ |x|\ge R}}
|f(t,x)|=0
\qquad\text{for }f\in C([0,T];H^2).
\]
This applies to $\rho\in C([0,T];H^3)$. And when $\gamma>1$, the same conclusion holds for $u\in C([0,T];H^2)$. When $\gamma=1$, we have $u_0\in H^2\subset C_0$ and $g_0\in C_0$ by \eqref{eq:g0-farfield}, and hence $v_0=u_0+g_0\in C_0$. Also, $r=v-v_0+tg_0$ and $z=g-g_0$ belong to $C([0,T];H^2)$, so both tend to zero at infinity uniformly in $t$. The same is true of $tg_0$, since $\sup_{0\le t\le T}|tg_0(x)|\le T|g_0(x)|.$ 
 Finally,
\[
g=g_0+z,\qquad v=v_0-tg_0+r,
\qquad u=-tg_0+(u_0+r-z),
\]
 Each term on the right tends to zero at infinity uniformly in $t\in[0,T]$. Thus $u,v,g$ have the asserted uniform far-field limits. 
The proof is complete.
\end{proof}
\section{Appendix} \label{Appendix7}

\subsection{Auxiliary lemmas}
\begin{lemma}[Product Estimates]\label{lem: Product Estimates}
    Let $N=2,3$, and let $\mathcal{B}$ be defined by \eqref{space: BUC and nabla H1}. Then there exists a constant $C>0$ such that:\\
    (i) If $a$ is a scalar function with $a\in\mathcal B$, and $f\in H^2$ is either a scalar function or a vector field, then
\begin{equation}\label{est: product est af 1}
    \|af\|_{H^2}\le C\|a\|_{\mathcal B}\|f\|_{H^2}.
\end{equation}
Therefore, if $a$ and $f$ are both vector fields belonging to $\mathcal B$ and $H^2$, respectively, applying \eqref{est: product est af 1} component-wise yields
\begin{equation}\label{est: product est af 2}
    \|a\cdot f\|_{H^2}\le C\|a\|_{\mathcal B}\|f\|_{H^2}. 
\end{equation}
    (ii) If $a,b\in\mathcal{B}$ are both vector fields, then 
    \begin{equation}
        \|(a\cdot\nabla )b\|_{H^1}+\|\nabla(a\cdot b)\|_{H^1}\le C\|a\|_{\mathcal{B}}\|b\|_{\mathcal{B}}.
    \end{equation}
    (iii) Let $F(v,g):=|g|^2-v\cdot g$. If $v_1,v_2,g_1,g_2\in\mathcal{B}$ and $$\|v_i\|_{\mathcal B}+\|g_i\|_{\mathcal B}\le M \quad(i=1,2), \quad v_1-v_2,\ g_1-g_2\in H^2,$$ then 
    \begin{equation}
        \|F(v_1,g_1)-F(v_2,g_2)\|_{H^2}\le C_M(\|v_1-v_2\|_{H^2}+\|g_1-g_2\|_{H^2}).
    \end{equation}
(iv) If $R\in H^3$ is a scalar function, $a\in L^\infty$ is a vector field with $\nabla a\in H^2$, then
\begin{equation}
  \|aR\|_{H^3}\le C(\|a \|_{L^{\infty}}+\|\nabla a\|_{H^2})\|R\|_{H^3}.  
\end{equation}    
\end{lemma}

We next state a linear parabolic lemma that will be used in the
construction of local solutions. The existence, uniqueness, and
regularity results below follow from basic energy estimates,
Gronwall's inequality, and standard approximation arguments.
We omit the proof.
\begin{lemma}[Linear Parabolic System]\label{lem: sol of linear eq}
    Let $T\le1$, $N=2,3$. Define
    \begin{equation}
        X_T:=\bigl\{Z=(z,w):\ Z\in C([0,T];H^2)\cap L^2(0,T;H^3),\ Z(0)=0\bigr\},
    \end{equation}
    equipped with the norm
   \begin{equation} \begin{aligned} \|(z,w)\|_{{X}_T}^2 &= \sup_{0\le t\le T}\big(4N^2\|z\|_{H^2}^2+\|w\|_{H^2}^2\big) + \int_0^T \big(4N^2\|\nabla z\|_{H^2}^2+\|\nabla w\|_{H^2}^2\big)\,dt. \end{aligned} \end{equation}
   Also define ${Y}_T=C([0,T];H^3)\cap L^2(0,T;H^4),$ equipped with the norm $$\|R\|_{Y_T}^2:=\sup_{0\le t\le T}\|R(t)\|_{H^3}^2+\int_0^T\|\nabla R\|_{H^3}^2\,dt.$$ Then:\\
   (i) For any $(f,k)\in L^2(0,T;H^1),$ \begin{equation}\label{eq:correction-linear-system}
       \begin{cases}
       z_t-\Delta z=f\\
       w_t-\Delta w+\nabla\mathrm{div}z=k\\
       (z,w)(0)=0
   \end{cases}
   \end{equation}
   has a unique solution $(z,w)\in X_T$, and 
   \begin{equation}
       \|(z,w)\|_{X_T}\le C\|(f,k)\|_{L^2(0,T;H^1)},
   \end{equation}
   where $C>0$ depends only on $N$.\\
   (ii) Let $a\in L^2(0,T;L^\infty)$, $\nabla a\in L^2(0,T;H^2)$, $b,f\in L^2(0,T;H^2)$, $J\in L^2(0,T;H^3)$, $R_0\in H^3$. Then\begin{equation}
       R_t-\Delta R+\mathrm{div}(aR)=bR+f-\mathrm{div}J,\qquad R(0)=R_0
   \end{equation}
   has a unique solution $R\in Y_T$, and
 \begin{equation}
     \|R\|_{Y_T}^2\le C\exp\Bigl(C\int_0^T\left(1+\|a \|_{L^{\infty}}^2+\|\nabla a\|_{H^2}^2+\|b\|_{H^2}^2\right)\,dt\Bigr)\Bigl(\|R_0\|_{H^3}^2+\|f\|_{L^2H^2}^2+\|J\|_{L^2H^3}^2\Bigr).
 \end{equation}
\end{lemma}
\begin{remark}\label{lem: sol of zero source terms}
    The above conclusion also applies to nonzero $H^m$ initial data $(m=1,2)$. Specifically, let $(z_0,w_0)\in H^m(\mathbb{R}^N)\times H^m(\mathbb{R}^N)$, $m=1,2$. Consider 
    \begin{equation}
        \begin{cases}
z_t-\Delta z=0,\\
w_t-\Delta w+\nabla\operatorname{div}z=0,\\
(z,w)(0)=(z_0,w_0).
\end{cases}
    \end{equation}
    Then there exists a unique solution $(z,w)$ with $z,w\in C([0,T];H^m)\cap L^2(0,T;H^{m+1})$, and
\begin{equation}
    \sup_{t\le T}\bigl(\|z(t)\|_{H^m}^2+\|w(t)\|_{H^m}^2\bigr)
+\int_0^T\bigl(\|\nabla z\|_{H^m}^2+\|\nabla w\|_{H^m}^2\bigr)\,dt
\le C\bigl(\|z_0\|_{H^m}^2+\|w_0\|_{H^m}^2\bigr),
\end{equation}
where
$$
\|(z,w)\|_{H^m}^2:=\|z\|_{H^m}^2+\|w\|_{H^m}^2,
$$
and $C>0$ depends only on $N$ and is independent of $T.$

In addition, if a sequence of initial data $(z_0^{(n)},w_0^{(n)})$ converges in $H^m$ to $(z_0,w_0)$, then the corresponding solutions $(z^{(n)},w^{(n)})$ converge in $C([0,T];H^m)$ to $(z,w).$ 
\end{remark}

\subsection{Proof of Lemma \ref{lem: loc sol for gamma>1}}\label{subsec: proof of loc sol}

\begin{proof}
\textbf{Step 1: Rewrite the System}

For $\rho>0$, define
$$
v:=u+\nabla\log\rho,\quad g:=\nabla\log\rho,\quad h:=\log\rho,\quad \mathfrak{p}:=\gamma\rho^{\gamma-1}g,
$$
such that 
\begin{equation}\label{rhovghF}
    \begin{cases}
\rho_t-\Delta\rho=-\operatorname{div}(\rho v),\\
v_t-\Delta v=-(v-2g)\cdot\nabla v-\mathfrak{p},\\
g_t-\Delta g+\nabla\operatorname{div}v=\nabla F,\qquad F:=|g|^2-v\cdot g,\\
h_t-\Delta h=F-\operatorname{div}v,\qquad \nabla h=g.
\end{cases}
\end{equation}

\textbf{Step 2: Construction of reference fields}

For $\nabla v_0,\nabla g_0\in H^1$, consider the derivative system
\begin{equation}
    \begin{cases}
(A_k)_t-\Delta A_k=0,\\
(B_k)_t-\Delta B_k+\nabla\operatorname{div}A_k=0,\\
A_k(0)=\partial_k v_0,\\
B_k(0)=\partial_k g_0.
\end{cases}
\end{equation}
By Remark \ref{lem: sol of zero source terms},  there exists a solution $(A_k,B_k)\in C([0,T];H^1)\cap L^2(0,T;H^2)$ such that
\begin{equation}\label{est: AkBk}
    \sum_{k=1}^{N}\Bigl[\sup_{0\le t\le T}\bigl(\|A_k\|_{H^1}^2+\|B_k\|_{H^1}^2\bigr)
+\int_0^T\bigl(\|\nabla A_k\|_{H^1}^2+\|\nabla B_k\|_{H^1}^2\bigr)\,dt\Bigr]
\le C\bigl(\|\nabla v_0\|_{H^1}^2+\|\nabla g_0\|_{H^1}^2\bigr).
\end{equation}
One can verify that
\begin{equation}\label{eq: curl-free of AB}
    \partial_l A_k=\partial_k A_l,\qquad \partial_l B_k=\partial_k B_l.
\end{equation}
Indeed, for
$$
D_{kl}:=\partial_l A_k-\partial_k A_l,E_{kl}:= \partial_{l}B_{k}-\partial_{k}B_{l}.
$$
we have
$$
\begin{cases}
(D_{kl})_t-\Delta D_{kl}=0\\
D_{kl}(0)=\partial_{lk}v_0-\partial_{kl}v_0=0,
\end{cases}
\implies D_{kl}=0,
$$
and $$\begin{cases}
(E_{kl})_{t}-\Delta E_{kl}=-\nabla \text{div}\left(D_{kl}\right)=0 \\
E_{kl}(0)=0
\end{cases}\implies E_{kl}=0.$$ Recover $\bar v,\bar g$ by time integration. Let 
\begin{equation}
    \begin{aligned}
\bar v(t)&:=v_0+\int_0^t\sum_{k=1}^{N}\partial_k A_k(s)\,ds,\\
\bar g(t)&:=g_0+\int_0^t\Bigl[\sum_{k=1}^{N}\partial_k B_k(s)
-\nabla\Bigl(\sum_{j=1}^{N}(A_j)_j(s)\Bigr)\Bigr]\,ds.
\end{aligned}
\end{equation}
By \eqref{eq: curl-free of AB}, we have
\begin{equation}
    \begin{aligned}
\partial_l\bar v
&=\partial_l v_0+\int_0^t\sum_{k=1}^{N}\partial_l\partial_k A_k(s)\,ds\\
&=\partial_l v_0+\int_0^t\sum_{k=1}^{N}\partial_k\partial_k A_l(s)\,ds\\
&=\partial_l v_0+\int_0^t\Delta A_l(s)\,ds\\
&=\partial_l v_0+\int_0^t\partial_s A_l(s)\,ds\\
&=A_l(t),
\end{aligned}
\end{equation}
and
\begin{equation}
    \begin{aligned}
\partial_l\bar g
&=\partial_l g_0+\int_0^t\left[\sum_{k=1}^{N}\partial_l\partial_k B_k(s)
-\nabla\Bigl(\sum_{j=1}^{N}\partial_l(A_j)_j(s)\Bigr)\right]ds\\
&=\partial_l g_0+\int_0^t\left[\Delta B_l(s)
-\nabla\Bigl(\sum_{j=1}^{N}\partial_j(A_l)_j(s)\Bigr)\right]ds\\
&=\partial_l g_0+\int_0^t\left[\Delta B_l(s)-\nabla\operatorname{div}A_l(s)\right]ds\\
&=\partial_l g_0+\int_0^t\partial_s B_l(s)\,ds\\
&=B_l(t).
\end{aligned}
\end{equation}
Therefore,
\begin{equation}
    \begin{cases}
\bar v_t-\Delta\bar v=0,\\
\bar g_t-\Delta\bar g+\nabla\operatorname{div}\bar v=0,\\
(\bar v,\bar g)(0)=(v_0,g_0).
\end{cases}
\end{equation}

Combining \eqref{est: AkBk}, we have $$\begin{aligned}
\|\bar v(t)-v_0\|_{L^2}
&\le\int_0^t\sum_k\|\partial_kA_k(s)\|_{L^2}\,ds\\
&\le t\sum_k\|A_k\|_{C([0,T];H^1)}
\le Ct\bigl(\|\nabla v_0\|_{H^1}+\|\nabla g_0\|_{H^1}\bigr),
\end{aligned}$$
and 
$$\begin{aligned}
\|\bar g(t)-g_0\|_{L^2}
&\le\int_0^t\left(
\sum_k\|\partial_kB_k(s)\|_{L^2}
+\sum_j\|\nabla(A_j)_j(s)\|_{L^2}\right)ds\\
&\le t\sum_k\bigl(\|A_k\|_{C([0,T];H^1)}+\|B_k\|_{C([0,T];H^1)}\bigr)\\
&\le Ct\bigl(\|\nabla v_0\|_{H^1}+\|\nabla g_0\|_{H^1}\bigr).
\end{aligned}$$
Since $\bar v-v_0,$ $\bar g-g_0$ are indefinite integrals with values in $L^2$, we have
$\bar v-v_0,\bar g-g_0\in C([0,T];L^2)$. We also proved $\partial_j(\bar v-v_0)=A_j-A_j(0)\in C([0,T];H^1),$ and $\partial_j(\bar g-g_0)=B_j-B_j(0)\in C([0,T];H^1), $ which imply  
$$\bar v-v_0,\ \bar g-g_0\in C([0,T];H^2).$$
Since $\nabla g_0\in H^1(\mathbb R^N)\hookrightarrow L^6(\mathbb R^N)$ for $N=2,3$, and $g_0\in L^\infty$, the $W^{1,6}(B_1(y))$ norms of $g_0$ are bounded uniformly in $y$. The local Morrey inequality therefore gives a uniformly continuous representative of $g_0$. Hence $g_0\in\mathcal B$. Since $u_0\in H^2\hookrightarrow\mathcal B$, we also have $v_0=u_0+g_0\in\mathcal B$. Combining these facts with $H^2\hookrightarrow\mathcal B$, we obtain
\begin{equation}
    \bar v=v_0+(\bar v-v_0),
\bar g=g_0+(\bar g-g_0)
\in C([0,T];\mathcal{B}). 
\end{equation}
At the same time, $\partial_j\bar v=A_j$, $\partial_j\bar g=B_j$. By \eqref{est: AkBk}, $A_j,B_j\in L^2(0,T;H^2)$, so $D\bar v,D\bar g\in L^2(0,T;H^2).$ Finally, from the definitions
$$
\bar v_t=\sum_k\partial_kA_k,
\bar g_t=\sum_k\partial_kB_k-\nabla\sum_j(A_j)_j,
$$
and $A_k,B_k\in C([0,T];H^1)\cap L^2(0,T;H^2)$, we obtain
$$
\bar v_t,\bar g_t\in C([0,T];L^2)\cap L^2(0,T;H^1). 
$$
Combining these bounds yields 
\begin{equation}\label{est: reference field}
    \begin{aligned}
&\|\bar{v}\|_{C([0,T];\mathcal{B})}
+\|\bar{g}\|_{C([0,T];\mathcal{B})}
+\|\nabla\bar{v}\|_{L^2(0,T;H^2)}
+\|\nabla\bar{g}\|_{L^2(0,T;H^2)}\\
+&\|\bar{v}-v_0\|_{C([0,T];H^2)}
+\|\bar{g}-g_0\|_{C([0,T];H^2)}
\le C\bigl(\|v_0\|_{\mathcal{B}}+\|g_0\|_{\mathcal{B}}\bigr).
\end{aligned}
\end{equation}

In addition, by the same reasoning as \eqref{eq: curl-free of AB}, one can see that $\bar g$ is curl-free. Define
\begin{equation}
    \bar h(t):=h_0+\int_0^t(\operatorname{div}\bar g-\operatorname{div}\bar v)\,ds,\qquad h_0:=\log\rho_0.
\end{equation}
Then the compatibility condition is satisfied:
\begin{equation}
    \begin{aligned}
\nabla\bar h
&=\nabla h_0+\int_0^t\nabla\operatorname{div}\bar g-\nabla\operatorname{div}\bar v\,ds\\
&=g_0+\int_0^t\Delta\bar g-\nabla\operatorname{div}\bar v\,ds\\
&=g_0+\int_0^t\partial_s\bar g(s)\,ds=\bar g(t).
\end{aligned}
\end{equation}
Thus,\begin{equation}
    \bar h_t-\Delta\bar h
=\operatorname{div}\bar g-\operatorname{div}\bar v
-\operatorname{div}\bar g
=-\operatorname{div}\bar v.
\end{equation}
We also have the regularity estimate
\begin{equation}
    \begin{aligned}
\|\bar h(t)-h_0\|_{L^\infty}
&{\le} C\|\bar h(t)-h_0\|_{H^2}\\
&\le C\int_0^t\bigl(\|D\bar g(s)\|_{H^2}+\|D\bar v(s)\|_{H^2}\bigr)\,ds\\
&{\le}
C\|D\bar g\|_{L^2(0,T;H^2)}\sqrt{t}
+C\|D\bar v\|_{L^2(0,T;H^2)}\sqrt{t}
\le B_0\sqrt{t},
\end{aligned}
\end{equation}
and the time continuity
\begin{equation}
    \bar h-h_0\in C([0,T];BUC),\qquad \nabla\bar h=\bar g\in C([0,T];\mathcal{B}).
\end{equation}
\textbf{Step 3: Solving the perturbation quantities via the fixed point theorem}

\textbf{Step 3.1: The fixed-point space and preliminary estimates.}

Consider the subset of $X_T\times C([0,T];BUC)$ defined by
\begin{equation}\label{def: perturbed field}
    \mathcal{K}_T=\bigl\{(z,w,q):\ q(0)=0,\ \nabla q=w,\ \|(z,w)\|_{X_T}+\|q\|_{C([0,T];BUC)}\le1\bigr\}.
\end{equation}Here $\nabla q = \boldsymbol{w}$ holds in the sense of distributions. This is a closed subset: suppose $(z_{n},w_{n},q_{n})\in \mathcal{K}_{T}$ converges to $(z,w,q)$. The norm bound and the condition $q(0)=0$ are obviously preserved under taking limits. For every $t$ and every $\varphi\in C_c^\infty(\mathbb{R}^N)$, since $\nabla q_n=w_n$, $q_n\to q$ in $C([0,T];BUC)$, and $w_n\to w$ in $C([0,T];H^2)$, 
\begin{equation}
    \begin{aligned}
\langle\partial_j q(t),\varphi\rangle
&=-\int q(t)\,\partial_j\varphi=-\lim_{n\to\infty}\int q_n(t)\,\partial_j\varphi
=\lim_{n\to\infty}\int (w_n)_j(t)\varphi
=\int w_j(t)\varphi.
\end{aligned}
\end{equation}
That is, $\nabla q(t)=w(t),$ $\forall t\in[0,T]$ in $\mathcal{D}'(\mathbb{R}^N).$ 
Take any $(z,w,q)\in \mathcal{K}_T$ and set
\begin{equation}
    v=\bar v+z,\quad g=\bar g+w,\quad h=\bar h+q,\quad \rho=e^h,
\end{equation}
so that
\begin{equation}
    \rho= e^{\bar h}e^{q}= e^{h_{0}}e^{\bar h-h_{0}}e^{q}\in[e^{-(B_{0}+1)}\rho_{0},e^{B_{0}+1}\rho_{0}],
\qquad
\|\rho\|_{L^\infty}\leq e^{B_{0}+1}\|\rho_{0}\|_{L^\infty}.
\end{equation}
Since $\bar v\in C([0,T];\mathcal{B})$ and $z\in X_T\subset C([0,T];H^2)\subset C([0,T];\mathcal{B})$, we have
\begin{equation}
    \|v\|_{C([0,T];\mathcal{B})}
\le \|\bar v\|_{C([0,T];\mathcal{B})}+\|z\|_{C([0,T];\mathcal{B})}
\le C(\|v_0\|_{\mathcal{B}}+\|g_0\|_{\mathcal{B}})+C.
\end{equation}
Similarly, since $\bar g\in C([0,T];\mathcal{B})$ and $w\in X_T\subset C([0,T];H^2)\subset C([0,T];\mathcal{B})$, we have
\begin{equation}
    \|g\|_{C([0,T];\mathcal{B})}
\le \|\bar g\|_{C([0,T];\mathcal{B})}+\|w\|_{C([0,T];\mathcal{B})}
\le C(\|v_0\|_{\mathcal{B}}+\|g_0\|_{\mathcal{B}})+C.
\end{equation}
For $\nabla v$, since $\nabla v=\nabla\bar v+\nabla z$, with $\nabla\bar v\in L^2(0,T;H^2)$ and $\nabla z\in L^2(0,T;H^2),$ we have $\nabla v\in L^2(0,T;H^2)$. Similarly, $\nabla g\in L^2(0,T;H^2)$. For $g-g_0$, since $g-g_0=(\bar g-g_0)+w$, with $\bar g-g_0\in C([0,T];H^2)\subset C([0,T];L^2)$ and $w\in C([0,T];H^2),$ we have $g-g_0\in C([0,T];H^2)$. Therefore,  
\begin{equation}\label{est: regular of vg}
    \|v\|_{C([0,T];\mathcal{B})}
+\|g\|_{C([0,T];\mathcal{B})}
+\|\nabla v\|_{L^2(0,T;H^2)}
+\|\nabla g\|_{L^2(0,T;H^2)}
+\|g-g_0\|_{C([0,T];H^2)}
\le C,
\end{equation}where $C$ depends only on the initial data.

We have $h-h_0=(\bar h-h_0)+q\in C([0,T];BUC)$. Since $\|h-h_0\|_{L^\infty}\le C$ and the exponential function is Lipschitz on bounded balls in $BUC$, it follows that $e^{h-h_0}\in C([0,T];BUC)$.  
Also, $\rho=e^h=e^{h_0}e^{h-h_0}=\rho_0 e^{h-h_0}$. Since $\rho_0\in H^3\hookrightarrow L^2\cap BUC$, we obtain $\rho\in C([0,T];L^2\cap BUC)$.  
Moreover, $\nabla\rho=\rho g=\rho g_0+\rho(g-g_0)$. Because $\rho\in C([0,T];L^2\cap BUC)$, $g_0\in L^\infty$, and $g-g_0\in C([0,T];H^2)\hookrightarrow C([0,T];BUC)\to C([0,T];L^\infty)$, it follows that $\nabla\rho\in C([0,T];L^2)$. We next estimate the uniform $H^1$ bound of the pressure. We have
$$
\mathfrak{p}= \gamma \rho^{\gamma-1}g
= \gamma \rho^{\gamma-1}(g-g_0)+\gamma \rho^{\gamma-1}g_0
= \gamma \rho^{\gamma-1}(g-g_0)+\left(\frac{\rho}{\rho_0}\right)^{\gamma-1}\mathfrak{p}_0,
$$
and
$$
\partial_j \mathfrak{p}_i
= \gamma \rho^{\gamma-1} \partial_j g_i+(\gamma-1)\gamma \rho^{\gamma-2}\partial_j\rho\, g_i
= \gamma \rho^{\gamma-1} \partial_j g_i+(\gamma-1)g_j \mathfrak{p}_i.
$$
Therefore,
\begin{equation}\label{est: regular-pressure-CL2}
    \|\mathfrak{p}\|_{C([0,T];L^2)}
\le \gamma\left(e^{B_0+1}\|\rho_0\|_{L^\infty}\right)^{\gamma-1}
\|g-g_0\|_{C([0,T];L^2)}
+e^{(B_0+1)(\gamma-1)}\|\mathfrak{p}_0\|_{L^2}
\le C,
\end{equation}
where $C$ depends only on the initial data. Also,
\begin{equation}\label{est: regular-pressure-CH1}
    \|\nabla \mathfrak{p}\|_{C([0,T];L^2)}
\le \gamma\|\rho\|_{C([0,T];L^\infty)}^{\gamma-1}
\|\nabla g\|_{C([0,T];L^2)}
+(\gamma-1)\|g\|_{C([0,T];L^\infty)}
\|\mathfrak{p}\|_{C([0,T];L^2)}
\le C,
\end{equation}
where $C$ depends only on the initial data. Furthermore,
\begin{equation}\label{est: regular-pressure-CLinfty}
    \|\mathfrak{p}\|_{C([0,T];L^\infty)}
\le \gamma\|\rho\|_{C([0,T];L^\infty)}^{\gamma-1}
\|g\|_{C([0,T];L^\infty)}
\le \gamma\|\rho\|_{C([0,T];L^\infty)}^{\gamma-1}
\|g\|_{C([0,T];\mathcal{B})}
\le C,
\end{equation}
where $C$ depends only on the initial data. Hence, it holds that $\mathfrak{p}\in C([0,T];H^1)\cap C([0,T];L^\infty).$

We also need to close the pressure difference. We have
\begin{equation}
    \begin{aligned}
\mathfrak{p}_1-\mathfrak{p}_2
&= \gamma\rho_1^{\gamma-1} g_1-\gamma\rho_2^{\gamma-1} g_2\\
&= \gamma\rho_1^{\gamma-1}(g_1-g_2)+\gamma(\rho_1^{\gamma-1}-\rho_2^{\gamma-1})g_2\\
&= \gamma\rho_1^{\gamma-1}(g_1-g_2)
+\gamma\bigl(e^{(\gamma-1)h_1}-e^{(\gamma-1)h_2}\bigr)g_2\\
&= \gamma\rho_1^{\gamma-1}(g_1-g_2)
+\gamma e^{(\gamma-1)h_2}\bigl(e^{(\gamma-1)(h_1-h_2)}-1\bigr)g_2\\
&= \gamma\rho_1^{\gamma-1}(g_1-g_2)
+\gamma\rho_2^{\gamma-1}\bigl(e^{(\gamma-1)(q_1-q_2)}-1\bigr)g_2\\
&= \gamma\rho_1^{\gamma-1}(g_1-g_2)
+\bigl(e^{(\gamma-1)(q_1-q_2)}-1\bigr)\mathfrak{p}_2.
\end{aligned}
\end{equation}
Since $\nabla\delta q=\delta g$ ($\delta(\cdot):= (\cdot)_{1}-(\cdot)_2$), we obtain
\begin{equation}
    \begin{aligned}
\nabla(\mathfrak{p}_1-\mathfrak{p}_2)
&= \gamma\nabla(\rho_1^{\gamma-1}\delta g)
+\nabla\bigl((e^{(\gamma-1)\delta q}-1)\mathfrak{p}_2\bigr)\\
&= \gamma(\gamma-1)\rho_1^{\gamma-1} g_1\delta g
+\gamma\rho_1^{\gamma-1}\nabla\delta g
+(\gamma-1)e^{(\gamma-1)\delta q}\mathfrak{p}_2\nabla\delta q
+(e^{(\gamma-1)\delta q}-1)\nabla\mathfrak{p}_2.
\end{aligned}
\end{equation}
Thus,
\begin{equation}\label{est: pressure L2}
    \|\mathfrak{p}_1-\mathfrak{p}_2\|_{L^2}
\le \gamma\|\rho_1\|_{L^\infty}^{\gamma-1}\|g_1-g_2\|_{L^2}
+\|e^{(\gamma-1)(q_1-q_2)}-1\|_{L^\infty}\|\mathfrak{p}_2\|_{L^2}
\le C\|g_1-g_2\|_{L^2}+C\|q_1-q_2\|_{L^\infty},
\end{equation}
and
\begin{equation}\label{est: pressure H1}
    \|\nabla(\mathfrak{p}_1-\mathfrak{p}_2)\|_{L^2}
\le  C\|g_1-g_2\|_{H^1}+C\|q_1-q_2\|_{L^\infty}.
\end{equation}
\textbf{Step 3.2: Construction of a contraction self-map}
Let
$$
\Phi:\mathcal{K}_T\to X_T\times C([0,T];BUC),\quad (z,w,q)\mapsto(z',w',q'),
$$
where $(z',w')\in X_T$ satisfies
$$
\begin{cases}
z'_t-\Delta z'=f:=-(v-2g)\cdot\nabla v-\mathfrak{p},\\
w'_t-\Delta w'+\nabla\operatorname{div}z'=\nabla F,\quad F:=|g|^2-v\cdot g,\\
(z',w')(0)=0.
\end{cases}
$$
The existence of a solution $(z',w')\in X_T$ is guaranteed by Lemma \ref{lem: sol of linear eq}. Let
$$
q'(t):=\int_0^t\bigl(\operatorname{div}w'(s)-\operatorname{div}z'(s)+F(s)\bigr)\,ds.
$$
Since $g,v\in C([0,T];\mathcal{B})\subset C([0,T];BUC)$, we have $F=|g|^2-v\cdot g\in C([0,T];BUC)$. Also, since $w',z'\in X_T$, we have $\operatorname{div}w',\operatorname{div}z'\in L^2(0,T;H^2)\subset L^1(0,T;BUC)$. Hence $q'\in C([0,T];BUC)$ and $q'(0)=0$.

Next, we verify that $\Phi$ is a self-map on the closed ball $\mathcal{K}_T$.

(1) $q'(0)=0$ and $\nabla q'=w'$. Since $\Delta w'-\nabla\operatorname{div}z'+\nabla F=w'_t$, we have
$$
\begin{aligned}
\nabla q'(t)
&= \int_0^t\bigl(\Delta w'-\nabla\operatorname{div}z'+\nabla F\bigr)\,ds\\
&= \int_0^t w'_t(s)\,ds
= w'(t)-w'(0)=w'(t).
\end{aligned}
$$
Here we use the fact that 
$$\begin{cases} \partial_t(\partial_iw'_j-\partial_jw'_i) -\Delta(\partial_iw'_j-\partial_jw'_i)=0\\ (\partial_iw'_j-\partial_jw'_i)(0)=0 \end{cases}\implies \partial_iw'_j=\partial_jw'_i, \qquad \nabla\operatorname{div}w'=\Delta w'.$$
(2) $\|(z',w')\|_{X_T}\le C\sqrt{T}$, $\|q'\|_{C([0,T];BUC)}\le CT$. Thus, when $C(\sqrt{T}+T)\le 1$, we have $\Phi(\mathcal{K}_T)\subset \mathcal{K}_T$.

By Lemma \ref{lem: Product Estimates}, one has 
\begin{equation}
    \begin{aligned}
\|f\|_{C([0,T];H^1)}
&\le \|(v-2g)\cdot\nabla v\|_{C([0,T];H^1)}+\|\mathfrak{p}\|_{C([0,T];H^1)}\\
&\le C(\|v\|_{C([0,T];\mathcal{B})}+\|g\|_{C([0,T];\mathcal{B})})\|v\|_{C([0,T];\mathcal{B})}+\|\mathfrak{p}\|_{C([0,T];H^1)}\le C,
\end{aligned}
\end{equation}
\begin{equation}
    \|F\|_{C([0,T];L^\infty)}
\le C(\|g\|_{C([0,T];\mathcal{B})}^2+\|v\|_{C([0,T];\mathcal{B})}\|g\|_{C([0,T];\mathcal{B})})\le C,
\end{equation}
and 
\begin{equation}
    \begin{aligned}
\|\nabla F\|_{C([0,T];H^1)}
&\le C(\|\nabla(g\cdot g)\|_{C([0,T];H^1)}+\|\nabla(v\cdot g)\|_{C([0,T];H^1)})\\
&\le C(\|g\|_{C([0,T];\mathcal{B})}^2+\|v\|_{C([0,T];\mathcal{B})}\|g\|_{C([0,T];\mathcal{B})})\\
&\le C.
\end{aligned}
\end{equation}
Therefore, by Lemma \ref{lem: sol of linear eq}, it holds that 
\begin{equation}
    \|(z',w')\|_{X_T}\le C\|(f,\nabla F)\|_{L^2(0,T;H^1)}
\le C\sqrt{T}(\|f\|_{C([0,T];H^1)}+\|\nabla F\|_{C([0,T];H^1)})
\le C\sqrt{T},
\end{equation}
and
\begin{equation}
    \begin{aligned}
\|q'\|_{C([0,T];L^\infty)}
&\le T\|F\|_{C([0,T];L^\infty)}
+C\int_0^T(\|z'\|_{H^3}+\|w'\|_{H^3})\,ds\\
&\le CT+C\sqrt{T}\|(z',w')\|_{X_T}\le CT.
\end{aligned}
\end{equation}
Next, we verify that $\Phi$ is a contraction. We need to show that
$$
\|\Phi(z_1,w_1,q_1)-\Phi(z_2,w_2,q_2)\|
\le\frac12\|(z_1,w_1,q_1)-(z_2,w_2,q_2)\|.
$$
Let
$$
d:
=\|(z_1-z_2,w_1-w_2)\|_{X_T}+\|q_1-q_2\|_{C([0,T];BUC)}.
$$
By Lemma \ref{lem: Product Estimates} and \ref{lem: sol of linear eq}, 
\begin{equation}
    \begin{aligned}
\|(z'_1-z'_2,w'_1-w'_2)\|_{X_T}
&\le C\|(f_1-f_2,\nabla(F_1-F_2))\|_{L^2(0,T;H^1)}\\
&\le C\sqrt{T}\bigl(\|f_1-f_2\|_{C([0,T];H^1)}+\|\nabla(F_1-F_2)\|_{C([0,T];H^1)}\bigr)\\
&\le C\sqrt{T}d,
\end{aligned}
\end{equation}
and
\begin{equation}
    \begin{aligned}
\|q_1'-q_2'\|_{C([0,T];L^\infty)}
&\le \int_0^T\|\nabla(z'_1-z'_2)(s)\|_{L^\infty}+\|\nabla(w'_1-w'_2)(s)\|_{L^\infty}\,ds
+T\|F_1-F_2\|_{C([0,T];L^\infty)}\\
&\le C\sqrt{T}\|(z'_1-z'_2,w'_1-w'_2)\|_{X_T}
+T\|F_1-F_2\|_{C([0,T];L^\infty)}\\
&\le CTd.
\end{aligned}
\end{equation}
Therefore,
\begin{equation}
    \begin{aligned}
\|\Phi(z_1,w_1,q_1)-\Phi(z_2,w_2,q_2)\|_{X_T\times C([0,T];BUC)}
&=\|(z'_1-z'_2,w'_1-w'_2)\|_{X_T}+\|q'_1-q'_2\|_{C([0,T];BUC)}\\
&\le C\sqrt{T}d+CTd
=C_*(\sqrt{T}+T)d\\
&\le \frac12\|(z_1,w_1,q_1)-(z_2,w_2,q_2)\|.
\end{aligned}
\end{equation}
By the Banach fixed point theorem, $\Phi$ has a unique fixed point $(z,w,q)$ in $\mathcal{K}_T$.

\textbf{Step 4: Density equation and higher-order regularity}

Let $(z_{0},w_{0},q_{0})=(0,0,0)$ and $(z_{n},w_{n},q_{n})=\Phi^{n}(z_{0},w_{0},q_{0})$. For $n\ge 1,$ set
\begin{equation}
    v_n=\bar v+z_n,\quad g_n=\bar g+w_n,\quad h_n=\bar h+q_n,\quad \rho_n=e^{h_n}.
\end{equation}
By the definition of $\Phi$, $(z_{n+1},w_{n+1},q_{n+1})$ satisfies
\begin{equation}\label{eq: aux-iter}
    \begin{cases}
(v_{n+1})_{t}-\Delta v_{n+1}=f_{n},\\
(g_{n+1})_{t}-\Delta g_{n+1}+\nabla\operatorname{div}v_{n+1}=\nabla F_n,\\
(q_{n+1})_{t}-\Delta q_{n+1}=F_n-\operatorname{div}z_{n+1}.
\end{cases}
\end{equation}
Here\begin{equation}
    f_{n}=-(v_n-2g_n)\cdot\nabla v_n-\mathfrak{p}_n,\quad
F_n=|g_n|^2-v_n\cdot g_n,\quad
\mathfrak{p}_{n}=\gamma \rho_{n}^{\gamma-1}g_{n}.
\end{equation}
Thus $h_{n+1}=\bar h+q_{n+1}$ satisfies
\begin{equation}
    (h_{n+1})_{t}-\Delta h_{n+1}=F_n-\operatorname{div}v_{n+1}.
\end{equation}
By the definition of $\mathcal{K}_{T}$, $\nabla q=w$, and by the definition of $\Phi$, the map $q\mapsto q'$ satisfies $\nabla q'=w'$. Hence, by induction, $\nabla q_{n}=w_{n}$. Combining this with $\nabla \bar h=\bar g$, we obtain
\begin{equation}
    \nabla h_{n}=\nabla q_{n}+\nabla \bar h=w_{n}+\bar g=g_{n}
{\implies}
g_{n}=\nabla h_{n}=\nabla\log\rho_{n}.
\end{equation}
Also,
\begin{equation}
    \rho_{n}=e^{\bar h}e^{q_{n}}=e^{h_{0}}e^{\bar h-h_{0}}e^{q_{n}}
\in[e^{-(B_{0}+1)}\rho_{0},e^{B_{0}+1}\rho_{0}],
\end{equation}
since
$$
\|\bar h(t)-h_{0}\|_{L^{\infty}}\le B_{0}\sqrt{t}\le B_{0}
\quad\text{and}\quad
\|q_n\|_{C([0,T];L^\infty)}\le 1.
$$
By the contraction estimate, $(z_{n},w_{n},q_{n})$ is a Cauchy sequence, such that
\begin{equation}
    \|(z_{n+1}-z_n,w_{n+1}-w_n)\|_{X_T}
+\|q_{n+1}-q_n\|_{C([0,T];L^\infty)}
\le C2^{-n}.
\end{equation}
Since $X_{T}$ and $C([0,T];BUC)$ are complete, there exists a limit $(z,w,q)\in X_T\times C([0,T];BUC)$ and 
\begin{equation}
    (z_n,w_n)\to(z,w)\quad\text{in } C([0,T];H^2)\cap L^2(0,T;H^3),
\qquad
q_n\to q\quad\text{in } C([0,T];BUC).
\end{equation}
Define the limits 
\begin{equation}
    v=\bar v+z,\quad g=\bar g+w,\quad h=\bar h+q,\quad \rho=e^h.
\end{equation}
Noting that $$|\rho_{n}-\rho|
=|e^{\bar h+q_{n}}-e^{\bar h+q}|
=|\rho(e^{q_{n}-q}-1)|
=\rho e^{\xi}|q_{n}-q|,
$$ for some $\xi$ between $0$ and $q_{n}-q,$ we have 
$$ \left\lVert \rho_{n}-\rho \right\rVert _{CL^{\infty}}\leq e^{2}\left\lVert \rho \right\rVert_{CL^{\infty}}\left\lVert q_{n}-q \right\rVert _{CL^{\infty}}\leq e^{2+B}\left\lVert \rho_{0} \right\rVert _{L^{\infty}}\left\lVert q_{n}-q \right\rVert _{CL^{\infty}}, $$
and 
$$ \|\rho_n-\rho\|_{CL^2}\le e^{2+B}\|\rho_0\|_{L^2}\left\lVert q_{n}-q \right\rVert _{CL^{\infty}}.$$
Therefore, 
\begin{equation}
    \|\rho_{n}-\rho\|_{C([0,T];L^2\cap L^\infty)}
\le C_B(\|\rho_0\|_{L^2}+\|\rho_0\|_{L^\infty})
\|q_{n}-q\|_{C([0,T];L^\infty)}.
\end{equation}
We will next take the limit of $\eqref{eq: aux-iter}_{1,2}$ to derive the equations for $v,g$. We already know that$$(v_n-v,g_n-g)=(z_n-z,w_n-w)\to0 
\text{ in } X_{T},$$
and they are uniformly bounded in $C([0,T];\mathcal{B})$:
$$
\|v_{n}\|_{C([0,T];\mathcal{B})}
\le \|\bar v\|_{C([0,T];\mathcal{B})}+\|z_{n}\|_{C([0,T];\mathcal{B})}
\le C(\|v_{0}\|_{\mathcal{B}}+\|g_{0}\|_{\mathcal{B}}+1),
$$
$$
\|g_{n}\|_{C([0,T];\mathcal{B})}
\le \|\bar g\|_{C([0,T];\mathcal{B})}+\|w_{n}\|_{C([0,T];\mathcal{B})}
\le C(\|v_{0}\|_{\mathcal{B}}+\|g_{0}\|_{\mathcal{B}}+1).
$$
For the convergence of the right hand side of  $\eqref{eq: aux-iter}_{1,2}$, by Lemma \ref{lem: Product Estimates}, it holds that 
\begin{equation}
\begin{aligned}
&\|(v_n-2g_n)\cdot\nabla v_n
  -(v-2g)\cdot\nabla v\|_{C([0,T];H^1)}
\\
=& \Bigl\|
   \bigl((v_n-v)-2(g_n-g)\bigr)\cdot\nabla v_n
\\
   &+(v-2g)\cdot\nabla(v_n-v)
   \Bigr\|_{C([0,T];H^1)}
\\
\le& C\Bigl(
   \|v_n-v\|_{C([0,T];\mathcal{B})}
   +\|g_n-g\|_{C([0,T];\mathcal{B})}
   \Bigr)
\\
&\times
   \|v_n\|_{C([0,T];\mathcal{B})}
\\
&+C\|v-2g\|_{C([0,T];\mathcal{B})}
   \|v_n-v\|_{C([0,T];\mathcal{B})}
   \longrightarrow 0.
\end{aligned}
\end{equation}
and 
\begin{equation}
    \begin{aligned}
\left\lVert F_{n}-F \right\rVert _{C([0,T];H^{2})}&= \left\lVert |g_n|^2-v_n\cdot g_n-\left| g \right| ^{2}+v\cdot g \right\rVert _{C([0,T];H^{2})}\\
&= \left\lVert (g_n+g)\cdot(g_n-g)-(v_n-v)\cdot g_n-v\cdot(g_n-g) \right\rVert _{C([0,T];H^{2})}\\
&\leq C(\left\lVert g_{n}\right\rVert_{C([0,T];\mathcal{B})} +\left\lVert g\right\rVert_{C([0,T];\mathcal{B})} +\left\lVert v\right\rVert_{C([0,T];\mathcal{B})} )\\
&\quad\times(\left\lVert g_{n}-g \right\rVert_{C([0,T];H^{2})}+\left\lVert v_{n}-v \right\rVert _{C([0,T];H^{2})})\\
&\leq C\bigl(\|v_n-v\|_{C([0,T];H^2)}+\|g_n-g\|_{C([0,T];H^2)}\bigr)\to 0.
\end{aligned}
\end{equation}
As for the pressure term, we obtain from \eqref{est: regular-pressure-CL2}-\eqref{est: regular-pressure-CLinfty}, \eqref{est: pressure L2} and \eqref{est: pressure H1} that $\mathfrak{p}\in C([0,T];H^{1})\cap C([0,T];L^{\infty})$ s.t. 
\begin{equation}
    \|\mathfrak{p}_n-\mathfrak{p}\|_{C([0,T];H^1)}
\le C\bigl(\|g_n-g\|_{C([0,T];H^1)}+\|q_n-q\|_{C([0,T];L^\infty)}\bigr)\to 0.
\end{equation}
Since the right-hand sides of $\eqref{eq: aux-iter}_{1,2}$ converge in $C([0,T];L^2)\cap L^2(0,T;H^1)$, we integrate these equations over $[0,t]$ and then take the limit $n \to \infty$. This yields, for every $t \in [0,T]$,
\begin{equation}
    \begin{aligned}
v(t)-v_0 &= \int_0^t \big(\Delta v - (v-2g)\cdot\nabla v - \mathfrak{p}\big)(s)\,ds, \\
 g(t)-g_0 &= \int_0^t \big(\Delta g - \nabla\mathrm{div}v + \nabla F\big)(s)\,ds,
 \end{aligned}
\quad \text{in } L^2(\mathbb{R}^N).
\end{equation}
The integrands belong to $C([0,T];L^2)\cap L^2(0,T;H^1)$. By the fundamental theorem of Bochner integrals, we have $v-v_0, g-g_0 \in C^1([0,T];L^2) \cap W^{1,2}(0,T;H^1)$, and their time derivatives satisfy
\begin{equation}
    \begin{cases}
 v_t-\Delta v=-(v-2g)\cdot\nabla v-\mathfrak{p},\\
 g_t-\Delta g+\nabla\operatorname{div}v=\nabla F,
 \end{cases}
 \quad \text{in } C([0,T];L^2)\cap L^2(0,T;H^1).
\end{equation}

We then take the limit of $\eqref{eq: aux-iter}_3$ to obtain the equation for $h.$ From the definitions of $\bar h$ and $q_{n+1}$, 
\begin{equation}
    \begin{aligned}
h_{n+1}-h_{0}&= \overline{h}-h_{0}+q_{n+1}\\
&=\int^{t}_{0}\left( \text{div}\left(\overline{g}\right) -\text{div}\left(\overline{v}\right)\right)ds+\int^{t}_{0}\text{div}\left(w_{n+1}\right)-\text{div}\left(z_{n+1}\right)+F_{n}ds \\
&= \int^{t}_{0}(\text{div}\left(g_{n+1}\right)-\text{div}\left(v_{n+1}\right)+F_{n})ds,
\end{aligned}
\end{equation}
Taking the limit in $C([0,T];BUC)$ on both sides, the left-hand side converges because $q_{n+1}\to q$ in $C([0,T];BUC)$, hence $h_{n+1}-h=q_{n+1}-q\to0$ in $C([0,T];BUC)$. For the right-hand side, since 
\begin{equation}
    \begin{aligned}
&\sup_{0\le t\le T}\left\|\int_0^t[\operatorname{div}(g_{n+1}-g)-\operatorname{div}(v_{n+1}-v)+F_n-F]\,ds\right\|_{BUC}\\
\le& C\sqrt{T}(\|\nabla(g_{n+1}-g)\|_{L^2(0,T;BUC)}+\|\nabla(v_{n+1}-v)\|_{L^2(0,T;BUC)})+CT\|F_n-F\|_{C([0,T];BUC)}\\
\le& C\sqrt{T}(\|g_{n+1}-g\|_{L^2(0,T;H^3)}+\|v_{n+1}-v\|_{L^2(0,T;H^3)})+CT\|F_n-F\|_{C([0,T];H^2)}\to0.
\end{aligned}
\end{equation}
We obtain
\begin{equation}
    h-h_0=\int_0^t(\operatorname{div}g-\operatorname{div}v+F)\,ds.
\end{equation}
Next we determine the regularity of the integrand. \eqref{est: regular of vg} implies that $\operatorname{div}v,\operatorname{div}g\in L^2(0,T;H^2).$ From Lemma \ref{lem: Product Estimates}, $v-g=(v_0-g_0)+(v-v_0)-(g-g_0)\in C([0,T];H^2),$ and $g\in C([0,T];\mathcal{B})$ yield $F=-(v-g)\cdot g\in C([0,T];H^2).$ Since the integrand $\operatorname{div}g - \operatorname{div}v + F \in L^2(0,T;H^2)$, the fundamental theorem for Bochner integrals yields $h-h_0 \in W^{1,2}(0,T;H^2) \subset C([0,T];H^2),$ and 
\begin{equation}\label{eq: h_t}
    h_t=\operatorname{div}g-\operatorname{div}v+F \in L^2(0,T;H^2).
\end{equation}

 We now verify that $\Delta h=\operatorname{div}g$ in the distributional sense. As $\mathcal{K}_{T}$ is closed, $\nabla q=\boldsymbol{w}$ in the sense of distributions. Combining this with $\nabla\bar h(t)=\bar g(t)$, we have
$$
\nabla h(t)=\nabla(\bar h+q)(t)=\bar g(t)+w(t)=g(t)\quad\text{in }\mathcal{D}'(\mathbb{R}^N),
$$
so
$$
\langle\Delta h(t),\varphi\rangle=-\langle\nabla h,\nabla\varphi\rangle=-\langle g,\nabla\varphi\rangle=-\int g\cdot\nabla\varphi=\langle\operatorname{div}g(t),\varphi\rangle,
$$
i.e.
$$
\Delta h=\operatorname{div}g,\quad\forall t\in[0,T]\text{ in }\mathcal{D}'(\mathbb{R}^N).
$$
Since $\operatorname{div}g\in L^2(0,T;H^2)$, this shows that the weak derivative $\Delta h=\operatorname{div}g\in L^2(0,T;H^2).$ Thus
\begin{equation}
    h_t-\Delta h=F-\operatorname{div}v \quad\text{in }L^2(0,T;H^2).
\end{equation}
Using the chain rule, we obtain the equation for $\rho$. Let $\rho=e^h$. Since $\rho_0 \in H^3$ is strictly positive, the local composition rule gives $h_0=\log\rho_0 \in H^3_{\mathrm{loc}}$. We have already established $h-h_0\in C([0,T];H^2)$. Moreover, \eqref{est: regular of vg} yields $\nabla(h-h_0)=g-g_0\in C([0,T];H^2).$ Thus we have $h-h_0\in C([0,T];H^3).$ Together with $h_t\in L^2(0,T;H^2)$ from \eqref{eq: h_t}, this provides the regularity needed for the following chain-rule calculation
$$
\rho_t=\rho h_t,\qquad \nabla\rho=\rho g,\qquad \Delta\rho=\rho(\operatorname{div}g+|g|^2).
$$
Therefore,
$$
\begin{aligned}
\rho_t-\Delta\rho
&=\rho(h_t-\operatorname{div}g-|g|^2)\\
&=\rho(F-\operatorname{div}v-|g|^2)\\
&=-\rho(v\cdot g+\operatorname{div}v)\\
&=-v\cdot\nabla\rho-\rho\operatorname{div}v\\
&=-\operatorname{div}(\rho v).
\end{aligned}
$$
In Step 3 we deduce $\rho\in C([0,T];L^2\cap L^\infty),$ and since $g\in C([0,T];L^\infty)$, we have $\nabla\rho=\rho g\in C([0,T];L^2)\cap C([0,T];L^\infty)$. Combining $h_t$, $\operatorname{div}g$, $\operatorname{div}v\in L^2(0,T;H^2)$ and $v,g\in C([0,T];L^\infty)$, we get
$$
\begin{aligned}
\|\rho_t\|_{L^2(0,T;L^2)} &\le \|\rho\|_{C([0,T];L^\infty)}\|h_t\|_{L^2(0,T;L^2)},\\
\|\Delta\rho\|_{L^2(0,T;L^2)} &\le \|\rho\|_{C([0,T];L^\infty)}\|\operatorname{div}g\|_{L^2(0,T;L^2)} +\sqrt{T}\|\rho\|_{C([0,T];L^2)}\|g\|_{C([0,T];L^\infty)}^2,\\
\|\operatorname{div}(\rho v)\|_{L^2(0,T;L^2)} &\le \sqrt{T}\|v\|_{C([0,T];L^\infty)}\|\nabla\rho\|_{C([0,T];L^2)} +\|\rho\|_{C([0,T];L^\infty)}\|\operatorname{div}v\|_{L^2(0,T;L^2)}.
\end{aligned}
$$
Hence
\begin{equation}\label{eq: rho in auxi form}
    \rho_t-\Delta\rho=-\operatorname{div}(\rho v) \quad\text{in }L^2(0,T;L^2).
\end{equation}
We proceed to estimate higher-order regularity of the density. By Lemma \ref{lem: sol of linear eq}, from $v\in C([0,T];L^\infty)\subset L^2(0,T;L^\infty),$ $
\nabla v\in L^2(0,T;H^2),$ $
\rho_0\in H^3,$
it follows that the density equation has a solution
$$
R\in C([0,T];H^3)\cap L^2(0,T;H^4).
$$
We then prove that this is indeed the limit $\rho$ obtained by iteration. Let
$$
d=R-\rho\in C([0,T];L^2)\cap L^2(0,T;H^1),
$$
which satisfies
$$
\begin{cases}
d_t-\Delta d=-\operatorname{div}(dv),\\
d(0)=0.
\end{cases}
$$
Testing with $d$ gives
$$
\frac12\frac{d}{dt}\|d\|_{L^2}^2+\|\nabla d\|_{L^2}^2
=\int vd\cdot\nabla d
\le \frac12\|\nabla d\|_{L^2}^2+C\|v\|_{L^\infty}^2\|d\|_{L^2}^2.
$$
By Gronwall's inequality, $d=0$, hence
$$
\rho=R\in C([0,T];H^3)\cap L^2(0,T;H^4),
$$
such that
\begin{equation}\label{est: ho regular of rho}
   \begin{aligned}
        \sup_{0\le t\le T}\|\rho(t)\|_{H^3}^2
+\int_0^T\|\nabla\rho\|_{H^3}^2\,dt
&\le C\|\rho_0\|_{H^3}^2
\exp\left(C\int_0^T\left(1+\|v\|_{L^\infty}^2+\|\nabla v\|_{H^2}^2\right)dt\right)\le C,
   \end{aligned}
\end{equation}
where we use $\|\rho_0\|_{H^3}\le A_0$ and \eqref{est: regular of vg}.

\textbf{Step 5: Return from the auxiliary variables to $(\rho,u)$}

\textbf{Step 5.1: Time regularity of the auxiliary variables}

We want to show that
\begin{equation}\label{regular of v_t g_t rho_t}
    v_t,g_t\in C([0,T];L^2)\cap L^2(0,T;H^1),\quad
\rho_t\in C([0,T];H^1)\cap L^2(0,T;H^2).
\end{equation}
It follows from \eqref{est: regular of vg} that $v\in C([0,T];\mathcal{B})$ and $\nabla v\in L^2(0,T;H^2).$ Thus we have $\Delta v\in C([0,T];L^2)\cap L^2([0,T];H^1).$ By Lemma \ref{lem: Product Estimates}, $v,g\in C([0,T];\mathcal{B})$ yields $(v-2g)\cdot\nabla v\in C([0,T];H^1),$ and $\mathfrak{p}\in C([0,T];H^1)$ by \eqref{est: regular-pressure-CL2} and \eqref{est: regular-pressure-CH1}.Hence
$$
v_t=\Delta v
-(v-2g)\cdot\nabla v
-\mathfrak{p}
\in C([0,T];L^2)\cap L^2(0,T;H^1).
$$
In the same way as for $\Delta v,$ it holds that $\Delta g,\nabla\operatorname{div}v\in C([0,T];L^2)\cap L^2(0,T;H^1).$ By Lemma \ref{lem: Product Estimates}, $$F=|g|^2-v\cdot g\implies \nabla F\in C([0,T];H^1)\subset C([0,T];L^2)\cap L^2(0,T;H^1).$$Therefore,
$$g_t=\Delta g-\nabla\operatorname{div}v+\nabla F\in C([0,T];L^2)\cap L^2(0,T;H^1).$$ Recall from \eqref{est: ho regular of rho} that $\Delta\rho\in C([0,T];H^1)\cap L^2(0,T;H^2).$ Lemma \ref{lem: Product Estimates} shows that $\rho v\in C([0,T];H^2),$ and $$\|\rho v\|_{L^2(0,T;H^3)}
\le C\|\rho\|_{C([0,T];H^3)}
\bigl(\sqrt{T}\|v\|_{C([0,T];L^\infty)}+\|\nabla v\|_{L^2(0,T;H^2)}\bigr)<\infty.$$Hence $$\rho_t=\Delta\rho-\operatorname{div}(\rho v)\in C([0,T];H^1)\cap L^2(0,T;H^2).$$

\textbf{Step 5.2: Recovery of the regularity of $u,u_t$ and verification of the equations}

\eqref{ass: initial data  of local sol}, \eqref{est: reference field}, \eqref{def: perturbed field} give 
\begin{equation}
    \begin{aligned}
u&= v-g=(\bar v+z)-(\bar g+w)=(\bar v-\bar g)+(z-w)\\
&=(\bar v-v_0)-(\bar g-g_0)
+(v_0-g_0)
+(z-w)
\in C([0,T];H^2).
\end{aligned}
\end{equation}
Combining \eqref{est: regular of vg}, we derive $u\in L^2(0,T;H^3).$ By virtue of \eqref{regular of v_t g_t rho_t}, we get $u_t=v_t-g_t\in C([0,T];L^2)\cap L^2(0,T;H^1).$

We then recover the original equations. Substituting $\nabla\rho=\rho g,$ $u=v-g$ into \eqref{eq: rho in auxi form}, we recover
\begin{equation}\label{eq: rho}
\rho_t+\operatorname{div}(\rho u)=0.
\end{equation}
Multiplying $\eqref{rhovghF}_2$ by $\rho$, substituting $\nabla\rho=\rho g$, $\mathfrak{p}=\gamma\rho^{\gamma-1}g$, we get
\begin{equation}\label{eq: v tmp}
    \rho(v_t+(u\cdot\nabla)v)+\nabla(\rho^\gamma)=\operatorname{div}(\rho\nabla v).
\end{equation}
We also need the equation for $g$. From \eqref{eq: h_t}, 
$$\begin{aligned}
h_t&=\operatorname{div}g-\operatorname{div}v+|g|^2-v\cdot g\\
&=-\operatorname{div}u-u\cdot g.
\end{aligned}$$
 Taking the gradient and multiplying by $\rho$, we obtain$$\rho g_t+\rho\nabla(u\cdot g)+\rho\nabla\operatorname{div}u=0.$$
It is noted that $$g=\nabla h\implies \nabla g=(\partial_j g_i)_{ij}=(\partial_j\partial_i h)_{ij}=(\partial_i\partial_j h)_{ij}=(\nabla g)^\top,$$and $\rho g_j=\rho(\nabla\log\rho)_j=\partial_j\rho,$ which gives
\begin{equation}
\begin{aligned}
\rho\nabla(u\cdot g)+\rho\nabla\operatorname{div}u
&=\left[ \rho\cdot\partial_i(u_j g_j)+\rho\partial_i(\operatorname{div}u) \right] _{i}\\
&=\left[ \partial_i u_j{\rho g_j}_{}
+\rho u_j\partial_i g_j+\rho\partial_i\operatorname{div}u \right] _{i}\\
&=\left[ \rho u_j\partial_j g_i+\rho\partial_i\operatorname{div}u+\partial_i u_j\partial_j\rho \right] _{i}\\
&=\rho\nabla g\cdot u+\operatorname{div}\bigl(\rho(\nabla u)^\top\bigr).
\end{aligned}
\end{equation}
Hence, we arrive at 
\begin{equation}\label{eq: g tmp}
    \rho(g_t+(u\cdot\nabla)g)=-\operatorname{div}\bigl(\rho(\nabla u)^\top\bigr).
\end{equation}
Subtracting \eqref{eq: g tmp} from \eqref{eq: v tmp} , and using $u=v-g$, we get
$$
\begin{aligned} \rho(u_t+u\cdot\nabla u)+\nabla(\rho^\gamma) &=\operatorname{div}(\rho \nabla v) +\operatorname{div}(\rho(\nabla u)^{\top})\\ &=\operatorname{div}(2\rho\mathcal{D} u) +\operatorname{div}(\rho \nabla g)\\ &=\operatorname{div}(2\rho\mathcal{D} u) +\operatorname{div}(\rho\nabla^2\log\rho). \end{aligned}$$
Using \eqref{eq: rho}, we deduce 
\begin{equation}
    (\rho u)_t+\mathrm{div}(\rho u\otimes u)+\nabla(\rho^\gamma)=\mathrm{div}(2\rho\mathcal{D}u)+\mathrm{div}(\rho\nabla^2\log\rho).
\end{equation}
Next we verify the far-field property. 
Since $H^2,H^3\hookrightarrow C_0$ in $\mathbb{R}^N$ ($N=2,3$), we have
$$
\rho\in C([0,T];H^3),\quad u\in C([0,T];H^2)
\implies
\lim_{|x|\to\infty}\rho(t,x)=0,\quad
\lim_{|x|\to\infty}u(t,x)=0.
$$
The uniformity in $t\in[0,T]$ follows because $t\mapsto\rho(t),u(t)$ are continuous maps from $[0,T]$ into $H^3,H^2$, respectively. And $H^2,H^3\hookrightarrow C_0$ are continuous embeddings, so the images $\{\rho(t):t\in[0,T]\}$, $\{u(t):t\in[0,T]\}\subset C_0$ are compact, which implies that they can be covered by finitely many balls in the uniform norm.

\textbf{Step 6: Uniqueness of local solutions}

Let $(\rho_1,u_1)$ and $(\rho_2,u_2)$ be any two solutions satisfying regularity \eqref{regular of local sol},  and define the auxiliary variables $$h_i=\log\rho_i,\quad g_i=\nabla h_i,\quad v_i=u_i+g_i,\quad
F_i=|g_i|^2-v_i\cdot g_i,\quad
\mathfrak{p}_i=\gamma\rho_i^{\gamma-1}g_i\quad(i=1,2).$$
Take $I=[t_0,t_0+\tau]$, and suppose the two solutions coincide at $t_0$. Let $X_I$ denote the time-translated version of $X_\tau$, with zero initial trace at $t_0$ and the same norm taken over $I$. Then 
\begin{equation}\label{eq: difference of sol}
    \begin{cases}
(v_1-v_2)_t-\Delta(v_1-v_2)
=-\bigl[(v_1-2g_1)\cdot\nabla v_1-(v_2-2g_2)\cdot\nabla v_2\bigr]
-(\mathfrak{p}_1-\mathfrak{p}_2),\\
(g_1-g_2)_t-\Delta(g_1-g_2)+\nabla\operatorname{div}(v_1-v_2)=\nabla(F_1-F_2),\\
(v_1-v_2)(t_0)=0,\quad (g_1-g_2)(t_0)=0,
\end{cases}
\end{equation}
and
$$
(v_i-v_0,g_i-g_0)\in C([0,T];H^2),\quad
\nabla v_i,\nabla g_i\in L^2(0,T;H^2)
\implies
(v_1-v_2,g_1-g_2)\in L^2(0,T;H^3).
$$
Next we estimate the $L^2(I;H^1)$ norm of the right hand side of \eqref{eq: difference of sol}. Since $\left\lVert (v_i,g_i) \right\rVert_{C([0,T];\mathcal{B})}$ are uniformly bounded, together with Lemma \ref{lem: Product Estimates} and the estimate$$\|f\|_{\mathcal{B}}=\|f\|_{L^\infty}+\|\nabla f\|_{H^1}\le C\|f\|_{H^2}\text{ in }\mathbb{R}^N,N=2,3,$$
we have 
\begin{equation}\label{est: diff eq RHS 1}
    \begin{aligned} &\left\lVert (v_1-2g_1)\cdot\nabla v_1-(v_2-2g_2)\cdot\nabla v_2 \right\rVert _{C(I;H^{1})}\\ =&\left\lVert \left[ (v_1-v_2)-2(g_1-g_2) \right] \cdot\nabla v_1 +(v_2-2g_2)\cdot\nabla(v_1-v_2) \right\rVert_{C(I;H^{1})} \\
\leq&C\left( \left\lVert v_{1}-v_{2} \right\rVert_{C(I;\mathcal{B})}  +\left\lVert g_{1}-g_{2} \right\rVert _{C(I;\mathcal{B})} \right) \left\lVert v_{1} \right\rVert _{C(I;\mathcal{B})}+C\left( \left\lVert v_{2} \right\rVert _{C(I;\mathcal{B})}+\left\lVert g_{2} \right\rVert _{C(I;\mathcal{B})} \right) \left\lVert v_{1}-v_{2} \right\rVert _{C(I;\mathcal{B})}\\
\le& C\|(v_1-v_2,g_1-g_2)\|_{X_I}. 
 \end{aligned}
\end{equation}
Applying Lemma \ref{lem: Product Estimates}, 
\begin{equation}\label{est: diff eq RHS 2}
    \|\nabla(F_1-F_2)\|_{C(I;H^1)}+\|F_1-F_2\|_{C(I;L^\infty)}
\le C\|(v_1-v_2,g_1-g_2)\|_{X_I}.
\end{equation}
For the pressure term, \eqref{est: pressure L2} and \eqref{est: pressure H1} give 
\begin{equation}\label{est: diff eq RHS 3}
    \begin{aligned}
\|\mathfrak{p}_1-\mathfrak{p}_2\|_{C(I;H^1)}
&\le C(\|g_1-g_2\|_{C(I;H^1)}+\|h_1-h_2\|_{C(I;L^\infty)})\\
&\le C(\|(v_1-v_2,g_1-g_2)\|_{X_I}+\|h_1-h_2\|_{C(I;L^\infty)}).
\end{aligned}
\end{equation}
Since
$$
(\rho_i)_t+\operatorname{div}(\rho_i u_i)=0
\implies
(h_i)_t=-\operatorname{div}u_i-u_i\cdot g_i
=\operatorname{div}g_i-\operatorname{div}v_i+F_i,
$$
and $\nabla v_i,$ $\nabla g_i\in L^2(0,T;H^2)$, $F_i=-u_i\cdot g_i\in C([0,T];H^2)$, we have
$$
(h_i)_t=\operatorname{div}g_i-\operatorname{div}v_i+F_i\in L^2(0,T;H^2)\subset L^1(0,T;H^2).
$$
Consequently, $$h_i(t)-h_i(t_0)=\int_{t_0}^t(\operatorname{div}g_i-\operatorname{div}v_i+F_i)\,ds.$$
As $\rho_1(t_0)=\rho_2(t_0)>0$, we have $h_1(t_0)=h_2(t_0)$, and hence
\begin{equation}\label{est: diff eq RHS 4}
    \begin{aligned}
\|h_1(t)-h_2(t)\|_{C(I;L^\infty)}
&=\left\|\int_{t_0}^t[\operatorname{div}(g_1-g_2)-\operatorname{div}(v_1-v_2)+F_1-F_2]\,ds\right\|_{C(I;L^\infty)}\\
&\le C\left(\int_{t_0}^{t_0+\tau}\|\nabla(g_1-g_2,v_1-v_2)\|_{L^\infty}+\|F_1-F_2\|_{L^\infty}\,ds\right)\\
&\le C\left(\int_{t_0}^{t_0+\tau}\|(g_1-g_2,v_1-v_2)\|_{H^3}+\|F_1-F_2\|_{L^\infty}\,ds\right)\\
&\le C\sqrt{\tau}\|(g_1-g_2,v_1-v_2)\|_{L^2(I;H^3)}+C\tau\|F_1-F_2\|_{C(I;L^\infty)}\\
&\le C(\sqrt{\tau}+\tau)\|(v_1-v_2,g_1-g_2)\|_{X_I}.
\end{aligned}
\end{equation}
In view of Lemma \ref{lem: sol of linear eq} and \eqref{est: diff eq RHS 1}-\eqref{est: diff eq RHS 4}, we obtain 
\begin{equation}
    \begin{aligned}
\|(v_1-v_2,g_1-g_2)\|_{X_I}
&\le C(\|(v_1-2g_1)\cdot\nabla v_1-(v_2-2g_2)\cdot\nabla v_2\|_{L^2(I;H^1)}
+\|\mathfrak{p}_1-\mathfrak{p}_2\|_{L^2(I;H^1)}\\
&\qquad+\|\nabla(F_1-F_2)\|_{L^2(I;H^1)})\\
&\le C\sqrt{\tau}(\|(v_1-2g_1)\cdot\nabla v_1-(v_2-2g_2)\cdot\nabla v_2\|_{C(I;H^1)}
+\|\mathfrak{p}_1-\mathfrak{p}_2\|_{C(I;H^1)}\\
&\qquad+\|\nabla(F_1-F_2)\|_{C(I;H^1)})\\
&\le C\sqrt{\tau}(1+\sqrt{\tau}+\tau)\|(v_1-v_2,g_1-g_2)\|_{X_I}.
\end{aligned}
\end{equation}
Choose $\tau\ll1$ such that $C\sqrt{\tau}(1+\sqrt{\tau}+\tau)<1$, then $\|(v_1-v_2,g_1-g_2)\|_{X_I}=0$, hence
$$
v_1=v_2,\quad g_1=g_2,\quad h_1=h_2
\implies \rho_1=\rho_2,\quad u_1=u_2 \text{ on }[t_0,t_0+\tau].
$$
Taking $t_0+\tau$ as the new starting point and repeating the above argument, since the constants in the estimates depend only on the finite norms of the two solutions on $[0,T]$, the same short time length $\tau$ can be chosen, and finitely many steps suffice to cover $[0,T]$.
\end{proof}
\subsection{Proof of Lemma \ref{lem:local-isothermal} }
\label{subsec:local-isothermal}

We prove Lemma~\ref{lem:local-isothermal} by modifying the reference
fields in the preceding construction. The pressure acceleration
is $g$, whose initial value need not belong to $L^2$. We therefore
use the decomposition $-g=-g_{0}-(g-g_{0})$ and solve
\begin{equation}\label{eq:iso-local-reference}
\begin{cases}
\bar v_t-\Delta\bar v=-g_{0},\\
\bar g_t-\Delta\bar g+\nabla\operatorname{div}\bar v=0,\\
(\bar v,\bar g)(0)=(v_{0},g_{0}).
\end{cases}
\end{equation}
After differentiation, the additional source is
$-\partial_k g_{0}\in H^1$. The derivative construction in Step~2
of the preceding proof, with this source included, gives
\[
\nabla\bar v,\nabla\bar g
\in C([0,T];H^1)\cap L^2(0,T;H^2)
\]
with norm at most $C_N \tilde{A}_{0}$ for $T\le1$. Moreover,
\begin{equation}\label{eq:iso-local-reference-differences}
\begin{aligned}
\bar v-v_{0}+t g_{0}&=\int_0^t\Delta\bar v(s)\,ds,\\
\bar g-g_{0}&=\int_0^t
 (\Delta\bar g-\nabla\operatorname{div}\bar v)(s)\,ds.
\end{aligned}
\end{equation}
Arguing as in the derivation of \eqref{est: reference field}, with \(\bar v-v_0\) replaced by \(\bar v-v_{0}+tg_{0}\), and using \(T\le1\), we obtain
\begin{equation}\label{eq:iso-local-reference-bound}
\begin{aligned}
&\|\bar v\|_{C([0,T];\mathcal B)}
 +\|\bar g\|_{C([0,T];\mathcal B)}
 +\|\nabla\bar v\|_{L^2(0,T;H^2)}
 +\|\nabla\bar g\|_{L^2(0,T;H^2)}\\
 +&\|\bar v-v_{0}+t g_{0}\|_{C([0,T];H^2)}
 +\|\bar g-g_{0}\|_{C([0,T];H^2)}
 \le C_N \tilde{A}_{0}.
\end{aligned}
\end{equation}
The additional source requires only \(\nabla g_{0}\in H^1\) and no \(L^2\) assumption on \(g_{0}\) is needed.

On the closed unit ball of $X_T$, write $v=\bar v+z$ and
$g=\bar g+w$, and define $(z',w')$ by
\begin{equation}\label{eq:iso-local-correction}
\begin{cases}
z'_t-\Delta z'=-(v-2g)\cdot\nabla v-(g-g_{0}),\\
w'_t-\Delta w'+\nabla\operatorname{div}z'
 =\nabla F(v,g),\qquad F(v,g)=|g|^2-v\cdot g,\\
(z',w')(0)=0.
\end{cases}
\end{equation}
Here $g-g_{0}=(\bar g-g_{0})+w\in C([0,T];H^2)$.
For two inputs, the difference of this pressure source is simply
$w_1-w_2$. Lemmas~\ref{lem: Product Estimates} and
\ref{lem: sol of linear eq} therefore give
\begin{align*}
\|(z',w')\|_{X_T}&\le C(\tilde{A}_{0})\sqrt{T},\\
\|(z'_1-z'_2,w'_1-w'_2)\|_{X_{T}}
 &\le C(\tilde{A}_{0})\sqrt{T}\,
       \|(z_1-z_2,w_1-w_2)\|_{X_{T}}.
\end{align*}
For sufficiently small ${T}={T}(N,\tilde{A}_{0})$, this map sends the closed unit ball of \(X_{T}\) into itself and is a contraction. Hence it has a unique fixed point. Let $(z,w)\in X_{T}$ be the fixed point. By the reference-field estimate \eqref{eq:iso-local-reference-bound} and the definition of $X_{T}$, $v=\bar v+z$ and $g=\bar g+w$ satisfy $$v,g\in C([0,{T}];\mathcal B),\qquad \nabla v,\nabla g \in C([0,{T}];H^1)\cap L^2(0,{T};H^2).$$The scalar variable $q$ is unnecessary here,
since the effective equations for $v,g$ are independent of $\rho$.

At the fixed point, the curl of $g$ solves the heat equation
with zero initial data, and hence $\operatorname{curl}g=0$.
Set $h_{0}:=\log\rho_{0}$ and define
\begin{equation}\label{eq:iso-local-log-recovery}
h=h_{0}+\int_0^t
\bigl(\operatorname{div}g-\operatorname{div}v+F(v,g)\bigr)(s)\,ds.
\end{equation}
The divergence terms belong to $L^2(0,{T};H^2)$ and
$F(v,g)\in C([0,{T}];BUC)$, so $h-h_{0}\in C([0,{T}];BUC)$ and is bounded by a constant depending only
on $N,\tilde{A}_{0}$. Taking the gradient in
\eqref{eq:iso-local-log-recovery}, using
$\nabla\operatorname{div}g=\Delta g$ and the equation for $g$,
gives $\nabla h=g$. Let \(\rho=e^h=\rho_{0}e^{h-h_{0}}\). Since \(h-h_{0}\in C([0,{T}];BUC)\) is bounded, and $\rho_{0}\in H^{3}\hookrightarrow L^2\cap BUC,$ we have $\rho\in C([0,{T}];L^2\cap L^\infty),$ and $0<e^{-C}\rho_{0}\le\rho\le e^C\rho_{0},$ where \(C\) depends only on \(N,\tilde{A}_{0}\). Moreover, \[
h_t=\operatorname{div}g-\operatorname{div}v+F
\in L^2(0,{T};BUC).
\]Since \(h_{0}\in H^3_{\mathrm{loc}}\) and \(\nabla h=g\), the Sobolev chain rule applies locally in space, as in $\gamma>1$ case, and yields
\[
\nabla\rho=\rho g,\qquad
\rho_t=\rho h_t,\qquad
\Delta\rho=\rho\bigl(\operatorname{div}g+|g|^2\bigr).
\]These identities give \(\nabla\rho\in C([0,{T}];L^2)\) and \(\Delta\rho\in L^2(0,{T};L^2)\). For the time derivative, we use $\|\rho_t\|_{L^2(0,{T};L^2)}
\le
\|\rho\|_{C([0,{T}];L^2)}
\|h_t\|_{L^2(0,{T};L^\infty)}
<\infty.$ The same calculation as in \eqref{eq: rho in auxi form}, gives
\[
\rho_t-\Delta\rho=-\operatorname{div}(\rho v)
\quad\text{in }L^2(0,{T};L^2).
\]Applying Lemma \ref{lem: sol of linear eq}(ii) and repeating the uniqueness argument following  \eqref{eq: rho in auxi form} we obtain
\[
\rho\in C([0,{T}];H^3)\cap L^2(0,{T};H^4),
\]with the same estimate as \eqref{est: ho regular of rho}. The argument for \(\rho_t\) in \eqref{regular of v_t g_t rho_t} then gives $\rho_t\in C([0,{T}];H^1)\cap L^2(0,{T};H^2).$ This argument does not require $F\in L^2.$

By the reference-field estimate \eqref{eq:iso-local-reference-bound} and $(z,w)\in X_{T},$ we have $\bar v-v_{0}+tg_{0},$ $\bar g-g_{0},$ $z, w\in C([0,{T}];H^2).$ Consequently, 
$$
\begin{aligned} v-v_{0}+tg_{0}&=(\bar v-v_{0}+tg_{0})+z \in C([0,{T}];H^2),\\ g-g_{0}&=(\bar g-g_{0})+w \in C([0,{T}];H^2). \end{aligned}
$$
For the time regularity, arguing as in the proof of \eqref{regular of v_t g_t rho_t}, but using $g-g_{0}\in C([0,{T}];H^2)$ in place of the pressure term, the equations yield
$$
\begin{aligned} v_t+g_{0}&=\Delta v-(v-2g)\cdot\nabla v-(g-g_{0}) \in C([0,{T}];L^2)\cap L^2(0,{T};H^1),\\ g_t&=\Delta g-\nabla\operatorname{div}v+\nabla F \in C([0,{T}];L^2)\cap L^2(0,{T};H^1). \end{aligned}
$$
Set $u=v-g$ and $u_{0}=v_{0}-g_{0}$. Then we have $u-u_{0}+tg_{0} =(v-v_{0}+tg_{0})-(g-g_{0}) \in C([0,{T}];H^2)$ and $u_t+g_{0}=(v_t+g_{0})-g_t \in C([0,{T}];L^2)\cap L^2(0,{T};H^1).$

The verification of the original equations is exactly
the calculation in Step~5.2 above, with $\rho^\gamma=\rho$.
If $v_{0},g_{0}\in C_0$, the identities $ v=v_{0}-tg_{0}+(v-v_{0}+tg_{0})$ and $g=g_{0}+(g-g_{0}),$ together with $H^2\hookrightarrow C_0$ give uniform far-field limits $v,g,u\in C([0,{T}];C_0).$ Since $\rho\in C([0,{T}];H^3)$, we also have  
$\rho\in C([0,{T}];C_0)$. The compactness argument gives$$
\lim_{R\to\infty} \sup_{\substack{0\le t\le{T}\\ |x|\ge R}} \bigl(|\rho(t,x)|+|u(t,x)|+|v(t,x)|+|g(t,x)|\bigr)=0.$$

For uniqueness, let two solutions with the same data coincide at $t_0$, and set $I=[t_0,t_0+\tau]\subset[0,T]$. Since $v_i-v_0+tg_0,$ $g_i-g_0\in CH^2,$ $\nabla v_i,\nabla g_i\in L^2H^2,$ their differences belong to $X_I$, with zero initial values at $t_0$. The difference equations are \eqref{eq: difference of sol}, with the pressure difference replaced by $g_1-g_2$. Estimates \eqref{est: diff eq RHS 1}, \eqref{est: diff eq RHS 2} remain valid, and $\|g_1-g_2\|_{C(I;H^1)} \le C\|(v_1-v_2,g_1-g_2)\|_{X_I}.$ Hence Lemma \ref{lem: sol of linear eq} give
$$\|(v_1-v_2,g_1-g_2)\|_{X_I} \le C\sqrt\tau\, \|(v_1-v_2,g_1-g_2)\|_{X_I}.$$
Choosing $C\sqrt{\tau}<1$ gives $v_1=v_2$ and $g_1=g_2$ on $I$. \eqref{eq:iso-local-log-recovery} then gives equality of $\rho,u$. A uniform choice of $\tau$ and finitely many intervals prove uniqueness on $[0,T]$.


	\section*{Acknowledgments}
	X. Huang is partially supported by Chinese Academy of Sciences Project for Young Scientists in Basic Research (Grant No. YSBR-031), National Natural Science Foundation of China (Grant Nos. 12494542, 11688101) and National Key R\&D Program of China (Grant No. 2021YFA1000801). 
	\vspace{1cm}
    
	\noindent\textbf{Data availability statement.} Data sharing is not applicable to this article.
	
	\vspace{0.3cm}
	\noindent\textbf{Conflict of interest.} The authors declare that they have no conflict of interest.


\begin{thebibliography}{99}
\bibitem{Antonelli-Bresch-Spirito}
P. Antonelli, D. Bresch, S. Spirito, Global weak solutions of the Navier--Stokes--Korteweg equations in one dimension, arXiv:2502.17147 (2025).

\bibitem{Bresch-Vasseur-Yu}
D. Bresch, A. F. Vasseur, C. Yu, Global existence of entropy-weak solutions to the compressible
Navier--Stokes equations with nonlinear density dependent viscosities, J. Eur. Math. Soc. \textbf{24} (2022), no.~5, 1791--1837.

\bibitem{Charve-Haspot}
F. Charve, B. Haspot, Existence of a global strong solution and vanishing capillarity-viscosity
limit in one dimension for the Korteweg system, SIAM J. Math. Anal. \textbf{45} (2013), no.~2, 469--494.

\bibitem{Chen-Zhao}
Z.~Chen and H.~Zhao,
Global existence and vanishing dispersion limit of
strong/classical solutions to the one-dimensional compressible
quantum Navier--Stokes equations with large initial data,
J. Math. Fluid Mech. \textbf{27} (2025), no.~4,
Paper No.~69.

\bibitem{Chen-Zhang-Zhu}
G.-Q. Chen, J. Zhang, S. Zhu, Global regular solutions of the degenerate
compressible Navier--Stokes equations with large initial data of spherical symmetry,
arXiv:2512.18545 (2025).

\bibitem{Dunn-Serrin}
J. E. Dunn, J. Serrin, On the thermomechanics of interstitial working, Arch. Rational
Mech. Anal. \textbf{88} (1985), no.~2, 95--133.

\bibitem{Danchin-Desjardins}
R. Danchin, B. Desjardins, Existence of solutions for compressible fluid models of Korteweg type, Ann. Inst. H. Poincar\'e C Anal. Non Lin\'eaire \textbf{18} (2001), no.~1, 97--133.

\bibitem{Gu-Huang}
Y. Gu, X. Huang, Global spherically symmetric classical solutions for arbitrary large initial data of the multi-dimensional non-isentropic compressible Navier--Stokes equations,
arXiv:2512.24799 (2025).

\bibitem{Gu-Huang-Meng-Zhou}
Y. Gu, X. Huang, W. Meng, H. Zhou, Global-in-time strong solutions for the 2D and
3D generalized compressible Navier--Stokes--Korteweg system with arbitrarily large initial
data, arXiv:2603.11762 (2026).

\bibitem{Germain-LeFloch}
P. Germain, P. LeFloch, Finite energy method for compressible fluids: the Navier--Stokes--Korteweg
model, Comm. Pure Appl. Math. \textbf{69} (2016), no.~1, 3--61.

\bibitem{Haspot-large-energy}
B.~Haspot,
Existence of global strong solution for Korteweg system
with large infinite energy initial data,
J. Math. Anal. Appl. \textbf{438} (2016), no.~1, 395--443.

\bibitem{Haspot-quantum}
B.~Haspot,
Global strong solution for the Korteweg system with
quantum pressure in dimension $N\geq2$,
Math. Ann. \textbf{367} (2017), no.~1--2, 667--700.

\bibitem{Hattori-Li-1}
H. Hattori, D. Li, Solutions for two-dimensional system for materials of Korteweg type, SIAM J. Math. Anal. \textbf{25} (1994), no.~1, 85--98.

\bibitem{Hattori-Li-2}
H. Hattori, D. Li, Global solutions of a high-dimensional system for Korteweg materials, J. Math. Anal. Appl. \textbf{198} (1996), no.~1, 84--97.

\bibitem{Huang-Li}
X. Huang, B. Li, Sharp sign criterion for global existence and blow-up of strong
solutions with arbitrarily large initial data in the 3D spherically
symmetric compressible Navier--Stokes equations, arXiv:2609.11304 (2026).

\bibitem{Huang-Meng-Zhang-111}
X. Huang, W. Meng, X. Zhang, On global classical and weak solutions with arbitrary large initial data to the multi-dimensional viscous Saint--Venant system and compressible
Navier--Stokes equations subject to the BD entropy condition under spherical symmetry,
arXiv:2512.15029 (2025).

\bibitem{Huang-Meng-Zhang}
X. Huang, W. Meng, X. Zhang, Global regularity of the multi-dimensional compressible
Navier--Stokes--Korteweg system with large initial data, arXiv:2602.00455 (2026).

\bibitem{Huang-Gu-Lei}
X. Huang, Y. Gu, M. Lei, Global strong solutions with large initial data for the
Cauchy problem of the multi-dimensional compressible Navier--Stokes--Korteweg system,
arXiv:2602.10700 (2026).

\bibitem{Huang-Lei-Zhou}
X. Huang, M. Lei, H. Zhou, On the Cauchy problem for the multi-dimensional compressible
Navier--Stokes--Korteweg system: Global strong solutions with arbitrarily large initial data,
arXiv:2604.23652 (2026).

\bibitem{Jungel-effective}
A.~J\"ungel,
Effective velocity in compressible Navier--Stokes
equations with third-order derivatives,
Nonlinear Anal. \textbf{74} (2011), no.~8, 2813--2818.

\bibitem{Korteweg}
D. J. Korteweg, Sur la forme que prennent les \'equations du mouvement des fluides
si l'on tient compte des forces capillaires caus\'ees par des variations de densit\'e
consid\'erables mais continues et sur la th\'eorie de la capillarit\'e dans l'hypoth\`ese
d'une variation continue de la densit\'e,
Arch. N\'eer. Sci. Exactes S\'er. II \textbf{6} (1901), 1--24.

\bibitem{Kotschote}
M.~Kotschote, Strong solutions for a compressible fluid model
of Korteweg type,
Ann. Inst. H. Poincar\'e Anal. Non Lin\'eaire
\textbf{25} (2008), no.~4, 679--696.

\bibitem{Li-Xin}
J. Li, Z. P. Xin, Global existence of weak solutions to the barotropic compressible Navier--Stokes
flows with degenerate viscosities, arXiv:1504.06826 (2015).

\bibitem{Vasseur-Yu}
A. F. Vasseur, C. Yu, Existence of global weak solutions for 3D degenerate compressible Navier--Stokes equations, Invent. Math. \textbf{206} (2016), no.~3, 935--974.
\end{thebibliography}
\end{document}